\documentclass[11pt, oneside]{article}   	
\usepackage{geometry}                		
\usepackage{graphicx}				

\graphicspath{{graphs/}}
\usepackage{subfigure}
\usepackage{dsfont}					
\usepackage{amssymb}
\usepackage{amsmath}
\usepackage{amsthm}
\usepackage{xcolor}
\usepackage[normalem]{ulem}
\usepackage{hyperref}
\usepackage[textsize=tiny]{todonotes}
\usepackage{mathtools}
\usepackage[ruled]{algorithm2e}
\usepackage{enumitem}
\usepackage{caption}

\newtheorem{theorem}{Theorem}[section]

\newtheorem{lemma}[theorem]{Lemma}
\newtheorem{proposition}{Proposition}[section]
\newtheorem{assumption}{Assumption}[section]

\newtheorem{definition}{Definition}[section]

\newcommand{\dd}{\mathrm{d}}

\newcommand{\diag}{\mathrm{diag}}

\newcommand{\Let}{\triangleq}
\newcommand{\eps}{\varepsilon}
\newcommand{\R}{\mathbb{R}}
\newcommand{\Tr}{\mathrm{Tr}}
\newcommand{\Wass}{\mathbb{W}}

\title{Wasserstein-based Minimax Estimation of Dependence in Multivariate Regularly Varying Extremes}

\definecolor{ssw}{rgb}{0.1,0.45,0.1}

\definecolor{jb}{rgb}{0,0,1}

\definecolor{ymm}{rgb}{1,0.53,0.0}

\SetKwComment{Comment}{/* }{ */}

\author{Xuhui Zhang$^1$, Jose Blanchet$^1$, Youssef Marzouk$^2$, Viet Anh Nguyen$^3$, Sven Wang$^4$}

\date{}	
\date{%
    $^1$Stanford University\\%
    $^2$Massachusetts Institute of Technology\\%
    $^3$Chinese University of Hong Kong\\%
    $^4$Humboldt University Berlin
    \\[2ex]
}	

\begin{document}
\maketitle

\begin{abstract}
We present the first minimax risk bounds for estimators of the spectral measure in multivariate linear factor models, where observations are linear combinations of regularly varying latent factors. Non-asymptotic convergence rates are derived for the multivariate Peak-over-Threshold estimator in terms of the $p$-th order Wasserstein distance, and information-theoretic lower bounds for the minimax risks are established. The convergence rate of the estimator is shown to be minimax optimal under a class of Pareto-type models analogous to the standard class used in the setting of one-dimensional observations known as the Hall-Welsh class. When the estimator is minimax inefficient, a novel two-step estimator is introduced and demonstrated to attain the minimax lower bound.  Our analysis bridges the gaps in understanding trade-offs between estimation bias and variance in multivariate extreme value theory.
\end{abstract}

\section{Introduction}

Multivariate extreme value theory (MEVT) provides a framework for analyzing and understanding the \textit{joint} tail behaviors for observations in datasets comprising multiple variables. This theory finds applications across a diverse array of fields including financial risk management, environmental science, structural engineering, and telecommunications~\cite{ref:heffernan2004conditional,ref:coles1994statistical,ref:joe1994multivariate}. Comprehensive mathematical treatments of this theory are explored in \cite{ref:de2006extreme,ref:resnick2007heavy,ref:resnick2013extreme}, among others. In essence, MEVT involves two key components: the marginal distributions describing the behavior of each variable individually and the dependency structures capturing the relationships and dependencies among variables. 

While there exists an extensive body of literature dedicated to estimating the marginal distributions using univariate approaches~\cite{ref:hill1975simple,ref:pickands1971twodim,ref:pickands1975statistical,ref:david1990models,ref:smith1989extreme}, capturing the dependency structures in the tails continues to be a focal point of current research. Early work on this topic began with the bivariate characterizations by~\cite{ref:Sibuya1960}, while the corresponding multivariate characterizations were later formulated by~\cite{ref:deHaan1977limit,ref:pickands1981multivariate}. Several parametric models have been developed to describe the tail dependence, as in~\cite{ref:HUSLER1989283,ref:tawn1990model,ref:boldi2007mixture,ref:beirlant2004statistics}, among others. 

In this paper, we work within the framework of multivariate regular variation, which is intricately connected to the modeling of tail dependence in the study of extremes. See~\cite{ref:resnick2007heavy} for a comprehensive treatment of multivariate regular variation. 
The tail dependence structure of a multivariate regularly varying random vector is described using an angular or spectral measure on the unit sphere. The models we consider in the present paper follow from the linear factor model constructed in~\cite{ref:cooley2019decomposition}, where the vector of observations are linear combinations of regularly varying latent factors. As shown in~\cite{ref:cooley2019decomposition}, the spectral measure of this linear model is a Dirac mixture distribution. These spectral measures are dense (for example, in the weak topology) in the class of possible spectral measures, allowing approximations to any continuous spectral measure by increasing the number of latent factors. A closely related model to our study is the max-linear model comprising max-linear combinations of independent regularly varying random variables~\cite{ref:FOUGERES2013dense}, also see~\cite{ref:Gissibl2018maxlinear,ref:Gissibl2021identifiability,ref:KLUPPELBERG2021estimating} for its connection to Bayesian networks. 

In both the linear and the max-linear models, the quantity of interest to the statistician is typically the expectation of some function evaluated under the spectral measure. The statistical estimation procedure for this quantity typically employs the multivariate version of the Peak-over-Threshold method, yielding a subset of samples comprising the $k$ largest order statistics. Though much of the literature focuses on constructing subsets by determining the number of order statistics to include, this thresholding strategy is equivalent to choosing samples whose radius exceeds a large $\tau$. The true spectral measure is replaced
with the empirical measure of the angular component of the thresholded samples, under which the expected values are calculated. As the empirical  angular measure follows only approximately the limiting spectral measure, one must assess the impact of the inherent model misspecification. In this regard, ~\cite{ref:cooley2019decomposition} and~\cite{ref:KLUPPELBERG2021estimating} provide theoretical guarantees on the consistency and asymptotic normality of the empirical estimate. The forms of their  guarantees are derived from the central limit theorem  established in~\cite{ref:larsson2012extremal}, which, in turn, is based on the convergence theory of tail empirical processes. While technically elegant, this central limit theorem result also introduces a practical challenge regarding selecting an optimal threshold $\tau$: a higher value of $\tau$ leads to a lower estimation bias. Still, it simultaneously results in a larger estimation variance and vice versa. The implications of these trade-offs are not explicitly delineated in the central limit theorem. Besides, much less is known about the best attainable performance of arbitrary estimators for this problem.

The primary objective of this paper is to derive minimax risk bounds for estimators of the spectral measure arising from the linear factor models constructed in~\cite{ref:cooley2019decomposition}. As far as we know, this is the first paper that develops these minimax bounds in this setting. We consider a family of nonparametric models for the minimax model class $M = \cup_{k=1}^\infty M_k$. This model class has a simple baseline model in which the latent factors are i.i.d.~and follow the Pareto distribution. We consider around this baseline a sub-family of models $M_k$ (indexed by $k$), in which the latent factors can deviate from the baseline, with the deviance (in a precise sense) controlled by the rate $O(k^{-s})$ asymptotically as $k\to\infty$. Intuitively, this type of non-parametric perturbation
relative to the Pareto baseline provides the basis for estimation procedures that
threshold samples, as the bulk of the distributions carry less information. Our construction of the model class is motivated by the classical work~\cite{ref:hall1984best}. However, there are obvious differences. In particular, \cite{ref:hall1984best} considers the direct observation of the regularly varying random variable in a one-dimensional setting and considers the problem of minimax estimation of the regularly varying index (or the so-called extremal exponent). Also see~\cite{ref:drees1998optimal,ref:drees2001minimax,ref:BEIRLANT2006705} for further extensions of this problem. In contrast, we assume the knowledge of the regularly varying index for the latent factors. We are mainly interested in the tail dependence structure of the model dictated by its spectral measure.

In characterizing the convergence rate of the Peak-over-Threshold estimator, the central limit theorem in~\cite{ref:larsson2012extremal} becomes unwieldy due to its asymptotic nature and the lack of a uniform convergence result in the literature. As a technical contribution of our paper, we take an alternative approach employing a statistical distance to quantify the error between the empirical angular measure and the limiting spectral measure. Specifically, we employ the Wasserstein distance and establish an essentially non-asymptotic upper bound for this error measurement. A crucial step in our approach is a coupling argument between the pre-limit and the limiting spectral measures, which extends the arguments in the recent work~\cite{ref:bobbia2021coupling} from univariate extremes to the multivariate extremes setting. Alternatively, one can naturally formulate the problem with the Wasserstein distance being replaced by the information entropy (i.e., the Kullback–Leibler divergence). However, this approach is precluded as the empirical angular measure does not have common support with the limiting spectral measure. For a comprehensive treatment of the optimal transport theory, refer to~\cite{ref:villani2008optimal}. In recent years, there is a surge of interest in applying optimal transport formalism in statistics and machine learning~\cite{ref:panaretos2019statistical}, such as distributionally robust optimization~\cite{ref:blanchet_kang_murthy_2019,ref:gao2023distributionally}, generative models in deep learning~\cite{ref:martin2017wgan}, and domain adaptation~\cite{ref:courty2017optimal}, among others. 

We now summarize our main contributions. Under appropriate assumptions:
\begin{itemize}[leftmargin = 5mm]
    \item We analyze a Peak-over-Threshold estimator, the estimator that underpins previous works including~\cite{ref:cooley2019decomposition,ref:KLUPPELBERG2021estimating}, under the minimax model class. We derive its non-asymptotic convergence rate in terms of the $p$-th order Wasserstein distance $\Wass_p(\cdot,\cdot)$, where $p \geq 1$. 
    \item We derive novel information-theoretic lower bounds (the first in this setting) for the minimax risks corresponding to the chosen model class. We show that the Peak-over-Threshold estimator has a matching convergence rate when the value of $s$, a parameter controlling the discrepancy of the models from the family of Pareto distributions, is smaller than a critical value. In this case, the Peak-over-Threshold estimator is minimax optimal.
    \item When the Peak-over-Threshold estimator is minimax inefficient, we propose a novel two-step estimator based on decomposing the problem into one-dimensional estimating equations. We show that the proposed estimator attains the minimax lower bound under minor additional assumptions on the existence of another estimator for the normalizing constant of the tail of the latent factors.
\end{itemize}

The structure of the present paper is as follows. We review the theory of regular variation and introduce our minimax estimation problem in Section~\ref{sec:setting}. We present our main theoretical results in Section~\ref{sec:mainresults} and present simple
numerical experiments  that empirically validate our main results in Section~\ref{sec:numerics}. 

\section{Notations, Preliminaries and Problem Setup}\label{sec:setting}

\subsection{Notations}
We first collect some important notations used in the paper. We denote by $\mathbb{R}^d$ the $d$-dimensional Euclidean space, and by $\mathbb{R}_+^d = [0,\infty)^d$ the non-negative quadrant. For an arbitrary matrix $A\in\mathbb{R}^{d\times m}$ where $d\leq m$, we denote by $JA$ the Jacobian determinant $JA=\sqrt{|\textrm{det}(AA^\top)|}$  associated with the linear map from $\mathbb{R}^m$ to $\mathbb{R}^d$ induced by $A$, see~\cite[Theorem~3.6]{ref:evans2015measure}. We assume there exists a probability space $(\mathcal{M}, \mathcal{F}, P)$, in which all the random elements are defined. For a $d$-dimensional random vector $X$, we denote by $\mathcal{L}(X)$ the law of $X$. We denote by $X_i$ the $i$-th coordinate of $X$, and denote by $\{X^{(1)},\ldots,X^{(n)}\}$ a random sample of size $n$, where $X^{(i)}$ denotes the $i$-th sample. The symbol $\mathbb{P}_n$ denotes  the empirical distribution of these $n$ samples. The symbol $\delta_{x}$ (or $\delta_{\{x\}}$) denotes the Dirac measure concentrated at $x\in\mathbb{R}^d$. We denote by $E[X]$ the expectation of $X$, if it exists. We also denote by $E_Y[X] = E[X | Y]$ the conditional expectation of $X$ given $Y$. The symbol $\Rightarrow$ denotes convergence in distribution, and the symbol $\xrightarrow{v} $ denotes vague convergence (cf.~\cite[Section~3.3.5]{ref:resnick2007heavy}). We denote by $\|\cdot\|$ 
and $\|\cdot\|_1$ any generic norm and the $\ell_1$-norm, respectively, in the Euclidean space. We follow the usual conventions in using the symbols $o(\cdot)$ (small-o), $O(\cdot)$ (big-O), $\Theta(\cdot)$ (asymptotic equivalence) and $O_p(\cdot)$ (bounded in probability), and we use $\tilde o(\cdot)$, $\tilde O(\cdot)$, $\tilde\Theta(\cdot)$ to hide poly-log factors. Finally, for two positive sequences $a_n$ and $b_n$, $a_n\simeq b_n$ denotes $\lim_{n\to\infty}a_n/b_n=1$.


\subsection{Preliminary: Theory of Regular Variation}
The theory of regularly varying functions provides a useful mathematical framework for heavy-tailed analysis~\cite{ref:resnick2007heavy}. A function is regularly varying if it behaves asymptotically like a power function. In technical terms, the regularly varying function is defined as follows.
\begin{definition}
[{Regularly varying function~\cite[Definition 2.1]{ref:resnick2007heavy}}]
A measurable function $U:\mathbb{R}_+ \to\mathbb{R}_+$ is regularly varying at $\infty$ with index $\rho\in\mathbb{R}$ (written $U\in\mathrm{RV}_\rho$), if for $x>0$, 
\[
\lim_{t\to\infty}\frac{U(tx)}{U(t)} = x^\rho,
\]
where $\rho$ is called the exponent of variation.
\end{definition}
Note that we always write ``regularly varying'' as a shorthand for ``regularly varying at $\infty$''.  
Consider $X_1$ as a nonnegative random variable with a cumulative distribution function $F$, and denote the tail probability function by $\bar{F} = 1- F$. Then $X_1$ is called regularly varying with index $\alpha\in(0,\infty)$ if $\bar{F}\in\mathrm{RV}_{-\alpha}$ (see also~\cite[Theorem 3.6]{ref:resnick2007heavy}). 

A prototypical example of a regularly varying distribution with index $\alpha\in(0,\infty)$ is the Pareto distribution, which has the following form.
\begin{definition}[Pareto distribution]\label{def:ourparetodef}
The random variable $X_1\in\mathbb{R}_+$ follows the Pareto distribution with index $\alpha\in(0,\infty)$ if $X_1$ has the density
    \[f(x_1) = \alpha(1+x_1)^{-(\alpha+1)}\quad\forall x_1>0.\]
\end{definition}
The density function above is a translation of the density in the literature; see, for example,~\cite[Equation~(4.4)]{ref:resnick2007heavy}. 
In this paper, we estimate the extremal dependency of jointly distributed multivariate random vectors, which are regularly varying. Assuming that $X$ is a $d$-dimensional random vector that takes values in the nonnegative quadrant $\mathbb{R}_+^d$, there are several equivalent ways to define multivariate regular variation. 

\begin{definition}[{Multivariate regularly varying tail probabilities~\cite[Theorems 6.1 \& 3.6]{ref:resnick2007heavy}}]\label{def:multireg}
    A random vector $X\in\mathbb{R}_+^d$ is called regularly varying with index $\alpha\in(0,\infty)$ if $X$ satisfies any of the following conditions:
    \begin{enumerate}[label=(\roman*), leftmargin = 6mm]
        \item There exists a function $b(\cdot)\in \mathrm{RV}_{\frac{1}{\alpha}}$, and a non-negative Radon measure $\nu$ on $\mathbb{R}^d_+\backslash \{\mathbf{0}\}$, such that  
        \[
        n P\left(\frac{X}{b(n)}\in\cdot\right)\xrightarrow{\text{v}} \nu \quad \text{as }n\to\infty,
        \]
        where $\xrightarrow{\text{v}} $ denotes vague convergence in $M_+(\mathbb{R}^d_+\backslash\{ \mathbf{0}\})$, the set of non-negative Radon measures on $\mathbb{R}^d_+\backslash \{\mathbf{0}\}$.
        \item For any norm $\|\cdot\|$ in $\mathbb{R}^d$ and positive unit sphere $\aleph_+ =\{\omega\in\mathbb{R}^d_+:\|\omega\|= 1\}$, there exists a probability measure $S(\cdot)$ on $\aleph_+$, called the spectral measure, and a sequence $b_n\to\infty$ such that for $(R,\Upsilon) = (\|X\|,\frac{X}{\|X\|})$, we have
        \[
        nP\left(\left(\frac{R}{b_n},\Upsilon\right)\in\cdot\right)\xrightarrow{\text{v}} c\cdot\nu_\alpha\times S
        \]
        in $M_+((0,\infty]\times \aleph_+)$ for some $c>0$, where $\nu_\alpha\times S$ denotes the product measure of $\nu_\alpha$ and $S$, and $\nu_{\alpha}([r,\infty)) = r^{-\alpha}$ for all $r>0$.
    \end{enumerate}
\end{definition}

As a prototypical example, the product density of i.i.d.~Pareto distributions with the same index $\alpha$ is multivariate regularly varying, see \cite[Section~6.5]{ref:resnick2007heavy}. In this paper, we shall primarily use part (ii) of the definition to characterize multivariate regularly varying random vectors, while we occasionally use the equivalent form in (i). In part (ii), the tail dependence structure is summarized using the spectral measure, which is the object of interest in our study.



\subsection{Model Description and Statistic of Interest}
Consider a random vector $Z= (Z_1,\ldots,Z_m)\in\mathbb{R}^m_+$, which is multivariate regularly varying with index $\alpha$ and has i.i.d.~components. We assume that the observation $X$ follows a linear model given by 
\begin{equation}
\label{eq:XAZ}
X = AZ,\quad A\in\mathbb{R}_+^{d\times m},
\end{equation}
where $A$ is entry-wise non-negative and of dimension $d$-by-$m$. We interpret $Z$ as the latent risk factor.
This paper requires $X$ to be non-degenerate in $\mathbb{R}_+^d$. We also wish to allow flexibility for modeling the number of risk factors or, equivalently, the number of components of $Z$. Hence, we assume that $d\leq m$. In particular, as $m\to\infty$, the linear model~\eqref{eq:XAZ} approximates a non-parametric model~\cite[Proposition~4]{ref:cooley2019decomposition}. Our model is closely related to the max-linear models in~\cite{ref:FOUGERES2013dense,ref:Gissibl2018maxlinear,ref:Gissibl2021identifiability,ref:KLUPPELBERG2021estimating} with a notable difference being that we replace the max-linear operator therein by the canonical (additive-)linear operator in~\eqref{eq:XAZ}. Despite this difference, the models are asymptotically equivalent, and we explain their connection in more detail in Section~\ref{sec:connnectiontomaxlinear}, see also their comparisons in~\cite{ref:cooley2019decomposition}.

For any specific threshold $\tau>0$, we denote by $\mu^\tau_A$ the conditional measure given by
\begin{equation}\label{eq:muAtau}
\mu^\tau_A = \mathcal{L} \left(\frac{X}{\|X\|_1}\Big|\|X\|_1>\tau\right).
\end{equation}
The next lemma asserts that as $\tau\to\infty$, $\mu_A^\tau$ converges in distribution to a limiting law. 

\begin{lemma}[Limit law]\label{lem:KAconv}
    Assume that $Z = (Z_1,\ldots,Z_m)$ is multivariate  regularly varying with index $\alpha$ and has~i.i.d.~components. Given an entry-wise non-negative matrix $A$ such that each column $A_{\cdot i}$ satisfies  $\|A_{\cdot i}\|_1>0$. Let $K_A$ be the measure given by
    \begin{equation}\label{eq:limitcondmeasure}
K_A \Let \sum_{i=1}^m\frac{\|A_{\cdot i}\|_1^\alpha}{\sum_{j=1}^m\|A_{\cdot j}\|_1^\alpha}\delta_{\{\frac{A_{\cdot i}}{\|A_{\cdot i}\|_1}\}}.
\end{equation}
 Then $\mu^\tau_A \Rightarrow K_A$ as $\tau \to \infty$.
\end{lemma}
\begin{proof}[Proof of Lemma~\ref{lem:KAconv}]
Let $e\in\mathbb{R}^d$ denote the vector of all ones. Define $a\in\mathbb{R}^m_+$ as the vector with components $a_j = \|A_{\cdot j}\|_1$. Because $X$ is non-negative, 
    \[
\|X\|_1 = e^\top X = a^\top Z = \|Z\|_a,
\]
where $\|z\|_a = \sum_{i=1}^m a_i|z_i|$ defines a weighted $\ell_1$-norm in $\mathbb{R}^m$. Thus
\begin{equation}\label{eq:XandAZtransform}
\mathcal{L}\left(\frac{X}{\|X\|_1} \Big|\|X\|_1>\tau\right) = \mathcal{L}\left(A\frac{Z}{\|Z\|_a}\Big |\|Z\|_a>\tau\right).
\end{equation}
Since $Z$ is regularly varying with index $\alpha$ and has i.i.d.~components, there exists a sequence $b_n$ such that
\[
n P\left(\frac{ Z}{b_n}\in \cdot\right)\xrightarrow{v} \nu,
\]
where the measure $\nu$ has the form, see also~\cite[Equation (6.30)]{ref:resnick2007heavy},
\[
\nu(\textrm{d}z_1,\ldots,\textrm{d}z_m) = \sum_{j=1}^m \delta_0(\textrm{d}z_1)\times\cdots\times \delta_0(\textrm{d}z_{j-1})\times \nu_{\alpha}(\textrm{d}z_j)\times\delta_0(\textrm{d}z_{j+1})\times\cdots\times\delta_0(\textrm{d}z_m).
\]
Note that the $a$-simplex $\aleph^a_+ =\{\omega\in\mathbb{R}^m_+:\|\omega\|_a= 1\}$ has the basis vectors $\{ e_{\cdot i}/a_i,i=1,\ldots,m\}$, where $\{ e_{\cdot i}\}_{i=1,\ldots,m}$ forms the standard basis in $\mathbb{R}^m_+$.   Choosing the norm $\|\cdot\|_a$ in Definition~\ref{def:multireg} (ii), and noting that $\nu_\alpha([a_j^{-1},\infty)) = a_j^\alpha$ for all $j$, we have
\[
\mathcal{L}\left(\frac{Z}{\|Z\|_a}\Big|\|Z\|_a>\tau\right)  \Rightarrow\sum_{i=1}^m\frac{a_i^\alpha}{\sum_{j=1}^m a_j^\alpha} \delta_{\{\frac{ e_{\cdot i}}{a_i}\}},
\]
where the right hand side above is the angular measure in Definition~\ref{def:multireg} (ii). By continuous mapping theorem~\cite[Theorem 3.1]{ref:resnick2007heavy} and relation~\eqref{eq:XandAZtransform}, we have that
\[
\mathcal{L}\left(\frac{X}{\|X\|_1} \Big|\|X\|_1>\tau\right) \Rightarrow K_A.
\]
This completes the proof.
\end{proof}

\begin{figure}[h]
\caption{Schematic representation of the convergence to the spectral measure. In this drawing, the black dashed line is the positive simplex. The black dashed arrows point to the locations of the Dirac atoms in the limiting spectral measure. As $\tau$ increases, the probability distribution of $X/\|X\|_1$ concentrates around the Dirac atoms, as depicted by the blue drawings.}
\centering
\includegraphics[width=0.7\textwidth]{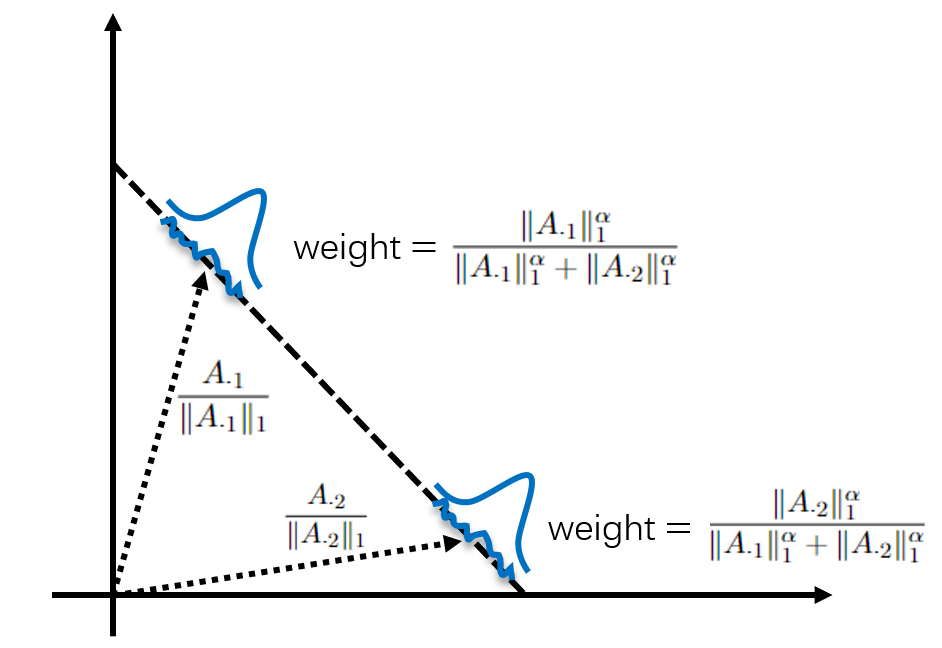}
\label{fig:spikerepr}
\end{figure}

Intuitively, the spectral measure $K_A$ represents some ``directions'' along which the components of $X$ jointly occur in the extremes, i.e., when the norm of $X$ is large. Equation~\eqref{eq:limitcondmeasure} asserts that the spectral measure $K_A$ is a mixture distribution of several Dirac measures. The weights of the mixture components are given by the relative magnitudes of the columns of $A$, while the locations of the Dirac atoms are determined by the directions of these columns. See Figure~\ref{fig:spikerepr} for a schematic representation of the convergence. Notice that a similar result is established in~\cite[Corollary~1]{ref:cooley2019decomposition}, though the authors work with a $\ell_2$-type norm (instead of our $\ell_1$-norm) and they restrict $\tau$ to a discrete sequence.  Our choice of $\ell_1$-norm in this paper is due to its more favorable algebraic properties in the analysis.

Given $n$ i.i.d.~samples $X^{(1)},\ldots, X^{(n)}$ distributed according to model~\eqref{eq:XAZ}, while $K_A$ is identifiable from Lemma~\ref{lem:KAconv}, the matrix $A$ itself is typically harder to estimate. The following proposition supports this assertion.
\begin{proposition}[Identifiability and degeneracy]\label{prop:relationAKA}
Under the assumptions of Lemma~\ref{lem:KAconv} and suppose additionally that $m\geq d \geq 2$, we have:
    \begin{enumerate}[label=(\roman*)]
        \item If the $m$ directions $\{\frac{A_{\cdot i}}{\|A_{\cdot i}\|_1}, i = 1,\ldots,m\}$ are all distinct, then $A$ is identifiable from model~\eqref{eq:XAZ} up to a permutation of columns. 
        \item If $Z_i$, $i=1,\ldots,m$ are i.i.d.~and have the common density $g(z_i) = 
 \alpha(1+z_i)^{-(\alpha+1)}$, then the Fisher information matrix of model~\eqref{eq:XAZ} is degenerate at all $A$'s where $JA>0$ and $A_{\cdot i } = A_{\cdot j}$ for some $i\neq j$. 
    \end{enumerate}
\end{proposition}
Here, we provide proof for part (i) of the proposition. The proof for part (ii) is relegated to Appendix~\ref{ap:settingproof}, where we compute the Fisher information matrix and construct a particular direction along which the Fisher information is degenerate. Here the Fisher information matrix of model~\eqref{eq:XAZ} is given by $E[\nabla^2_A \log \tilde g_A(X)]$, where $\nabla^2_A$ is the Hessian operator with respect to $A$ and the likelihood function $\tilde g_A(x)$ is given by a surface integral, refer to equation~\eqref{eq:surfaceint} in Appendix~\ref{ap:settingproof}.
\begin{proof}[Proof of Proposition~\ref{prop:relationAKA}(i)] First, from Lemma~\ref{lem:KAconv}, we know that $K_A$ is identifiable from model~\eqref{eq:XAZ}. Second, from the expression of $K_A$ in~\eqref{eq:limitcondmeasure}, we can identify the column directions of $A$ (i.e., the normalized vector $A_{\cdot i}/\|A_{\cdot i}\|_1$, because they are distinct), as well as their norms (i.e., $\|A_{\cdot i}\|_1$) up to a common constant scaling. Finally, $X$ and $tX$ cannot follow the same distribution if $t\neq1$. For example, since $JA>0$, we have that $X_i$ is non-zero and non-negative, hence $E[X_i^v]\neq E[(tX_i)^v]$ for any sufficiently small exponent $0<v<\alpha$ so that the expectation exists. Thus, we can identify $A$ up to a reordering of the columns.
\end{proof}

In light of our discussion, $K_A$ is our quantity of interest, representing the extremal dependency structure of observations $X$ from model~\eqref{eq:XAZ}. 

\subsubsection{Connection to the Recursive Max-linear Models}
\label{sec:connnectiontomaxlinear}
Our linear model~\eqref{eq:XAZ} is related to the \textit{recursive max-linear models} on directed acyclic graphs (DAGs) considered by~\cite{ref:Gissibl2018maxlinear,ref:Gissibl2021identifiability,ref:KLUPPELBERG2021estimating}. More specifically, consider a given DAG $\mathcal{D} = (V,E)$ with nodes $V=\{1,\ldots,d\}$ and edges $E = \{(k,i): i\in V \textrm{ and }k\in\textrm{pa}(i)\}$, where $\textrm{pa}(i)$ denotes the set of parents of node $i$. A recursive max-linear model on DAG $\mathcal{D}$ is defined by
\begin{equation}\label{kluppelbergweights}
X_i = \bigvee_{k\in\textrm{pa}(i)} c_{ki} X_k\vee c_{ii}Z_i, \quad i=1,\ldots,d,
\end{equation}
with independent, regularly varying random variables $Z_1,\ldots, Z_d\in\mathbb{R}_+$ and positive weights $c_{ki}$ and $c_{ii}$.
The results in~\cite{ref:Gissibl2018maxlinear,ref:Gissibl2021identifiability,ref:KLUPPELBERG2021estimating} indicate that the max-linear model can be equivalently expressed as
\begin{equation}\label{eq:kluppelbergreform}
X_i = \bigvee_{j=1}^d a_{ij}Z_j, \quad i =1,\ldots d, 
\end{equation}
where the coefficients $a_{ij}$ are determined by an appropriate balancing of the weights $c$ to calibrate~\eqref{eq:kluppelbergreform} with~\eqref{kluppelbergweights}. We notice that model~\eqref{eq:kluppelbergreform} differs from~\eqref{eq:XAZ} in using the max-linear operator. However, they are asymptotically equivalent in the sense that their angular measures coincide and are given by $K_A$~\cite[Remark 1.(i) and equation~(5)]{ref:KLUPPELBERG2021estimating}. Also notice that model~\eqref{eq:kluppelbergreform} restricts $m = d$, while we allow $m$ to be potentially \textit{larger} than $d$ in~\eqref{eq:XAZ}.
In particular, as $m\to\infty$, the linear model~\eqref{eq:XAZ} approximates a non-parametric model~\cite[Proposition~4]{ref:cooley2019decomposition},
hence, our model can approximate any continuous spectral measure by increasing the number of latent factors. We chose model~\eqref{eq:XAZ} over model~\eqref{eq:kluppelbergreform} in this paper because the linear operator is more amenable for analysis. However, our techniques in analyzing the asymptotic convergence rate are generalizable to the case with the max-linear operator.

\subsection{Minimax Estimation Risk over Pareto-like Models}
Let $\hat{K}$ denote a generic estimator for $K_A$, i.e.,~a probability measure supported in  the positive 1-simplex $\aleph_+^1 = \{\omega\in\mathbb{R}^d_+:\|\omega\|_1 = 1\}$. We quantify the estimation error by the $p$-th order Wasserstein distance, where $p\in[1,\infty)$. The Wasserstein distance was previously used in univariate extreme value theory to compare the empirical distribution of exceedances to the empirical distribution associated with the limit model~\cite{ref:bobbia2021coupling}. In general, for two probability measures $\mu$ and $\mu'$ supported in an arbitrary Polish space $\mathcal{S}$, the $p$-th order Wasserstein distance between them is defined as 
\begin{equation}\label{eq:wassdef}
\mathbb{W}_p(\mu,\mu') =  \left(\inf_{\gamma\in\Gamma(\mu,\mu')}E_{(y,y')\sim\gamma} \left[d(y,y')^p\right]\right)^{\frac{1}{p}},
\end{equation}
where $\Gamma(\mu,\mu')$ is the set of all couplings of $\mu$ and $\mu'$. A coupling $\gamma\in\Gamma(\mu,\mu')$ is a joint probability measure on $\mathcal{S}\times\mathcal{S}$ whose marginals are $\mu$ and $\mu'$ in
the first and second coordinates, respectively. Here $d(\cdot,\cdot)$ is the ground transportation cost on the space $\mathcal{S}$.  In our setting, we have $\mathcal{S} = \aleph_+^1$ and make the following assumption.

\begin{assumption}[Ground transportation cost] \label{a:cost}
Consider the transportation cost for  the Wasserstein distance as $d(\omega, \omega') = \|\omega-\omega'\|_1$ for any $\omega$ and $\omega'\in\aleph_+^1$.
\end{assumption}

In this paper, we establish minimax upper and lower bounds to quantify the statistical optimality of certain estimation procedures. In this view, we are interested in the worst-case performance of  estimation procedures for the model~\eqref{eq:XAZ} over a  family of matrices $A$ and distributions for $Z$. In particular, we define the class of models as follows
\begin{definition}[Model class] \label{def:model}
    We assume a class of models 
    \begin{subequations} \label{eq:minimaxmodel}
\begin{equation}\label{eq:M}
M = \cup_{k=1}^\infty M_k,
\end{equation}
where each submodel class is
$M_k = \left\{ \mathcal{L}(X): X = AZ, ~\textrm{for } A\in\mathcal{A}~\textrm{and } \mathcal{L}(Z)\in \tilde M_k\right\}$. The set of possible distributions for the components $Z$ is
\begin{equation} \label{eq:lawZ}
\tilde M_k = \left\{
\mathcal{L}(Z) : \begin{array}{l}
Z \textrm{ admits a (Lebesgue) density }g(z)\textrm{ in }\mathbb{R}^m_+ \\
\left|\frac{g(z) - \alpha^m\prod_{i=1}^m (1+z_i)^{-(\alpha+1)}}{\alpha^m\prod_{i=1}^m (1+z_i)^{-(\alpha+1)}}\right|\leq \xi  k^{-s}~\forall z \\
g(z) \propto\prod_{i=1}^m (1+z_i)^{-(\alpha+1)} \textrm{ if } \|z\|_1> \zeta  k^{\frac{1-2s}{\alpha}}
\end{array}
\right\} ,
\end{equation}
and the set of possible  matrices $\mathcal A$ is
\[
\mathcal{A} = \left\{A\in\mathbb{R}_+^{d\times m}: l \leq \min_{i}\|A_{\cdot i}\|_1 \leq \max_i\|A_{\cdot i}\|_1\leq u\textrm{ and }JA\geq \sigma\right\}.
\]
\end{subequations}
Throughout, we assume the constants satisfy $m\geq d\geq2$, $0<l<1<u$, $\sigma>0$, $0<s<\frac{1}{2}$, $ 0<\xi<1$, and $\zeta>0$.
\end{definition}

\begin{figure}[h]
\caption{Schematic representation of the structure of $g(z)$. In this drawing, the black solid line is the Pareto density, and the black dashed lines are the ranges allowable for the perturbations, essentially an $L_\infty$ ball (in percentage terms) around the Pareto density. The black vertical line represents the threshold $\zeta k^{\frac{1-2s}{\alpha}}$. Above the threshold, the density is required to be proportional to the exact Pareto density, which we showcase by the red solid line. Below the threshold, the density can vary arbitrarily within the $L_\infty$ ball, which we represent by the red dashed arrows. Our construction is motivated by the Hall-Welsh class in~\cite{ref:hall1984best}.}
\centering
\includegraphics[width=0.7\textwidth]{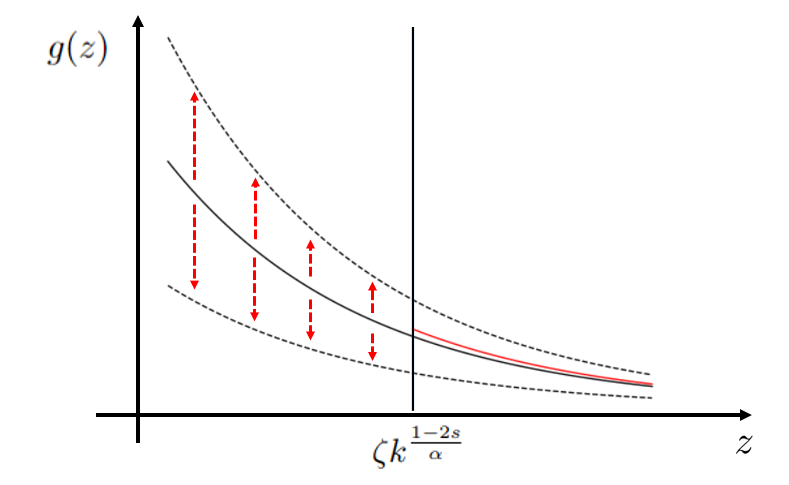}
\label{fig:schematicrepr}
\end{figure}

The model class $M_k$ is induced by two classes $\tilde M_k$ and $\mathcal A$. The class $\tilde M_k$ modeling the perturbation on the distribution of $Z$ can be explained as follows: First, there is a baseline model in which the components of $Z$ are i.i.d.~and each follows the Pareto distribution with known parameter $\alpha$, which implies that $Z$ admits the density $\alpha^m\prod_{i=1}^m (1+z_i)^{-(\alpha+1)}$. We consider around this Pareto baseline a family of non-parametric models $\tilde M_k$ indexed by $k$, where the latent factors $Z$ can deviate from the baseline with the discrepancy controlled by the rate $k^{-s}$. Further, we impose that the density of $Z$ is proportional to a Pareto density above a certain threshold of order $k^{\frac{1-2s}{\alpha}}$, but the density can vary freely below this threshold. See Figure~\ref{fig:schematicrepr} for a schematic representation of this structure. As $k$ increases, the overall discrepancy $k^{-s}$ shrinks; however, this also leads to an expansion of the region where the density is modeled non-parametrically. Any reasonable estimation procedure would need to balance this tradeoff carefully. Intuitively, this perturbation relative to the Pareto baseline motivates a thresholding procedure.

The class of matrices $\mathcal{A}$ is easier to parse: we consider all matrices whose column norms are bounded between $l$ and $u$ (hence the set $\mathcal{A}$ is compact). Note that the assumptions $m \geq d$ and $JA \geq \sigma$ are required to ensure that $X$ admits a density in $\mathbb{R}^d_+$ by the generalized version of change-of-variables formula~\cite[Theorem~3.11]{ref:evans2015measure}. 

Our construction of the model class is motivated by the Hall-Welsh class in~\cite{ref:hall1984best}. However, there are obvious differences. In particular, \cite{ref:hall1984best} considers the direct observation of $Z$ in a one-dimensional setting and considers the problem of minimax estimation of the regularly varying index $\alpha$. Also see~\cite{ref:drees1998optimal,ref:drees2001minimax,ref:BEIRLANT2006705} for further extensions. In contrast, we assume the knowledge of $\alpha$ and are interested in the joint extremal dependency structure of the vector of observation $X$, which is dictated by the spectral measure $K_A$.


We also define the estimation risk over this model class as follows.
\begin{definition}[Convergence rate]
    For any estimator $\hat K$ of $K_A$, we say that $\hat K$ has a convergence rate $\beta_n$ over the model class $M$ if for any $\eps
>0$, there exists a constant $C_\eps$ such that for all $n$ large enough
\[ 
\sup_{\mathcal L(X) \in M} P(\mathbb{W}_p^p(\hat K, K_A)\ge C_\eps \beta_n )\le \eps.
\]
\end{definition}

This paper studies the minimax convergence rate, the fastest convergence rate achievable over all possible $\hat{K}$, uniformly over the class $M$. In the next section, we investigate these rates by establishing minimax upper and lower bounds.

\section{Main Results}\label{sec:mainresults}

Among all parameters that jointly specify the model $M$ in~\eqref{eq:minimaxmodel}, we highlight the dependency on the parameter $s$, as it is an important parameter for our subsequent analysis. One can view $s$ as a precision parameter that determines the size of the model class: a bigger value of $s$ indicates that the distribution of $Z$ is closer to a Pareto distribution. On the other extreme, if $s$ is small, the distribution of $Z$ can deviate further away from the Pareto distribution.  In our analysis, an important insight emerges regarding the performance of estimators in different regimes defined by $s$. Specifically, in the ``small $s$ regime'', a non-parametric estimator proves to be minimax optimal.  Conversely, in the ``large $s$ regime'', the situation demands a more sophisticated estimator to handle the more nuanced structure of the models effectively.

\subsection{Minimax Upper Bound}\label{sec:upperbound}

A conventional strategy in heavy-tailed analysis is truncating samples at a chosen threshold and using only samples exceeding the threshold for estimation purposes. This strategy has gained many successes in~\cite{einmahl2001nonparametric,einmahl2009maximum,einmahl2012mestimator,drees2015estimating,DEHAAN2008parametric}, among others.

Given $n$ i.i.d.~samples $X^{(1)},\ldots, X^{(n)} \in \R^d_+$ distributed according to a member of the model family~\eqref{eq:minimaxmodel}, we denote by $\mathbb{P}_n$ the empirical measure supported on these samples. Given a threshold level $\tau$, let $n_\tau$ be the number of samples whose $\ell_1$-norm exceeds $\tau$:
\[
    n_\tau = \#\{ 1 \le i \le n: \| X^{(i)}\|_1 > \tau \}.
\]
We denote by $\mathbb{P}_n^\tau$ the empirical measure of the normalized (using the $\ell_1$-norm) and thres-holded (at a level $\tau$) samples, i.e., 
\[
    \mathbb{P}_n^\tau = \frac{1}{n_\tau} \sum_{i: \|X^{(i)} \|_1 > \tau} \delta_{\{\frac{X^{(i)}}{\|X^{(i)}\|_1}\}}.
\]
Note that $\mathbb{P}_n^\tau$ is a probability measure supported on $\aleph^1_+$, and $n_\tau$ is an integer-valued random variable whose distribution depends on the number of samples $n$ and the threshold $\tau$. It is reasonable to expect $n_\tau$ to concentrate around its expectation $nP(\|X\|_1>\tau)$ when both $n$ and $\tau$ tend to infinity in an appropriate fashion. To simplify the exposition, we denote by $X^{(i,n_\tau)}$ for $1\leq i\leq n_\tau$ the $i$-th normalized and thresholded sample that constitutes the measure $\mathbb{P}_n^\tau$.

We can use the empirical measure as a straightforward estimator of $K_A$, i.e., use $\hat{K} = \mathbb{P}_n^\tau$ with an appropriately selected threshold $\tau$. Note that the samples that constitute $\mathbb{P}_n^\tau$ are~i.i.d.~with the common law $\mu_A^\tau$ defined in~\eqref{eq:muAtau}. Motivated by the coupling method in~\cite{ref:bobbia2021coupling}, we consider a set of samples $\{Y^{(1)},\ldots,Y^{({n_\tau})}\}$ which are i.i.d.~generated from the law $K_A$, such that each pair of data $(X^{(i,n_\tau)},Y^{(i)})$ are coupled by the optimal coupling corresponding to $\mathbb{W}_p(\mu^\tau_A,K_A)$ (cf.~minimizer of the optimal transport problem~\eqref{eq:wassdef}). This optimal coupling exists according to~\cite[Theorem~4.1]{ref:villani2008optimal}. We denote by $\mathbb{K}_A^{n_\tau}$ the empirical measure of this set of virtual samples $\{Y^{(1)},\ldots,Y^{({n_\tau})}\}$. By the triangle inequality,
\[
\mathbb{W}_p(\mathbb{P}_n^\tau,K_A)\leq \mathbb{W}_p(\mathbb{P}_n^\tau,\mathbb{K}_A^{n_\tau}) + \mathbb{W}_p(\mathbb{K}_A^{n_\tau},K_A).
\]
Hence by the inequality $(a+b)^p \leq 2^{p-1} (a^p + b^p)$ for any $a, b\geq0$, we obtain
\[
\mathbb{W}_p^p(\mathbb{P}_n^\tau,K_A)\leq 2^{p-1}\left(\mathbb{W}_p^p(\mathbb{P}_n^\tau,\mathbb{K}_A^{n_\tau}) + \mathbb{W}_p^p(\mathbb{K}_A^{n_\tau},K_A)\right).
\]
Let $E_{n_\tau}$ denote the conditional expectation conditioning on the value of ${n_\tau}$, by taking expectations on both sides and by the linearity of expectation, we find
\[
E_{n_\tau}[\mathbb{W}_p^p(\mathbb{P}_n^\tau,K_A)]\leq 2^{p-1}\left( E_{n_\tau}[\mathbb{W}_p^p(\mathbb{P}_n^\tau,\mathbb{K}_A^{n_\tau})] +E_{n_\tau}[ \mathbb{W}_p^p(\mathbb{K}_A^{n_\tau},K_A)]\right).
\]
To proceed, we will look at the upper bound on the right-hand side as the bias-variance decomposition.
From~\cite[Theorem 2.3]{ref:bobbia2021coupling}, we know that 
\[
E_{n_\tau}[\mathbb{W}_p^p(\mathbb{P}^\tau_n,\mathbb{K}_A^{n_\tau})]= \mathbb{W}_p^p(\mu_A^\tau,K_A).
\]
The first insight is to interpret $\mathbb{W}_p(\mu_A^\tau,K_A)$ as the bias for the estimator $\mathbb{P}_n^\tau$. 
The following proposition quantifies the worst-case error of the bias term over model family $M$ in~\eqref{eq:minimaxmodel}.

\begin{proposition}[Bias term]\label{thm:bias} Let $\mu^\tau_A$ and $K_A$ be defined in equations~\eqref{eq:muAtau} and~\eqref{eq:limitcondmeasure}, respectively. We have
\[
\sup_{\mathcal{L}(X)\in M}\mathbb{W}_p^p(\mu_A^\tau,K_A) \leq \tilde O\left(\tau^{-\min\{1,\alpha\}}\right) + O\left(\tau^{-\frac{\alpha s}{1-2s}}\right)\quad\text{as }\tau\to\infty.
\]
\end{proposition}
Proposition~\ref{thm:bias} dictates that the bias term converges in the order of $\tilde O(\tau^{-\alpha}) + O(\tau^{-\frac{\alpha s}{1-2s}})$ if $\alpha\leq1$, and in the order of $\tilde O(\tau^{-1}) + O(\tau^{-\frac{\alpha s}{1-2s}})$ if $\alpha>1$. 

Moving to the second term in the upper bound, we now analyze the worst-case magnitude of the conditional variance $E_{n_\tau}[ \mathbb{W}_p^p(\mathbb{K}_A^{n_\tau},K_A)]$. Proposition~\ref{thm:limitingmeasureconcentrationbound} below dictates that the conditional variance term decays in the order of $O(\frac{1}{\sqrt{n_\tau}})$. 

\begin{proposition}[Conditional variance term]\label{thm:limitingmeasureconcentrationbound}Conditioning on any ${n_\tau} \geq1 $, we have
\[
\sup_{\mathcal{L}(X)\in M} E_{n_\tau}[\mathbb{W}_p^p(\mathbb{K}_A^{n_\tau},K_A)] \leq  \sqrt{\frac{2^{2p+3} m  \log(2)}{n_\tau}}.
\]
\end{proposition}

Finally, we obtain the following result on the convergence of  $\mathbb{W}_p^p(\mathbb{P}_n^{\tau}, K_A)$ by synthesizing the bias and variance analysis. An essential fact we utilize is that $n_\tau$  concentrates around its expectation $nP(\|X\|_1>\tau)$ as $n$ and $\tau$ tend to infinity.

\begin{proposition}[General convergence theorem]\label{thm:variance}
There exists some constant $C$ such that for any $\eps\in(0,1)$ and any sequence of positive numbers $(\tau_n)_{n=1}^\infty$ satisfying $\tau_n\to\infty$ and $n\tau_n^{-\alpha}\to\infty$ as $n\to\infty$, 
\[
\liminf_{n\to\infty} \inf_{\mathcal{L}(X)\in M} P\left(\mathbb{W}_p^p(\mathbb{P}_n^{\tau_n},K_A)  \leq  2^{p-1}\mathbb{W}_p^p(\mu_A^{\tau_n},K_A) +  C\frac{1 + \sqrt{\log\frac{1}{\eps}}}{\sqrt{nP(\|X\|_1>\tau_n)}}\right)\geq1-\eps.
\]
\end{proposition}

Combining Proposition~\ref{thm:bias} and Proposition~\ref{thm:variance}, we have that with high probability, $\mathbb{W}_p^p(\mathbb{P}^\tau_n,K_A)$ is uniformly convergent in an order
\[
\mathbb{W}_p^p(\mathbb{P}^\tau_n,K_A) \leq \underbrace{ \tilde O\left(\tau^{-\min\{1,\alpha\}}\right) + O\left(\tau^{-\frac{\alpha s}{1-2s}}\right)}_\text{Bias}+\underbrace{O\left(\frac{1}{\sqrt{n P(\|X\|_1>\tau)}}\right)}_\text{Variance}.
\]
The leading order of the bias term is notably influenced by the value of $s$, which controls the deviation rate away from the Pareto nominal distribution in~\eqref{eq:minimaxmodel}. This leads to two distinct regimes on the convergence rate, asserted by the next theorem.

\begin{theorem}[Convergence rates]\label{thm:lower} 
The following holds.
\begin{enumerate}[label=(\roman*)]
    \item If $\frac{1}{2+\max\{1,\alpha\}}\leq s<\frac{1}{2}$, then by choosing $\tau_n = \tilde\Theta(n^{\frac{1}{\min\{2+\alpha,3\alpha\}}})$ which balances the bias and variance terms, we obtain that for any $\eps\in(0,\frac{1}{2})$,
    \[
    \liminf_{n\to\infty} \inf_{\mathcal{L}(X)\in M} P\left(\mathbb{W}_p^p(\mathbb{P}_n^{\tau_n},K_A)  \lesssim (-\log{\eps})^{\frac{1}{2}}n^{-\frac{1}{2+\max\{1,\alpha\}}}\right)\geq1-\eps.
    \]
    where $\lesssim$ hides all constant factors and a poly-log factor of $n$.
    \item If $0< s<\frac{1}{2+\max\{1,\alpha\}}$, then by choosing $\tau_n =\Theta(n^{\frac{1-2s}{\alpha}}) $ which balances the bias and variance terms, we obtain that for any $\eps\in(0,\frac{1}{2})$,
    \[
     \liminf_{n\to\infty} \inf_{\mathcal{L}(X)\in M} P\left(\mathbb{W}_p^p(\mathbb{P}_n^{\tau_n},K_A)  \lesssim (-\log{\eps})^{\frac{1}{2}}n^{-s}\right)\geq1-\eps,
    \]
    where $\lesssim$ hides all constant 
 factors.
\end{enumerate}   
\end{theorem}
\begin{proof}[Proof of Theorem~\ref{thm:lower}]
If $\frac{1}{2+\max\{1,\alpha\}}\leq s<\frac{1}{2}$, we have $O(\tau^{-\frac{\alpha s}{1-2s}}) = O(\tau^{-\min\{1,\alpha\}})$. Balancing the term $\tilde O(\tau^{-\min\{1,\alpha\}})$
from the bias and the term $O\left(\left(nP(\|X\|_1>\tau_n)\right)^{-\frac{1}{2}}\right)=O\left( (n \tau_n^{-\alpha})^{-\frac{1}{2}}\right)$
from the variance gives us the optimal order $\tilde O(n^{-\frac{1}{2+\max\{1,\alpha\}}})$. If $0< s<\frac{1}{2+\max\{1,\alpha\}}$, we have $\tilde O(\tau^{-\min\{1,\alpha\}}) = o\left(\tau_n^{-\frac{\alpha s}{1-2s}}\right)$. 
Balancing the term $O(\tau_n^{-\frac{\alpha s}{1-2s}})$
from the bias and the term $O\left( (n \tau_n^{-\alpha})^{-\frac{1}{2}}\right)$
from the variance gives us the optimal order $O(n^{-s})$.
\end{proof}

We note that the upper bound in Theorem~\ref{thm:lower} is sharp for the case $\frac{1}{2+\max\{1,\alpha\}}< s<\frac{1}{2}$. In particular, we have the following lower bound on the convergence rate for this estimator. 
\begin{proposition}[Sharpness]\label{prop:sharpupperbound}
    Suppose $\frac{1}{2+\max\{1,\alpha\}}< s<\frac{1}{2}$, and  let $\tau_n$ and $\beta_n$ be sequences of positive numbers. Suppose that for any $\eps\in(0,1)$, there exists a constant $C_\eps$, such that
    \[
    \liminf_{n\to\infty} \inf_{\mathcal{L}(X)\in M} P\left(\mathbb{W}_p^p(\mathbb{P}_n^{\tau_n},K_A)  \leq C_\eps \beta_n^p  \right)\geq1-\eps.
    \]
    Then it must be that 
    \[
    \liminf_{n\to\infty}n^{\frac{1}{2+\max\{1,\alpha\}}}\beta_n^p >0.
    \]
\end{proposition}
For the case $0< s\leq\frac{1}{2+\max\{1,\alpha\}}$, the upper bound in Theorem~\ref{thm:lower} is again sharp, quotient a small probability of $\eps$, as demonstrated by the minimax lower bound stated in the next section.

\subsection{Minimax Lower Bound}\label{sec:lowerbound}
Now, we derive a lower bound for the minimax risk of estimation over the class of models $M$, again in the $p$-th order Wasserstein distance with respect to the $\ell_1$-norm. 

\begin{theorem}[Minimax lower bound] \label{thm:lowerbound}
Let $\beta_n$ be a sequence of positive numbers. 
    Suppose that an estimator $\hat{K}=\hat{K}(X^{(1)},\ldots,X^{(n)})$ satisfies
\begin{equation}\label{eq:lbbetadef}
\liminf_{n\to\infty}\inf_{\mathcal{L}(X)\in M}P(\mathbb{W}_p^p(\hat{K},K_A)\leq\beta_n^p) =  1.
\end{equation}
Then it must be that $
\liminf_{n\to\infty} n^s \beta_n^p >0$.
\end{theorem}
Theorem~\ref{thm:lowerbound} establishes the information-theoretic limit for the minimax risk to be of order $O(n^{-s})$. To derive this lower bound, we employ a general reduction scheme to hypothesis testing and a change-of-measure argument, which traces back to the classical works~\cite{ref:hall1984best,Farrell1972on}. In particular, we consider two  hypotheses $A,\check{A}\in\mathcal{A}$, and demonstrate that the minimax risk is lower bounded by their separation in Wasserstein distance. To simplify our notation, we suppress the explicit dependence of the choice of hypotheses on the sample size $n$.
\begin{lemma}[Reduction to hypothesis testing]\label{lem:changeofmeasure}
    Consider a pair of random vectors $X=AZ$ and $\check X=\check A\check Z$ such that $\mathcal{L}(X),\mathcal{L}(\check X)\in M$ and denote by $f$ and $\check f$ their corresponding density functions on $\mathbb{R}_+^d$. Suppose that their likelihood ratio satisfies $E\left[\left(\frac{\check f(X)}{f(X)}\right)^2\right]= 1+O(n^{-1})$. Then it must be that 
    \begin{equation}\label{eq:hyposep}
    \beta_n^p\geq\frac{1}{2^p}\mathbb{W}_p^p(K_A,K_{\check A}),
    \end{equation}
    for all $n$ sufficiently large, where $\beta_n$ is defined in~\eqref{eq:lbbetadef}.
\end{lemma}
The right-hand side in equation~\eqref{eq:hyposep} quantifies the separation between the two hypotheses. In light of Lemma~\ref{lem:changeofmeasure}, the proof of Theorem~\ref{thm:lowerbound} now consists in constructing these hypotheses with $n^{-s}$ separation. 

We conclude from Theorem~\ref{thm:lower}, Proposition~\ref{prop:sharpupperbound} and Theorem~\ref{thm:lowerbound} that
\begin{enumerate}[label=(\roman*)]
    \item If $0< s<\frac{1}{2+\max\{1,\alpha\}}$, then letting $\hat{K}=\mathbb{P}^{\tau_n}_n$, where $\tau_n$ is optimally chosen according to Theorem~\ref{thm:lower}, we obtain an estimator for $K_A$ which is minimax optimal, quotient a small probability of $\eps$, and we find the optimal sample complexity to be $O(n^{-s})$. 
    \item If $s=\frac{1}{2+\max\{1,\alpha\}}$, the convergence rate of the estimator $\hat{K}  =\mathbb{P}^{\tau_n}_n$ is  $\tilde O(n^{-s})$, which differs from the minimax lower bound by at most a poly-log factor.
    \item If $\frac{1}{2+\max\{1,\alpha\}}<s<\frac{1}{2}$, however, we find that this straightforward estimator does not match the minimax lower bound. In particular, the convergence rate for the estimator has an order $\tilde O(n^{-\frac{1}{2+\max\{1,\alpha\}}})$, while the minimax lower bound has an order $O(n^{-s})$. 
\end{enumerate}
Cases (i) and (ii) indicate that the Peak-over-Threshold estimator $\hat{K}=\mathbb{P}^{\tau_n}_n$ is minimax optimal (or nearly optimal) when the value of $s$, a parameter controlling the discrepancy of the models from the family of Pareto distributions, is smaller than (or equal to) the critical value $\frac{1}{2+\max\{1,\alpha\}}$. Intuitively speaking, a small value of $s$ corresponds to models that are \textit{unlike} Paretos, as the ``perturbation error'' $O(k^{-s})$ decreases slowly, as $k\to\infty$. In such scenarios, the relatively crude estimator $\hat{K}=\mathbb{P}^{\tau_n}_n$ turns out to be sufficient for optimality. For Case (iii), we enter into the regime where the models are \textit{like} Paretos, as the ``perturbation error'' $O(k^{-s})$ decreases rapidly as $k\to\infty$. We deal with this scenario in the next section.

\subsection{A Potential Minimax Estimator for $s>\frac{1}{2+\max\{1,\alpha\}}$}\label{sec:potentialestimator}

This section investigates an alternative estimator that achieves minimax convergence rate $O(n^{-s})$ under minor additional assumptions when  $s>\frac{1}{2+\max\{1,\alpha\}}$. Throughout the section, we consider the slightly restrictive case when $A\in\mathbb{R}_+^{d\times m}$ is a square matrix, namely, $d=m$. 
\subsubsection{A Motivating Example For Univariate Case} \label{sec:potential1dim}

It is appealing to consider a maximum likelihood estimator using samples whose values are above a carefully selected threshold $\tau$. Unfortunately, we illustrate that this approach may fail even in a one-dimensional analog of our problem when $d = 1$. Consider the parametric family of distributions for $X_1$ given by $X_1 = \theta_1 Z_1$, where the non-negative random variable $Z_1$ follows a Pareto distribution with index $\alpha$. In other words, $X_1$ admits the density
\[
f(x_1) = \frac{1}{\theta_1}\alpha(1+\frac{x_1}{\theta_1})^{-(\alpha+1)} \quad \forall x_1>0.
\]
We denote by $\mathcal{L}_{\theta_1}(X_1|X_1>\tau)$ the law for the conditional distribution under parameter $\theta_1$. Given a high enough threshold $\tau$, we have the density
\begin{align*}
f_\tau(x_1) &= \frac{f(x_1)}{\mathrm{Prob}(X_1 > \tau)}\\
&=\frac{1}{\theta_1}\alpha(1+\frac{x_1}{\theta_1})^{-(\alpha+1)} / (1+\frac{\tau}{\theta_1})^{-\alpha}\quad \forall x_1>\tau.
\end{align*}
Consequently, the log-likelihood can be computed as
\[
\log f_\tau(x_1) = -\log\theta_1 +\log\alpha -(\alpha+1)\log(1+\frac{x_1}{\theta_1}) +\alpha\log(1+\frac{\tau}{\theta_1})\quad \forall x_1>\tau.
\]
Differentiating the log-likelihood function twice, we obtain the second derivative: 
\[
(\log f_\tau )^{''}(x_1) = \frac{1}{\theta_1^2}\left((\alpha+1)\frac{\theta_1^2}{(\theta_1+x_1)^2} - \alpha\frac{\theta_1^2}{(\theta_1+\tau)^2}\right)\quad \forall x_1>\tau.
\]
Since $\tau\to\infty$, the second derivative function converges to zero uniformly for $\{x_1:x_1\geq\tau\}$. Hence, by the dominated convergence theorem, the Fisher information for the conditional law $\mathcal{L}_{\theta_1}(X_1|X_1>\tau)$ degenerates asymptotically. Hence, the usual theory of parametric likelihood inference, as in~\cite[Chapter~5]{ref:vaart1998asymptotic}, does not apply to establish the square-root consistency for $\theta_1$ using the conditional law $\mathcal{L}_{\theta_1}(X_1|X_1>\tau)$, when $\tau\to\infty$. An alternative view is that the conditional density can be re-written as (after considering the transformation $\tilde x = x/\tau$, hence $\tilde x\geq1$)
\begin{align*}
\tilde f_\tau(\tilde x_1) = \frac{\tau}{\theta_1}\alpha(1+\frac{\tau}{\theta_1}\tilde x_1)^{-(\alpha+1)}(1+\frac{\tau}{\theta_1})^{\alpha} &= \alpha\frac{\frac{\tau}{\theta_1}}{1+\frac{\tau}{\theta_1}} \left(\frac{1+\frac{\tau}{\theta_1}\tilde x_1}{1+\frac{\tau}{\theta_1}}\right)^{-(\alpha+1)} \\
&\to \alpha \tilde x_1^{-(\alpha+1)}\textrm{ as }\tau\to\infty,
\end{align*}
where we highlight that information of $\theta_1$ is lost in the limiting density function.

Besides the conditional information, another important piece of information is the tail probability of threshold exceedance, which we show to lead to an estimation procedure that is `square-root consistent' in the number of samples exceeding $\tau$. We carefully select $\tau$ to show that the procedure attains the minimax convergence rate $O(n^{-s})$. In this section, we first discuss intuitions and state the results in the one-dimensional setup (corresponding to $d=1$). We will generalize the ideas to multivariate setups ($d \geq 2$) in Section~\ref{sec:potentialmultidim}.

First, we consider the one-dimensional analog of the model family~\eqref{eq:minimaxmodel}. Assume a family of models 
\begin{subequations} \label{eq:minimaxmodel1dim}
\begin{equation}\label{eq:Monedim}
M^1 = \cup_{k=1}^\infty M^1_k,
\end{equation}
where
\[
M_k^1 = \left\{ \mathcal{L}(X_1): X_1 = \theta_1 Z_1 ~\textrm{for some } \theta_1\in[l,u]~\textrm{and } \mathcal{L}(Z_1)\in \tilde M^1_k\right\},\]
and
\begin{equation} \label{eq:lawZ1dim}
\tilde M_k^1 = \left\{
\mathcal{L}(Z_1) : \begin{array}{l}
Z_1 \textrm{ admits a (Lebesgue) density }g_1(z_1)\textrm{ in }\mathbb{R}_+ \\
\left|\frac{g_1(z_1) -\alpha(1+z_1)^{-(\alpha+1)}}{\alpha(1+z_1)^{-(\alpha+1)}}\right|\leq \xi k^{-s}~\forall z_1\\
g_1(z_1) \propto (1+z_1)^{-(\alpha+1)} \textrm{ if } z_1> \zeta  k^{\frac{1-2s}{\alpha}}
\end{array}
\right\} . 
\end{equation}
\end{subequations}

We denote by $\{X_1^{(1)},\ldots,X_1^{(n)}\}$ arbitrary~i.i.d.~samples of size $n$. Our procedure to estimate $\theta_1$ requires the existence of another estimator $\hat{r}$ for the normalizing constant appearing in 
\[
g_1(z_1) \propto (1+z_1)^{-(\alpha+1)},
\]
i.e., the last line of equation~\eqref{eq:lawZ1dim}. More precisely, let
\[
r_{n,\mathcal{L}(X_1)} =  \frac{\int_{\zeta  n^{\frac{1-2s}{\alpha}}}^\infty g_1(z_1) \dd z_1}{(1+\zeta n^{\frac{1-2s}{\alpha}})^{-\alpha} }
\]
be the normalizing constant of the tail probability of $Z_1$ relative to a Pareto (cf.~Definition~\ref{def:ourparetodef}). 
We make the following assumption. 
\begin{assumption}[Estimator for the normalizing constant]\label{assump:hatr1dim}
We assume the existence of an estimator $\hat{r} = \hat{r} \left(X_1^{(1)},\ldots,X_1^{(n)}\right)$ such that for any $\eps\in(0,1)$, there exists $C_\eps$ such that
\[
\liminf_{n\to\infty}
\inf_{\mathcal{L}(X_1)\in M^1}P\left(n^{s} \left|\hat{r}-r_{n,\mathcal{L}(X_1)}\right|\leq C_\eps\right) \geq 1-\eps.
\]
\end{assumption}
Assumption~\ref{assump:hatr1dim} postulates that with high probability, the estimator $\hat{r}$ estimates $r_{n,\mathcal{L}(X_1)}$ within $O(n^{-s})$ error uniformly over the model family. While $\hat{r}$ is not trivial to find in general, we show that given the existence of such an estimator, it is possible to estimate $\theta_1$ within $O(n^{-s})$ error as well. 

\begin{theorem}[One dimensional estimation]\label{thm:oracleonedim}
    Suppose $\frac{1}{2+\max\{1,\alpha\}}< s<\frac{1}{2}$. Let $\hat{r} = \hat{r} \left(X_1^{(1)},\ldots,X_1^{(n)}\right)$ be an estimator that satisfies Assumption~\ref{assump:hatr1dim}.
  Choosing $\tau_n= u\zeta n^{\frac{1-2s}{\alpha}}$, and let $\hat\theta_1$ solve the estimating equation
  \[
\frac{\sum_{i=1}^n1_{\{X_1^{(i)}>\tau_n\}}}{n} = \hat{r}(1+\frac{\tau_n}{\hat\theta_1})^{-\alpha}.
\]
Then, for any $\eps\in(0,1)$, there exists $\tilde C_\eps$, such that
    \[
    \liminf_{n\to\infty}
\inf_{\mathcal{L}(X_1)\in M^1}P\left( n^s\left|\hat\theta_1-\theta_1\right|\leq \tilde C_\eps \right) \geq 1-\eps.
    \]
\end{theorem}

\subsubsection{Generalizations for $d\geq2$}\label{sec:potentialmultidim}

Recall that we consider the case where $d = m$, and the models can be written as $X=AZ$ with $A$ being a square $m$-by-$m$ matrix. Let  $\theta_i= \|A_{\cdot i}\|_1$ be the length of the $i$-column of $A$ measured using a $\ell_1$-norm. It is useful to consider a normalized version of $A$, namely, the matrix $\bar{A}$ whose $i$-th column can be derived as
\[\bar{A}_{\cdot i} \Let \frac{A_{\cdot i} }{ \theta_i } \qquad \forall i.\]
In this way, the columns of $\bar{A}$ represent the direction of the columns of $A$. Then, the model can be reparametrized by the magnitude-direction information $(\theta, \bar A)$ as 
\[ X = \bar{A} \diag(\theta_1,\ldots, \theta_m) Z,\]
and $\diag(\theta_1,\ldots, \theta_m)$ denotes the diagonal matrix with diagonal elements $\theta_i$. The decomposition of $A$ into $\bar{A}$ and $\theta$ is useful because, as it will turn out, estimating $\bar{A}$ can be easier than estimating $\theta$. In fact, our proposed procedure first estimates $\bar{A}$ with an error of order $O(n^{-s})$, then uses a plug-in estimate of $\bar{A}$ to estimate $\theta$ also with an error of order $O(n^{-s})$. For technical reasons, we further restrict to the case $\frac{1}{2+\max\{1,\alpha\}}<s<\frac{1}{\alpha}$ in order to derive an estimate of $\bar{A}$ of the required error rate.

Similar to Section~\ref{sec:potential1dim}, we denote by
\[
r_{n,\mathcal{L}(X)}=\frac{\int_{z\in\mathbb{R}^m_+:\|z\|_1>\zeta n^{\frac{1-2s}{\alpha}}} g(z)\mathrm{d}z}{ \int_{z\in\mathbb{R}^m_+:\|z\|_1>\zeta n^{\frac{1-2s}{\alpha}}}\alpha^m\prod_{i=1}^m (1+z_i)^{-(\alpha+1)}\mathrm{d}z}
\]
the normalizing constant of the tail probability of $Z$ (that is associated with the law of $X$) relative to the product of Paretos (cf.~Definition~\ref{def:ourparetodef}). We have the following theorem.

\begin{theorem}[Multivariate estimation]\label{thm:bigthmpotentialestimator}
     Assume $\frac{1}{2+\max\{1,\alpha\}}< s< \min\{\frac{1}{2},\frac{1}{\alpha}\}$. Let $\tilde \tau_n = \Theta\left(\left(\frac{\log n }{n}\right)^{-\frac{1}{\alpha}}\right)$, and let $\tilde A$ be the matrix whose columns are the locations of the $m$ cluster centers identified with k-means clustering~\cite[Algorithm~14.1]{hastie2009elements} using $\mathbb{P}_n^{\tilde\tau_n}$.  Let $\hat{r} = \hat{r} \left(X^{(1)},\ldots,X^{(n)}\right)$ be an estimator such that for any $\eps\in(0,1)$, there exists $C_\eps$ such that
   
\begin{equation}\label{eq:oraclehatrwhole}
    \liminf_{n\to\infty}
\inf_{\mathcal{L}(X)\in M}P\left(n^{s} \left|\hat{r}-r_{n,\mathcal{L}(X)}\right|\leq C_\eps\right) \geq 1-\eps.
\end{equation}
Let $\tau_n= u\zeta n^{\frac{1-2s}{\alpha}}$, and for any $i=1,\ldots,m$, let $\hat\theta_i$ solve the estimating equation
  \[
\frac{\sum_{j=1}^n1_{\{\tilde X_i^{(j)}>\tau_n\}}}{n} = \hat{r}(1+\frac{\tau_n}{\hat\theta_i})^{-\alpha},
\]
where $\tilde X^{(j)} = \tilde {A}^{-1}X^{(j)}$. Denoting $\hat{A} = \tilde A\diag(\hat{\theta}_1,\ldots,\hat{\theta}_m)$, then for any $\eps\in(0,1)$, there exists $\tilde C_\eps$, such that
    \[
    \liminf_{n\to\infty}
\inf_{\mathcal{L}(X)\in M}P\left(\mathbb{W}_p^p(K_{\hat{A}},K_A)\leq \tilde C_\eps n^{-s}\right) \geq 1-\eps.
    \] 
\end{theorem}
Note that $K_{\hat{A}}$ is defined by equation~\eqref{eq:limitcondmeasure}, and the inverse $\tilde A^{-1}$ exists with high probability. 
We summarize the proposed procedure of Theorem~\ref{thm:bigthmpotentialestimator} in Algorithm~\ref{alg:pseudocode}.
Our proposed procedure first estimates the ``directions'' of the spectral measure $K_A$, by selecting a high threshold $\tilde\tau = \Theta\left(\left(\frac{\log n}{n}\right)^{-\frac{1}{\alpha}}\right)$ and employing the k-means clustering algorithm on the thresholded samples $\mathbb{P}_n^{\tilde\tau}$ to identify the locations of the cluster centers. With this threshold $\tilde\tau$, the number of thresholded samples would be $O_p(\log n)$. Nevertheless, this logarithmic number of samples in the tail suffices for the required degree of accuracy. The algorithm collects the cluster centers (as column-vectors) into a matrix $\tilde A$, where the ordering of the columns can be arbitrary.  The algorithm then uses $\tilde A^{-1}$ to transform each observation $X^{(j)},j=1,\ldots,n$, to make each dimension of the problem independent. Finally, the algorithm solves an estimating equation for each dimension of the problem based on the information of the tail probability of threshold exceedance for each dimension. 
\begin{algorithm}
\caption{Pseudocode for the potential minimax estimator}\label{alg:pseudocode}
\KwData{$n$~i.i.d.~samples $X^{(1)},\ldots,X^{(n)}$ from model $X=AZ$}
\KwResult{an estimator $\hat{K}$ for $K_A$}
$\tilde\tau \gets \Theta\left(\left(\frac{\log n}{n}\right)^{-\frac{1}{\alpha}}\right)$\Comment*[r]{Assume $\alpha$ is given}
$\mathbb{P}_n^{\tilde\tau} \gets \left\{\frac{X^{(i)}}{\|X^{(i)}\|_1}\in\aleph^1_+:\|X^{(i)}\|_1>\tilde\tau, 1\leq i\leq n\right\}$\;
$\tilde A \gets \text{cluster centers of }\textbf{k-means} (\mathbb{P}_n^{\tilde\tau})$\Comment*[r]{Assume $m=d$ is given}
\For{$j=1,\ldots,n$} {
    $\tilde X^{(j)} \gets \tilde A^{-1} X^{(j)}$\;
}
$\tau \gets \Theta\left(n^{\frac{1-2s}{\alpha}}\right)$\Comment*[r]{Assume $s$ is given}
\For{$i=1,\ldots,m$} {
   $\hat{\theta}_i\gets \textrm{ solution of }\frac{\sum_{j=1}^n1_{\{\tilde X_i^{(j)}>\tau\}}}{n} = \hat{r}(1+\frac{\tau}{\hat\theta_i})^{-\alpha}$\Comment*[r]{Assume $\hat r$ is given}
}
$\hat{A} \gets \tilde A\textrm{diag}(\hat{\theta}_1,\ldots,\hat{\theta}_m)$\;
$\hat{K} \gets K_{\hat{A}}$
\end{algorithm}

\section{Numerical Experiments}\label{sec:numerics}
We formulate a novel class of Pareto-like models under which we derive the statistical minimax estimation rate, elucidate scenarios where the conventional estimator using the empirical measures of thresholded samples is minimax optimal, and propose an alternative estimator that attains the minimax rate under minor additional assumptions when the conventional estimator is not minimax optimal. We now present some numerical experiments that empirically validate the rate we derived and offer further insights
into the choice of tuning parameters. 

In particular, we consider the scenario $d=m=2$, namely, the dimension of observations $X$ equals the dimension of latent factors $Z$. We consider a range of choices of $\alpha\in\{0.5,1,2\}$ and $s\in\{0.2,0.4\}$. The ground truth $A$ is chosen as the matrix
\[
A = \begin{pmatrix}
    1+n^{-s} & 0 \\
    0 & 1-n^{-s} 
\end{pmatrix},
\]
which appears as a worst-case example in the analysis of the minimax lower bound in the proof of Theorem~\ref{thm:lowerbound}. 

For any chosen sample size $n$, we generate $n$ i.i.d.~samples $X^{(i)} = A Z^{(i)},i=1,\ldots,n$. We generate $Z^{(i)}=(Z^{(i)}_1,Z^{(i)}_{2})$ according to the worst-case distribution that appears in the analysis of the minimax lower bound  in the proof of Theorem~\ref{thm:lowerbound} as follows:
\begin{enumerate}[label=(\roman*)]
        \item Generate $Z^{(i)}_1$ from the univariate distribution with density function 
        \[
        \alpha(1+n^{-s})\left(1+(1+n^{-s}) z_1\right)^{-(\alpha+1)} \quad \forall z_1 \geq 0.
        \]
        Independently, generate $Z^{(i)}_2$ from the univariate distribution with density function 
        \[
        \alpha(1-n^{-s})\left(1+(1-n^{-s}) z_2\right)^{-(\alpha+1)} \quad \forall z_2 \geq 0.
        \]
        \item If $\|Z^{(i)}\|_1< n^{\frac{1-2s}{\alpha}}$, terminate and exit. On the other hand, if $\|Z^{(i)}\|_1\geq n^{\frac{1-2s}{\alpha}}$, replace this sample by the following subroutine
        \begin{itemize}
            \item Generate a vector $\tilde Z^{(i)} = (\tilde Z^{(i)}_1,\tilde Z^{(i)}_2)$, where for $j\in\{1,2\}$, $\tilde Z^{(i)}_j$ is an independent sample from the univariate Pareto distribution with density function \[\alpha(1+z_j)^{-(\alpha+1)}, z_j\geq 0.\]
            \item While $\|\tilde Z^{(i)}\|_1< n^{\frac{1-2s}{\alpha}}$, discard this sample and regenerate according to the previous step, ..., until we obtain one sample  $\tilde Z^{(i)}$ such that $\|\tilde Z^{(i)}\|_1\geq n^{\frac{1-2s}{\alpha}}$. Use this sample to replace the original $Z^{(i)}$.
        \end{itemize}
\end{enumerate}
After this procedure to generate the $i$-th sample $Z^{(i)}$, we compute $X^{(i)} = AZ^{(i)}$. Given i.i.d.~observations $X^{(1)},\ldots,X^{(n)}$, the goal is to construct estimators for $K_A$, where $K_A$ is the mixture of Diracs given by equation~\eqref{eq:limitcondmeasure}.

In the experiments, we consider the following two estimators:
\begin{enumerate}[label=(\roman*)]
    \item The conventional estimator whose rate is given by Theorem~\ref{thm:lower}. Depending on the value of $s$, we choose the threshold $\tau$ as
    \[
    \tau = \bar\kappa n^{\frac{1}{\min\{2+\alpha,3\alpha\}}}\quad\textrm{or}\quad \tau = \bar\kappa n^{\frac{1-2s}{\alpha}},
    \]
    where $\bar\kappa$ is a tuning hyperparameter. In this design, we have ignored the poly-log factor in Theorem~\ref{thm:lower}~$(i)$. We run k-means clustering on $\mathbb{P}_n^\tau$ to identify the locations and the weights of $d$ cluster centers and denote the empirical measure of the weighted mixture of the cluster centers by $\hat{K}_{conv}$. We measure the estimation error by the 1-Wasserstein distance
    \begin{equation}\label{eq:wasserrconv}
    \mathbb{W}_1(\hat{K}_{conv},K_A).
    \end{equation}
    We employ the Python package \textit{POT}~\cite{flamary2021pot} to compute this Wasserstein distance numerically.
    \item The proposed new estimator depicted in Algorithm~\ref{alg:pseudocode}, whose rate is given by Theorem~\ref{thm:bigthmpotentialestimator}. It can be verified that the trivial estimator $\hat{r}=1$ satisfies the desired accuracy in equation~\eqref{eq:oraclehatrwhole} for the distribution at hand. Note that there are two hyperparameters for tuning
    \[
    \tilde\tau  = \tilde\kappa  \left(\frac{\log n}{n}\right)^{-\frac{1}{\alpha}} \quad\text{and}\quad  \tau = \kappa n^{\frac{1-2s}{\alpha}}.
    \]
    Let $\hat{K}_{new} = \hat{K}$, the output of Algorithm~\ref{alg:pseudocode}, we measure the estimation error by the 1-Wasserstein distance 
    \begin{equation}\label{eq:wasserrnew}
        \mathbb{W}_1(\hat{K}_{new},K_A).
    \end{equation}
    The package $POT$~\cite{flamary2021pot} is used to compute this Wasserstein distance.
\end{enumerate}

\subsection{Empirical Validation of the Conventional Estimator}\label{sec:empiconv}
We first validate the convergence rate of the conventional estimator given by Theorem~\ref{thm:lower}. We consider a range of choices of $\alpha\in\{0.5,1,2\}$ and $s\in\{0.2,0.4\}$. The choice $s=0.2$ would satisfy $s<\frac{1}{2+\max\{1,\alpha\}}$ for all the $\alpha$ considered and the choice $s=0.4$ would satisfy $\frac{1}{2+\max\{1,\alpha\}}< s < \frac{1}{2}$ for all the $\alpha$ considered.

First, we consider a representative choice $\alpha=2$ and $s=0.2$. The tuning of hyperparameter $\bar\kappa$ is crucial, as it controls the number of samples exceeding the threshold $\tau$, which then controls when and how the asymptotics dictated by Theorem~\ref{thm:lower} kicks in. Hence, we plot the relation between the number of thresholded samples and the total sample size $n$, as we increase $n$, and use the plots as diagnostic plots to select $\bar\kappa$. We trial with four representative choices of $\bar\kappa\in\{0.1,0.5,1,1.3\}$, and illustrate the relation in Figure~\ref{fig:barkappatune}. Theoretically, the logarithm of the number of thresholded samples has a linear relation with the logarithm of the total sample size. Empirically, we find that different choice of hyperparameter $\bar\kappa$ induces different bias and variance tradeoffs. In Figure~\ref{fig:barkappatune}(a), for $\bar\kappa=0.1$, we find a slight variance in the curve. However, the relation is not linear and hence exhibits a bias. In Figure~\ref{fig:barkappatune}~(b), for $\bar\kappa=0.5$, we find that variance in the curve increases. Meanwhile, the curve still departs from a linear relation. In Figure~\ref{fig:barkappatune}~(c), for $\bar\kappa=1$, we find that the variance further increases, and the bias is reduced. In Figure~\ref{fig:barkappatune}~(d), for $\bar\kappa=1.3$, we find that the variance is significant. Any choice of $\bar\kappa$ larger than $1.3$ causes the program to throw an error, as there are too few thresholded samples available when $n$ is small. 

For each choice of $\bar\kappa$, we further plot the estimation error~\eqref{eq:wasserrconv} versus the total sample size $n$, as we increase $n$. We illustrate the relations in Figure~\ref{fig:barkappatunevis}. Theoretically, the logarithm of the estimation error has a linear relation with the logarithm of the total sample size. Empirically, we find that different choice of hyperparameter $\bar\kappa$ again induces different bias and variance tradeoffs. We summarize the convergence rates under different $\bar\kappa$ in Table~\ref{tab:ratebarkappa}. We observe that for the choice $\bar\kappa=1$, which appears optimal from the diagnostic plots (namely, from Figure~\ref{fig:barkappatune}), the empirical convergence rate $n^{-0.207}$ closely follows the theoretical rate $n^{-s}$ (where $s=0.2$).

\begin{figure}[h]
    \centering
    \subfigure[$\bar\kappa=0.1$]{
	 \includegraphics[width=0.45\columnwidth]{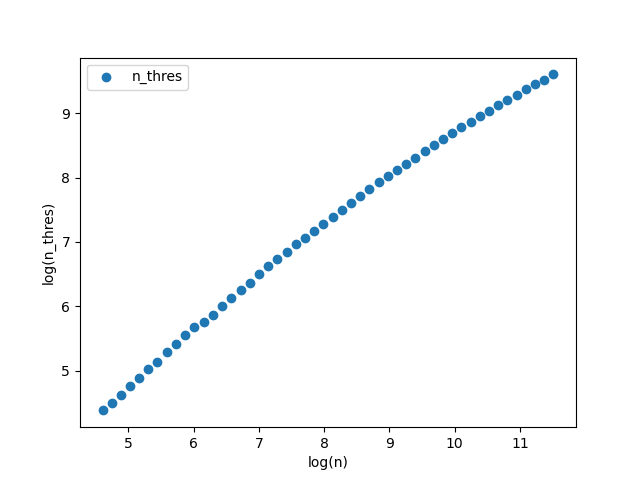}} \hspace{1mm}
	\subfigure[$\bar\kappa=0.5$]{
	\includegraphics[width=0.45\columnwidth]{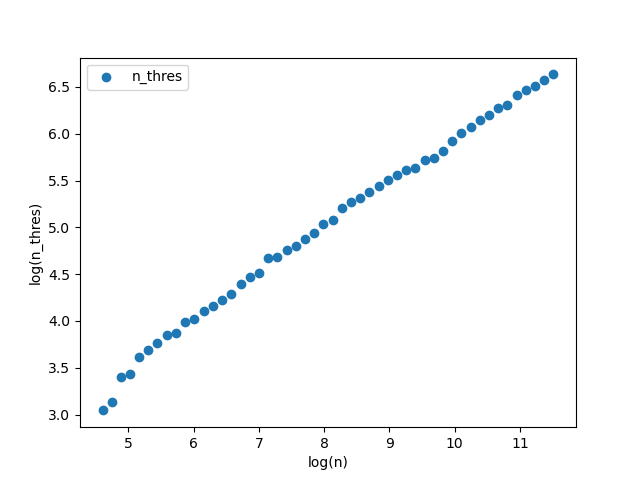}}
	    \hspace{1mm}
	\subfigure[$\bar\kappa=1$]{
		\includegraphics[width=0.45\columnwidth]{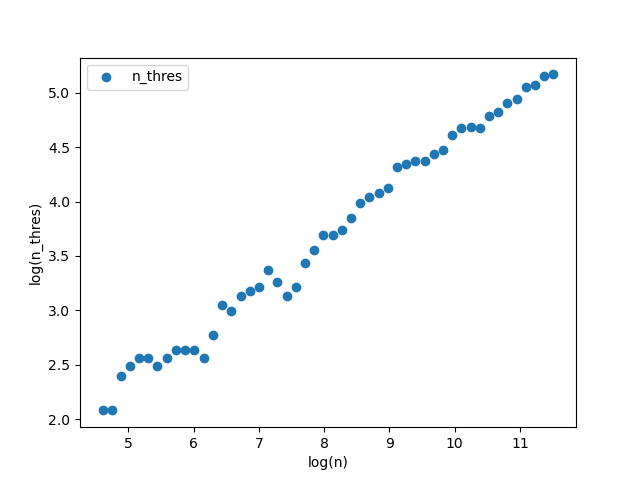}}
	\subfigure[$\bar\kappa=1.3$]{
		\includegraphics[width=0.45\columnwidth]{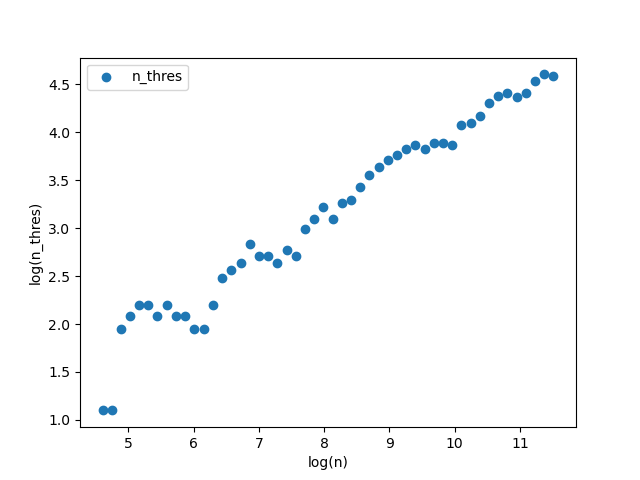}} 
	\caption{Logarithm of the number of samples exceeding the threshold $\tau$ versus $\log(n)$, where $n$ is the total sample size, following the conventional estimator, for each choice of $\bar\kappa$.}
    \label{fig:barkappatune}
\end{figure}

\begin{figure}[h]
    \centering
    \subfigure[$\bar\kappa=0.1$]{
	 \includegraphics[width=0.45\columnwidth]{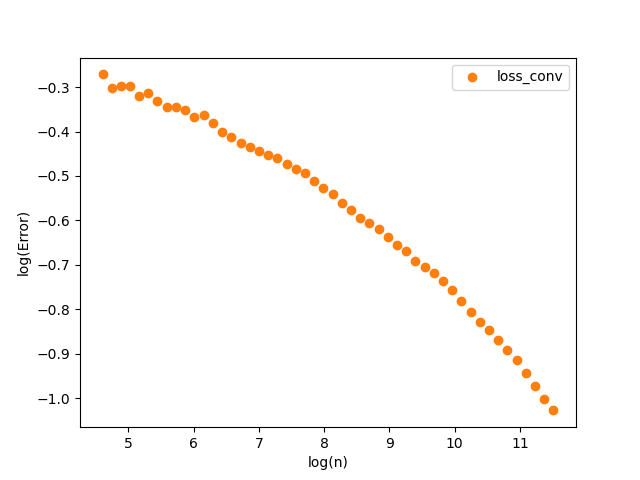}} \hspace{1mm}
	\subfigure[$\bar\kappa=0.5$]{
	\includegraphics[width=0.45\columnwidth]{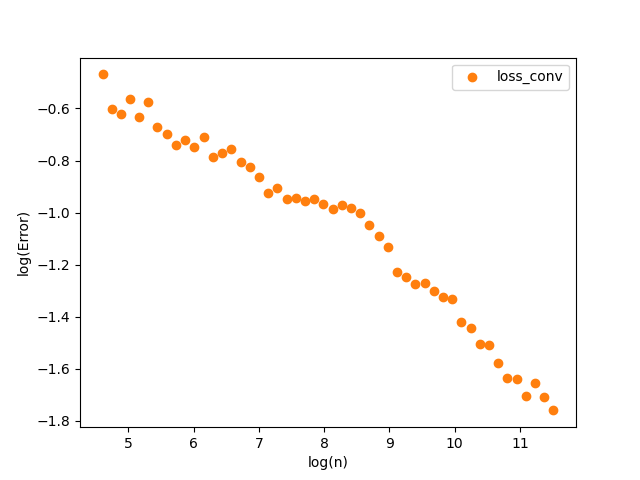}}
	    \hspace{1mm}
	\subfigure[$\bar\kappa=1$]{
		\includegraphics[width=0.45\columnwidth]{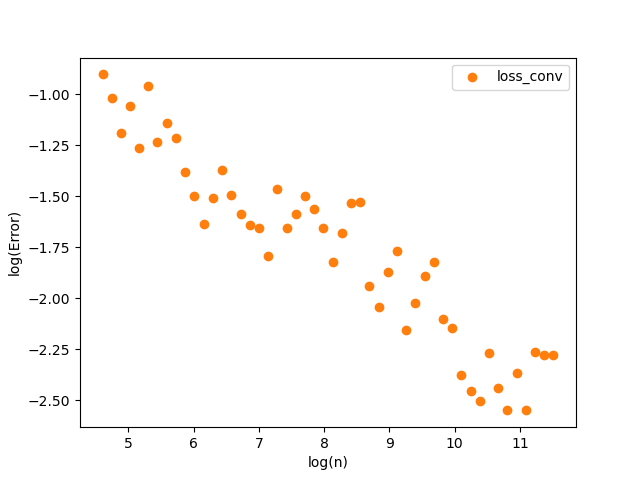}}
	\subfigure[$\bar\kappa=1.3$]{
		\includegraphics[width=0.45\columnwidth]{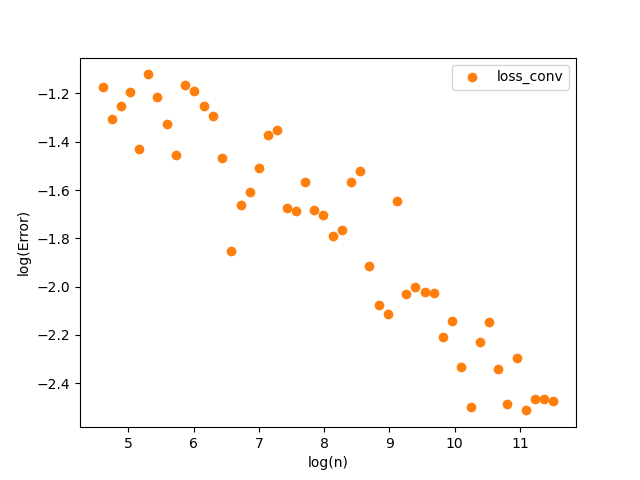}} 
	\caption{Empirical convergence of the conventional method versus $\log(n)$, where $n$ is the total sample size, for each choice of $\bar\kappa$.}
    \label{fig:barkappatunevis}
\end{figure}

\begin{table}[h!]
\centering
\begin{tabular}{||c|c|c|c|c||} 
\hline
$\bar\kappa$ & 0.1 & 0.5 & \textbf{1} & 1.3\\
\hline
rate & -0.104 & -0.171 & \textbf{-0.207} & -0.199\\
\hline
\end{tabular}
\caption{Empirical convergence rate for the conventional estimator for difference choices of hyperparameter $\bar\kappa$, when $\alpha=2$ and $s=0.2$.}
\label{tab:ratebarkappa}
\end{table}

Next, we extend this approach of tuning $\bar\kappa$ to a range of choices of $\alpha\in\{0.5,1,2\}$ and $s\in\{0.2,0.4\}$, and summarize the empirical convergence rate (under reasonably tuned $\bar\kappa$) in Table~\ref{tab:convsummary}. We observe that the empirical convergence rate closely follows the theoretical rate in Theorem~\ref{thm:lower}.

\begin{table}[h!]
\centering
\begin{tabular}{||c|c|c|c||} 
\hline
 & $\alpha=0.5$ & $\alpha=1$ & $\alpha=2$\\
\hline
$s=0.2$ & -0.233 (-0.200) & -0.215 (-0.200) & -0.207 (-0.200) \\
\hline
$s=0.4$ & -0.375 (-0.333) & -0.329 (-0.333)& -0.271 (-0.250)\\
\hline
\end{tabular}
\caption{Empirical convergence rate for the conventional estimator for different choices of $\alpha$ and $s$. Numbers in parentheses are the corresponding theoretical convergence rates.}
\label{tab:convsummary}
\end{table}

\subsection{Empirical Validation of the Proposed New Estimator}\label{sec:empinew}
We now discuss the selection of tuning parameters $\kappa$ and $\tilde\kappa$ via diagnostic plots and empirically compare the convergence rate of the new estimator against the conventional estimator.

First, we consider a representative choice $\alpha=2$ and $s=0.4$. We fix $\kappa=1$, and consider the tuning of $\tilde\kappa$. In Appendix~\ref{app:additionalsim}, we demonstrate that the tuning of $\kappa$ is more robust, and the choice of $\kappa=1$ is also reasonably optimal. Here, the number of thresholded samples would be logarithmic in $n$, which is used to identify the ``directions'' of the spectral measure $K_A$. Hence, we plot the relation between the number of thresholded samples and the total sample size $n$, as we increase $n$, and use the plots as diagnostic plots to select $\tilde\kappa$. We trial with four representative choices of $\tilde\kappa\in\{0.1,0.3,0.5,1\}$, and illustrate the relation in Figure~\ref{fig:tildekappatune}. Theoretically, the number of thresholded samples has a linear relation with the logarithm of the total sample size. Empirically, as in Section~\ref{sec:empiconv}, we find that different choice of hyperparameter $\tilde\kappa$ induces different bias and variance tradeoffs. In particular, from Figure~\ref{fig:tildekappatune}~(a) to Figure~\ref{fig:tildekappatune}~(d), we find that the variance of the curve increases while the bias of the curve decreases. Moreover, the choice $\tilde\kappa=0.3$ appears to strike an optimal bias-variance tradeoff. 

For each choice of $\tilde\kappa$, we further plot the estimation error~\eqref{eq:wasserrnew} versus the total sample size $n$, as we increase $n$. We illustrate the relations in Figure~\ref{fig:tildekappatunevis}, where we also add the error of the conventional estimator for comparison. Theoretically, the logarithm of the estimation error has a linear relation with the logarithm of the total sample size. Empirically, we find that different choice of hyperparameter $\tilde\kappa$ again induces different bias and variance tradeoffs. We summarize the convergence rates under different $\tilde\kappa$ in Table~\ref{tab:ratetildekappa}. We observe that for the choice $\tilde\kappa=0.3$, which appears optimal from the diagnostic plots (namely, from Figure~\ref{fig:tildekappatune}), the empirical convergence rate $n^{-0.408}$ closely follows the theoretical rate $n^{-s}$ (where $s=0.4$).

Finally, we extend this approach of tuning $\kappa$ and $\tilde\kappa$ to a range of choices of $\alpha\in\{0.5,1,2\}$ and $s=0.4$. Note that these combinations of $\alpha$ and $s$ satisfy the condition in Theorem~\ref{thm:bigthmpotentialestimator}. We summarize the empirical convergence rate and the comparison to the conventional estimator in Table~\ref{tab:newsummary}. We observe that the empirical convergence rates closely follow the corresponding theoretical rates, and the proposed new estimator consistently outperforms the conventional estimator.

\begin{figure}[h]
    \centering
    \subfigure[$\tilde\kappa=0.1$]{
	 \includegraphics[width=0.45\columnwidth]{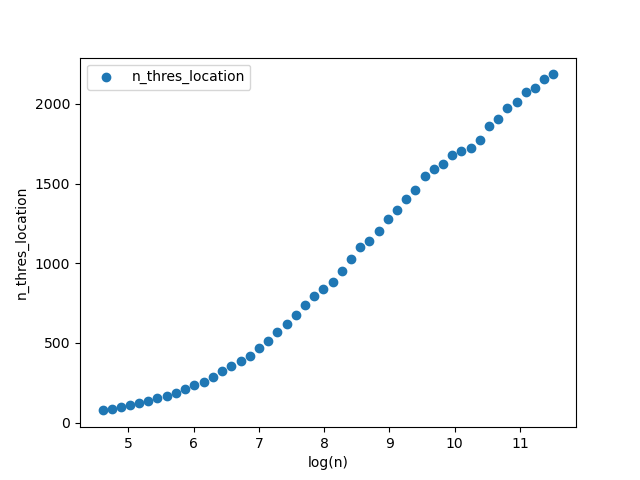}} \hspace{1mm}
	\subfigure[$\tilde\kappa=0.3$]{
	\includegraphics[width=0.45\columnwidth]{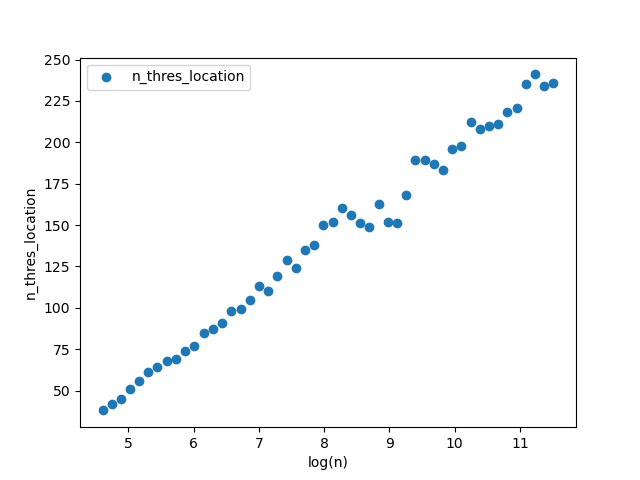}}
	    \hspace{1mm}
	\subfigure[$\tilde\kappa=0.5$]{
		\includegraphics[width=0.45\columnwidth]{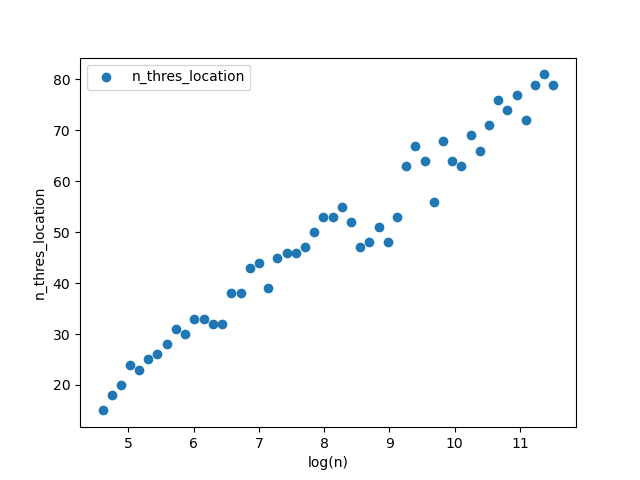}}
	\subfigure[$\tilde\kappa=1$]{
		\includegraphics[width=0.45\columnwidth]{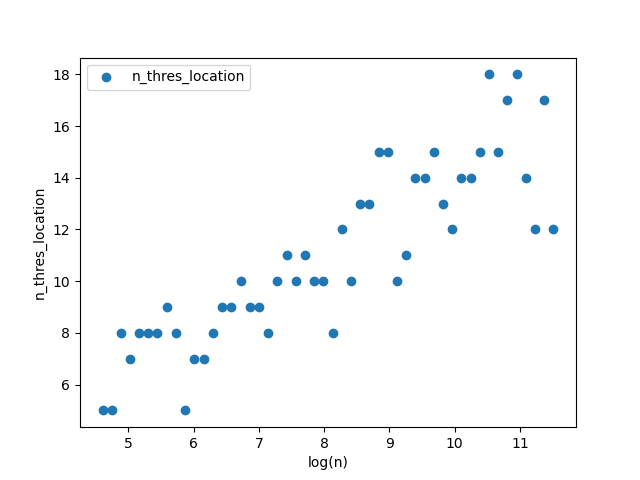}} 
	\caption{Number of samples exceeding the threshold $\tau$ versus $\log(n)$, where $n$ is the total sample size, for each choice of $\tilde\kappa$.}
    \label{fig:tildekappatune}
\end{figure}


\begin{figure}[h]
    \centering
    \subfigure[$\tilde\kappa=0.1$]{
	 \includegraphics[width=0.45\columnwidth]{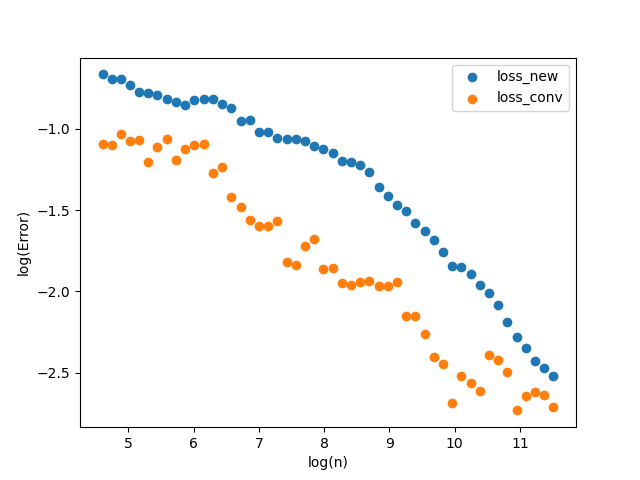}} \hspace{1mm}
	\subfigure[$\tilde\kappa=0.3$]{
	\includegraphics[width=0.45\columnwidth]{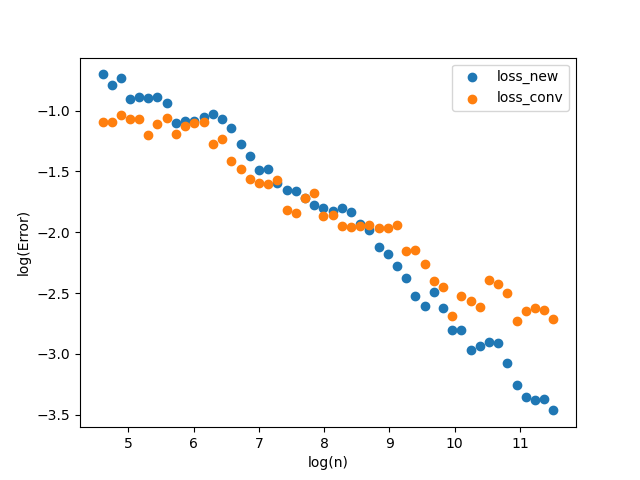}}
	    \hspace{1mm}
	\subfigure[$\tilde\kappa=0.5$]{
		\includegraphics[width=0.45\columnwidth]{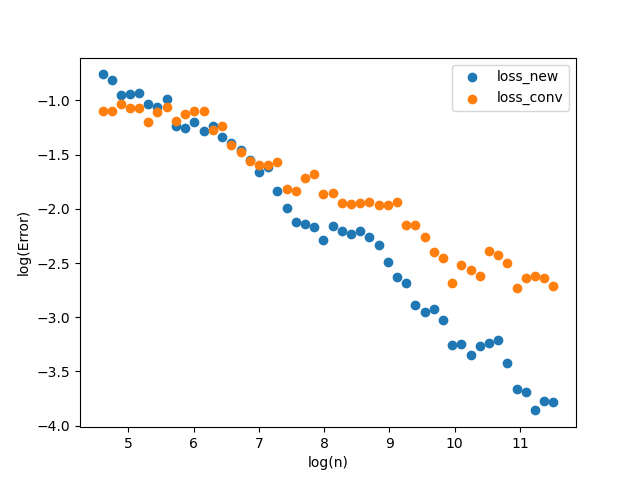}}
	\subfigure[$\tilde\kappa=1$]{
		\includegraphics[width=0.45\columnwidth]{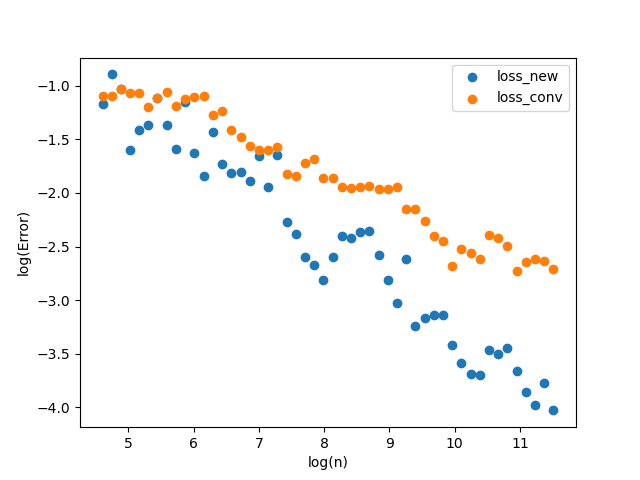}} 
	\caption{Empirical convergence of methods versus $\log(n)$, where $n$ is the total sample size, for each choice of $\tilde\kappa$.}
 \label{fig:tildekappatunevis}
\end{figure}

\begin{table}[h!]
\centering
\begin{tabular}{||c|c|c|c|c||} 
\hline
$\tilde\kappa$ & 0.1 & \textbf{0.3} & 0.5 & 1\\
\hline
rate & -0.257 & \textbf{-0.408} & -0.456 & -0.433\\
\hline
\end{tabular}
\caption{Convergence rate for the proposed estimator for difference choice of hyperparameter $\tilde\kappa$.}
\label{tab:ratetildekappa}
\end{table}

\begin{table}[h!]
\centering
\begin{tabular}{||c|c|c|c||} 
\hline
 & $\alpha=0.5$ & $\alpha=1$ & $\alpha=2$\\
\hline
new method rate & -0.419 (-0.400) & -0.499 (-0.400) & -0.408 (-0.400)\\
\hline
conv method rate & -0.375 (-0.333) & -0.329 (-0.333)& -0.271 (-0.250) \\
\hline
\end{tabular}
\caption{Empirical convergence rates for both methods for different choices of $\alpha$, with $s=0.4$ fixed. Numbers in parentheses are the corresponding theoretical convergence rates.}
\label{tab:newsummary}
\end{table}

\section{Discussions and Open Questions}
We focused on the information-theoretic limits of minimax estimation for the tail dependency structure in a class of non-parametric linear factors models. We established the minimax upper bound by analyzing the convergence rate of the estimator comprising empirical angular measure under the Wasserstein distance. The minimax lower bound was also established based on information theory. The comparison between the minimax upper and lower bounds 
reveals qualitative insights into when the straightforward Peak-over-Threshold method is optimal. We found that when the model has a large magnitude of deviation from the Paretos, quantified by $0<s\leq\frac{1}{2+\max\{1,\alpha\}}$, the straightforward estimator using the empirical thresholded samples is minimax optimal. However, when the model has only a small magnitude of deviation, quantified by $s>\frac{1}{2+\max\{1,\alpha\}}$, this estimator is minimax inefficient. We improved the conventional approach with a novel two-step estimator that decomposes the problem into one-dimensional estimating equations; this estimator attains the minimax lower bound rate. In our proposed estimator, a technical challenge arose in estimating the normalizing constant in the tail of $Z$, which we have chosen to leave unaddressed in this work. Hence, finding good estimators for this normalizing constant is a prospective future direction. The present paper also assumes knowledge of the regularly varying index $\alpha$ of each marginal of the latent factors and the knowledge of the deviation parameter $s$. Hence, it is a prospective direction to analyze minimax convergence rates without such knowledge and develop optimal corresponding \textit{adaptive} statistical procedures.


Multidimensional extreme value theory is an area that, over the years, has attracted considerable attention both in academia and industry. Not surprisingly, after this paper was posted on arXiv, in the last few months, additional scientific efforts were reported on topics closely related to this paper. In particular, the work of~\cite{boulin2024estimatingmaxstablerandomvectors} considers the problem of estimating the dependence structure (i.e., $A$), assuming a linear model as we do here. The estimator considered in \cite{boulin2024estimatingmaxstablerandomvectors} of threshold type, as the one studied in Theorem \ref{thm:lower}, is also studied in a high-dimensional setting and weakly dependent data, while under additional sparsity constraint of $A$. 

\bigskip
\bigskip
\bigskip
\textbf{Acknowledgments.} The material in this paper is based upon work supported by the Air Force Office of Scientific Research under award number FA9550-20-1-0397. J. Blanchet acknowledges additional support from NSF 1915967, 2118199, 2229012, 2312204.

\newpage
\appendix
\section{Appendix}

\subsection{Proofs of Results in Section~\ref{sec:setting}}\label{ap:settingproof}
\begin{proof}[Proof of Proposition~\ref{prop:relationAKA} (ii)] 
\textbf{Step 1: Analysis of the likelihood functions.} By the generalized version of change-of-variables formula~\cite[Theorem~3.11]{ref:evans2015measure}, the density of $X$ in $\mathbb{R}^d_+$ can be written as
\begin{equation}\label{eq:surfaceint}
\tilde g_{A}(x) = (JA)^{-1}\int_{z\in\mathbb{R}^m_+:~ Az = x} g(z)\dd\mathcal{H}^{m-d},
\end{equation}
where
\[g(z) =\alpha^m\prod_{i=1}^m (1+z_i)^{-(\alpha+1)},\]
and $\mathcal{H}^{m-d}$ is the $(m-d)$-dimensional Hausdorff measure in $\mathbb{R}^m$. Consider two models parametrized by $A$ and a perturbation of $A$ respectively,
\[X=AZ, \quad X= (A+hB)Z,\]
where $h\in\mathbb{R}$ and $B\in\mathbb{R}^{d\times m}$. Since $JA>0$, $A$ has a row rank equal to $d$. Thus there exists $D\in\mathbb{R}^{m\times m}$, such that
\[
AD = B.
\] 
Thus, we can write the second model as
\[
X = A(I + hD) Z = A\tilde Z,
\]
where $\tilde Z = (I+hD) Z$. By the generalized version of change of variables formula~\cite[Theorem~3.11]{ref:evans2015measure}, the density of $\tilde Z$ at $z\in\mathbb{R}^m_+$ is given by 
\[
\textrm{det}\left(I + hD\right)^{-1} g\left(\left(I + hD\right)^{-1}z\right).
\]
Thus, by a change of variables again, the density of the second model at $x$ is
\[
\tilde g_{A+hB}(x)=(JA)^{-1}\int_{z\in\mathbb{R}^m_+:~ Az = x} \textrm{det}\left(I + hD\right)^{-1} g\left(\left(I + hD\right)^{-1}z\right)\dd\mathcal{H}^{m-d}.
\]

\textbf{Step 2: Expansion of the log-likelihood.} Note that  $\textrm{det}\left(I + hD\right)^{-1}$ is a smooth function of $h$ in a neighborhood of $0$, and $g\left(\left(I + hD\right)^{-1}z\right)$ is smooth jointly in $h$ (in a neighborhood of $0$) and $z$. Since the region of integration $\{z\in\mathbb{R}^m_+:~ Az = x\}$ is compact, we have that $\tilde g_{A+hB}(x)$ is second-order continuously differentiable in $h$. We consider expanding $\log \tilde g_{A+hB}(x)$ around $\log\tilde g_A(x)$ for small values of $h$.
Note that
\[
\textrm{det}\left(I + hD\right)^{-1} = \prod_{i=1}^m ( 1+ h\lambda_i)^{-1} = \prod_{i=1}^m (1-h\lambda_i + h^2 \lambda_i^2 + O(h^3)),
\]
where $\lambda_i$'s are the eigenvalues of $D$, and
\begin{align*}
 g\left(\left(I + hD\right)^{-1}z\right) & = g(z) + \nabla g(z) \left(\left(I + hD\right)^{-1} - I\right) z \\
  & \quad + \frac{1}{2} z^\top \left(\left(I + hD\right)^{-1} - I\right)^\top \nabla g^2(z) \left(\left(I + hD\right)^{-1} - I\right) z + O(h^3),
\end{align*}
where we can further simplify
\[
\left(I + hD\right)^{-1} - I = \sum_{p=1}^\infty (-1)^ph^pD^p .
\]
The convergence is guaranteed for small $h$. To compute the Fisher information matrix, it suffices to compute the second moment of the first differential, when expanding $\log \tilde g_{A+hB}(x)$ around $\log\tilde g_A(x)$. We note that
\begin{align*}
     &\quad \log \tilde g_{A+hB}(x)- \log\tilde g_A(x) \\
     & = \log \int_{z\in\mathbb{R}^m_+:~ Az = x} \det\left(I + hD\right)^{-1} g\left(\left(I + hD\right)^{-1}z\right)\dd\mathcal{H}^{m-d} - \log \int_{z\in\mathbb{R}^m_+:~ Az = x} g(z)\dd\mathcal{H}^{m-d}\\
     & = \log \int_{z\in\mathbb{R}^m_+:~ Az = x} \left(\prod_{i=1}^m (1-h\lambda_i +  O(h^2))\right) \left(g(z) + \nabla g(z) \left(-hD\right) z + O(h^2)\right)\dd\mathcal{H}^{m-d} \\
     & \quad - \log \int_{z\in\mathbb{R}^m_+:~ Az = x} g(z)\dd\mathcal{H}^{m-d}\\
     & = \log \int_{z\in\mathbb{R}^m_+:~ Az = x} \left(g(z) - h \cdot \Tr(D)g(z)  - h \nabla g(z) Dz \right)\dd\mathcal{H}^{m-d} + O(h^2) \\ 
     &\quad -\log \int_{z\in\mathbb{R}^m_+:~ Az = x} g(z)\dd\mathcal{H}^{m-d}  \\
     & = \log \left(1 - h \frac{\int_{z\in\mathbb{R}^m_+:~ Az = x}  \left(\Tr(D)g(z)  +  \nabla g(z) Dz \right)\dd\mathcal{H}^{m-d}}{\int_{z\in\mathbb{R}^m_+:~ Az = x} g(z)\dd\mathcal{H}^{m-d}}\right) + O(h^2)\\
     & = - h \frac{\int_{z\in\mathbb{R}^m_+:~ Az = x}  \left(\Tr(D)g(z)  +  \nabla g(z) Dz \right)\dd\mathcal{H}^{m-d}}{\int_{z\in\mathbb{R}^m_+:~ Az = x} g(z)\dd\mathcal{H}^{m-d}} + O(h^2)\\
     & = - h \left(\Tr(D) + \frac{\int_{z\in\mathbb{R}^m_+:~ Az = x}    \nabla g(z) Dz \dd\mathcal{H}^{m-d}}{\int_{z\in\mathbb{R}^m_+:~ Az = x}g(z)\dd\mathcal{H}^{m-d}}\right)  + O(h^2).
\end{align*}

\textbf{Step 3: Choice of the degenerate direction.} Note that we can also compute
\[
\frac{\partial g}{\partial z_i} =  - g(z) (\alpha+1) (1+z_i)^{-1}.
\]
 If $A_{\cdot i} = A_{\cdot j}$ for some $i\neq j$, then
\[
\int_{z\in\mathbb{R}^m_+:~ Az = x} \frac{z_i}{1+z_j}g(z)\dd\mathcal{H}^{m-d} = \int_{z\in\mathbb{R}^m_+:~ Az = x} \frac{z_j}{1+z_i}g(z)\dd\mathcal{H}^{m-d}.
\]
Choosing $D$ as the matrix where all entries are zeros except $D_{ij} = -D_{ji}\neq0$, we have
\[
\Tr(D) + \frac{\int_{z\in\mathbb{R}^m_+:~ Az = x}    \nabla g(z) Dz \dd\mathcal{H}^{m-d}}{\int_{z\in\mathbb{R}^m_+:~ Az = x}g(z)\dd\mathcal{H}^{m-d}} = 0
\]
for all $x\in\mathbb{R}_+^d$. Therefore, we conclude that the Fisher information matrix of model~\eqref{eq:XAZ} is degenerate for the direction $B=AD$. This direction is non-zero by the assumption of Lemma~\ref{lem:KAconv} that $\|A_{\cdot i}\|_1>0$. 
\end{proof}

\subsection{Proofs of Results in Section~\ref{sec:upperbound}}\label{ap:upperboundproof}
\begin{proof}[Proof of Proposition~\ref{thm:bias}]
The proof follows the steps below.

\textbf{Step 1: Reduction to Wasserstein distance in the space of $Z$.} As in the proof of Lemma~\ref{lem:KAconv}, we can rewrite
\[
\mu_A^\tau = \mathcal{L}\left(\frac{X}{\|X\|_1}\Big|\|X\|_1>\tau\right)  =  \mathcal{L}\left(A\frac{ Z}{\|Z\|_a}\Big|\|Z\|_a>\tau\right).
\]
Similarly, we rewrite
\[
K_A = \sum_{i=1}^m\frac{\|A_{\cdot i}\|_1^\alpha}{\sum_{j=1}^m\|A_{\cdot j}\|_1^\alpha}\delta_{\{\frac{A_{\cdot i}}{\|A_{\cdot i}\|_1}\}} = \sum_{i=1}^m\frac{a_i^\alpha}{\sum_{j=1}^ma_j^\alpha}\delta_{\{\frac{A_{\cdot i}}{a_i}\}}.
\]
We first consider the case that $Z_i$ are~i.i.d and follow the Pareto distribution with index $\alpha$, so that the density function of $Z$ is $g(z) = \alpha^m \prod_{i=1}^m (1+z_i)^{-(\alpha+1)}~\forall z_i\geq0$. Let $b(t) = t^{\frac{1}{\alpha}}$.
Then, by theory of regular variation~\cite[Theorem~6.1]{ref:resnick2007heavy}, 
\[
tP\left(\frac{Z}{b(t)}\in[\mathbf{0},z]^c\right)\to \sum_{i=1}^m (z_i)^{-\alpha}\quad\text{as } t\to\infty,
\]
for all $z\in\mathbb{R}^m_+\backslash\{\mathbf{0}\}$. Hence, we have
\[
\mathcal{L}\left(\frac{Z}{\|Z\|_a}\Big|\|Z\|_a>b(t)\right) \Rightarrow \mathcal{L}\left(\sum_{i=1}^m \frac{a_i^\alpha}{\sum_{j=1}^m a_j^\alpha} \delta_{\{\frac{e_{\cdot i}}{a_i}\}}\right),
\]
where $e_{\cdot i}$ is the $i$-th element of the standard basis in $\mathbb{R}^m$. 
Considering the push-forward map from $\mathbb{R}^m$ to $\mathbb{R}^d$ induced by the matrix $A$, we see that
\[
\mu_A^\tau = A_{\#}\mathcal{L}\left(\frac{Z}{\|Z\|_a}\Big|\|Z\|_a>\tau\right)\text{ and }
K_A = A_{\#}  \mathcal{L}\left(\sum_{i=1}^m \frac{a_i^\alpha}{\sum_{j=1}^m a_j^\alpha} \delta_{\{\frac{e_{\cdot i}}{a_i}\}}\right).
\]
Thus, any coupling between 
\[
\mathcal{L}\left(\frac{Z}{\|Z\|_a}\Big|\|Z\|_a>\tau\right)\text{ and }
\mathcal{L}\left(\sum_{i=1}^m \frac{a_i^\alpha}{\sum_{j=1}^m a_j^\alpha} \delta_{\{\frac{e_{\cdot i}}{a_i}\}}\right)
\]
induces a coupling between $\mu^\tau_A$ and $K_A$. Note that $\|A(z-\tilde z)\|_1 \leq \|z-\tilde z\|_a,\forall z,\tilde z \in \mathbb{R}^m_+$. Thus we have
\[
\mathbb{W}_p(\mu^\tau_A,K_A)\leq \tilde{\mathbb{W}}_p\left(\mathcal{L}\left(\frac{Z}{\|Z\|_a}\Big|\|Z\|_a>\tau\right),\mathcal{L}\left(\sum_{i=1}^m \frac{a_i^\alpha}{\sum_{j=1}^m a_j^\alpha} \delta_{\{\frac{e_{\cdot i}}{a_i}\}}\right)\right),
\]
where $\tilde{\mathbb{W}}_p$ is the Wasserstein distance in $\mathbb{R}^m$ with the ground transportation cost defined by the  $\|\cdot\|_a$ norm.

\textbf{Step 2: Computation of the Wasserstein distance $\tilde{\mathbb{W}}_p.$}
Consider the polar coordinates transformation for any $z\in\mathbb{R}^m_+$
\[
z\mapsto (r,v),\quad \text{with }r = \|z\|_a = \sum_{i=1}^m a_i
|z_i| =\sum_{i=1}^m a_i
z_i  \text{ and } \upsilon_i = \frac{z_i}{r}, i=1,\ldots,m-1.
\]
Denote by $JP$  the Jacobian determinant of this polar coordinates transformation. The inverse transformation is 
\[
z_i = r\upsilon_i \quad \forall i=1,\ldots,m-1,\qquad \text{and} \quad z_m = \frac{r - \sum_{i=1}^{m-1} a_ir \upsilon_i}{a_m}.
\]
The Jacobian matrix of the inverse transformation can be computed as
\[
\begin{pmatrix}
\upsilon_1 & r & 0 &\cdots & 0 \\
\upsilon_2 & 0 & r &\cdots & 0 \\
\vdots & \vdots & \vdots & \vdots &\vdots\\
\upsilon_{m-1} & 0 & 0 &\cdots & r\\
\frac{1-\sum_{i=1}^{m-1}a_i\upsilon_i}{a_m} & -\frac{a_1r}{a_m} & -\frac{a_2r}{a_m} & \ldots & -\frac{a_{m-1}r}{a_m}
\end{pmatrix}.
\]
The absolute value of the determinant of the above matrix is (same as $(JP)^{-1}$)
\[ (JP)^{-1}   = \sum_{i=1}^{m-1} \upsilon_i  r^{m-1} \frac{a_i}{a_m} + \frac{1-\sum_{i=1}^{m-1} a_i\upsilon_i}{a_m} r^{m-1} = \frac{r^{m-1}}{a_m}. 
\]
The density of $\frac{Z}{b(t)}$ at $z$ is $b(t)^m\alpha^m \prod_{i=1}^m (1+b(t)z_i)^{-(\alpha+1)}$,
which, in terms of the polar coordinates $z=(r,\upsilon)$, can be expressed as
\[
g_t(r,\upsilon)= (JP)^{-1}b(t)^m\alpha^m \prod_{i=1}^{m-1} (1+ b(t)\cdot r\upsilon_i)^{-(\alpha+1)}\left(1+b(t) \cdot r\frac{1 -\sum_{i=1}^{m-1} a_i \upsilon_i}{a_m}\right)^{-(\alpha+1)}.
\]
After expanding and rearranging the terms, we have the density is
\begin{align*}
    g_t(r,\upsilon) & = \frac{r^{m-1}}{a_m}\alpha^m b(t)^m  (b(t)r)^{-m(\alpha+1)}\left(\frac{1}{(b(t)r)^m} + 
\frac{1}{(b(t)r)^{m-1}} HO^{(m-1)}\cdots + \frac{1}{b(t) r} HO^{(1)} + FO\right)^{-(\alpha+1)}\\
& = \frac{\alpha^m}{a_m} b(t)^{-m\alpha} r^{-m\alpha-1}\left(\frac{1}{(b(t)r)^m} + 
\frac{1}{(b(t)r)^{m-1}} HO^{(m-1)}\cdots + \frac{1}{b(t) r} HO^{(1)} + FO\right)^{-(\alpha+1)},
\end{align*}
where 
\begin{align*}
HO^{(m-1)} &= \sum_{i=1}^{m-1} \upsilon_i + \frac{1 -\sum_{i=1}^{m-1} a_i \upsilon_i}{a_m}, \\
\vdots \\
HO^{(1)} &= \prod_{i=1}^{m-1} \upsilon_i + \sum_{i=1}^{m-1}\frac{1 -\sum_{k=1}^{m-1} a_k \upsilon_k}{a_m}\prod_{1\leq j \leq m-1,j\neq i} \upsilon_j,
\end{align*}
and
\[
FO = \frac{1 -\sum_{k=1}^{m-1} a_k \upsilon_k}{a_m}\prod_{i=1}^{m-1} \upsilon_i.
\]
Note that the quantity \[\tilde{\mathbb{W}}_p^p\left(\mathcal{L}\left(\frac{Z}{\|Z\|_a}\Big|\|Z\|_a>b(t)\right),\mathcal{L}\left(\sum_{i=1}^m \frac{a_i^\alpha}{\sum_{j=1}^m a_j^\alpha} \delta_{\{\frac{e_{\cdot i}}{a_i}\}}\right)\right)\]
can be upper bounded by $m$ disjoint integrals, with an example integral of the form
\[
\int_{1<r<\infty}\int_{\upsilon:\upsilon_i\leq\eps~\forall i=1,\ldots,m-1} \left(\sum_{i=1}^{m-1}a_i\upsilon_i\right)^pg_t(r,\upsilon)\dd\upsilon \dd r,
\]
plus the error of the probability of the set $\{(r,\upsilon):r>1,\upsilon_i\leq\eps~\forall i\leq m-1\}$ compared to the quantity $\frac{a_m^\alpha}{t}$. In the following we analyze these two terms respectively.

\textbf{Step 3: Computation of the integral.}
Let $\eps\to0$ in a slower-than-polynomial rate, e.g., $\eps =\log(t)^{-1}$. We first upper bound the density $g_t(r,v)$ to simplify the integral. For any $\mathcal{I}\subseteq\{1,\ldots,m-1\}$, we denote by $\mathcal{R}_{\mathcal{I}}$ the rectangle 
\[\mathcal{R}_\mathcal{I}=\{\upsilon: 0\leq \upsilon_i \leq \frac{1}{b(t)r}~\forall i\in\mathcal{I},\textrm{ and }\frac{1}{b(t)r}\leq \upsilon_i\leq \eps ~\forall i\notin \mathcal{I}\}.\]
Then in the rectangle $\mathcal{R}_\mathcal{I}$
\[
 g_t(r,\upsilon) \leq \Theta\left(b(t)^{-m\alpha}r^{-m\alpha-1}\left((b(t) r)^{-\text{card}(\mathcal{I})}\prod_{i\notin\mathcal{I}}\upsilon_i\right)^{-(\alpha+1)}\right),
\]
where $\text{card}(\mathcal{I})$ denotes the cardinality of the set $\mathcal{I}$. Hence integrating over the rectangle $\mathcal{R}_\mathcal{I}$, we have
\begin{align*}
& \int_{\upsilon:\upsilon\in\mathcal{R}_\mathcal{I}} \left(\sum_{i=1}^{m-1}a_i\upsilon_i\right)g_t(r,\upsilon)\dd\upsilon\\
& \leq \Theta\left(b(t)^{-m\alpha + \text{card}(\mathcal{I})(\alpha+1)}r^{-m\alpha-1+\text{card}(\mathcal{I})(\alpha+1)}\right)\cdot\int_{\upsilon:\upsilon\in\mathcal{R}_\mathcal{I}} \left(\sum_{i=1}^{m-1}a_i\upsilon_i\right)\left(\prod_{i\notin\mathcal{I}}\upsilon_i\right)^{-(\alpha+1)}\dd\upsilon\\
& = \Theta\left(b(t)^{-m\alpha + \text{card}(\mathcal{I})(\alpha+1)}r^{-m\alpha-1+\text{card}(\mathcal{I})(\alpha+1)}\right)\cdot\sum_{i=1}^{m-1}a_i\int_{\upsilon:\upsilon\in\mathcal{R}_\mathcal{I}} \upsilon_i\left(\prod_{j\notin\mathcal{I}}\upsilon_j\right)^{-(\alpha+1)}\dd\upsilon.
\end{align*}
If $i\in\mathcal{I}$, then
\begin{align*}
 & \int_{\upsilon:\upsilon\in\mathcal{R}_\mathcal{I}} \upsilon_i\left(\prod_{j\notin\mathcal{I}}\upsilon_j\right)^{-(\alpha+1)}\dd\upsilon\\
 = &\int_0^{\Theta((b(t)r)^{-1})}\upsilon_i\dd\upsilon_i\cdot\prod_{j\in\mathcal{I},j\neq i} \int_0^{\Theta((b(t)r)^{-1})}1\dd\upsilon_j\cdot \prod_{j\notin \mathcal{I}} \int_{\Theta((b(t)r)^{-1})}^\eps\upsilon_j^{-(\alpha+1)}\dd\upsilon_j\\
 = & \Theta((b(t)r)^{-2})\cdot \Theta((b(t)r)^{-(\text{card}(\mathcal{I})-1)})\cdot \left(\Theta((b(t)r)^{\alpha}) - \Theta(\eps^{-\alpha})\right)^{m-1 - \text{card}(\mathcal{I})}\\
  = & \Theta\left((b(t)r)^{-2 - \text{card}(\mathcal{I}) + 1 +\alpha(m-1-\text{card}(\mathcal{I}))}\right)\\
  = & \Theta\left((b(t)r)^{\alpha m -\alpha-1 -(\alpha+1)\text{card}(\mathcal{I})}\right).
\end{align*}
On the other hand, if $i\notin\mathcal{I}$, then
\begin{align}
 & \int_{\upsilon:\upsilon\in\mathcal{R}_\mathcal{I}} \upsilon_i\left(\prod_{j\notin\mathcal{I}}\upsilon_j\right)^{-(\alpha+1)}\dd\upsilon\notag\\
 = & \int_{\Theta((b(t)r)^{-1})}^\eps\upsilon_i^{-\alpha}\dd\upsilon_i\cdot\prod_{j\in\mathcal{I}} \int_0^{\Theta((b(t)r)^{-1})}1\dd\upsilon_j\cdot \prod_{j\notin \mathcal{I},j\neq i} \int_{\Theta((b(t)r)^{-1})}^\eps\upsilon_j^{-(\alpha+1)}\dd\upsilon_j\notag\\
 = & \int_{\Theta((b(t)r)^{-1})}^\eps\upsilon_i^{-\alpha}\dd\upsilon_i\cdot\Theta((b(t)r)^{-\text{card}(\mathcal{I})})\cdot \left(\Theta((b(t)r)^{\alpha}) - \Theta(\eps^{-\alpha})\right)^{m-2 - \text{card}(\mathcal{I})}\label{eq:inotinmathcalIsimplify}.
 \end{align}
Since 
\[
\int_{\Theta((b(t)r)^{-1})}^\eps\upsilon_i^{-\alpha}\dd\upsilon_i = \begin{cases}
    \Theta(\log(b(t)r)) & \textrm{ if }\alpha = 1,\\
    \tilde\Theta(1) & \textrm{ if } 0<\alpha<1,\\
    \Theta((b(t)r)^{\alpha-1}) & \textrm{ if }\alpha > 1,
\end{cases}
\]
we can simplify equation~\eqref{eq:inotinmathcalIsimplify} as
\[
\eqref{eq:inotinmathcalIsimplify} = \begin{cases}
    \Theta(\log(b(t)r)) \cdot \Theta((b(t)r)^{\alpha m -2\alpha - (\alpha+1)\text{card}(\mathcal{I})}) & \textrm{ if }\alpha = 1,\\
   \tilde\Theta((b(t)r)^{\alpha m -2\alpha - (\alpha+1)\text{card}(\mathcal{I})}) & \textrm{ if } 0<\alpha<1,\\
    \Theta((b(t)r)^{\alpha m -\alpha - 1 - (\alpha+1)\text{card}(\mathcal{I})}) & \textrm{ if }\alpha > 1.
\end{cases}
\]
Hence, in summary, we can conclude that
\[
\int_{\upsilon:\upsilon\in\mathcal{R}_\mathcal{I}} \left(\sum_{i=1}^{m-1}a_i\upsilon_i\right)g_t(r,\upsilon)\dd\upsilon \leq \begin{cases}
      \Theta(\log(b(t)r)) \Theta(b(t)^{-2\alpha})\Theta(r^{-2\alpha-1}) & \textrm{ if }\alpha = 1,\\
      \tilde\Theta(b(t)^{-2\alpha})\Theta(r^{-2\alpha-1}) & \textrm{ if } 0<\alpha<1,\\
      \Theta(b(t)^{-\alpha-1})
    \Theta(r^{-\alpha-2}) & \textrm{ if }\alpha > 1.
\end{cases}
\]
Summing over all $\mathcal{I}$'s and $\mathcal{R}_\mathcal{I}$'s (there are only finitely many) we get
\[
\int_{\upsilon:\upsilon_i\leq\eps~\forall i} \left(\sum_{i=1}^{m-1}a_i\upsilon_i\right)g_t(r,\upsilon)\dd\upsilon \leq \begin{cases}
      \Theta(\log(b(t)r)) \Theta(b(t)^{-2\alpha})\Theta(r^{-2\alpha-1}) & \textrm{ if }\alpha = 1,\\
      \tilde\Theta(b(t)^{-2\alpha})\Theta(r^{-2\alpha-1}) & \textrm{ if } 0<\alpha<1,\\
      \Theta(b(t)^{-\alpha-1})
    \Theta(r^{-\alpha-2}) & \textrm{ if }\alpha > 1.
\end{cases}
\]
Note that the above statement holds uniformly for $r$ bounded away from $0$, therefore, integrating with respect to $r>1$, we have
\[
\int_{1<r<\infty}\int_{\upsilon:\upsilon_i\leq\eps~\forall i} \left(\sum_{i=1}^{m-1}a_i\upsilon_i\right)g_t(r,\upsilon)\dd\upsilon \dd r \leq \begin{cases}
      \Theta(\log(b(t))) \Theta(b(t)^{-2\alpha}) & \textrm{ if }\alpha = 1 , \\
      \tilde\Theta(b(t)^{-2\alpha})& \textrm{ if } 0<\alpha<1 ,\\
      \Theta(b(t)^{-\alpha-1})
     & \textrm{ if }\alpha > 1.
\end{cases}
\]
Note that since $\eps\to0$, we have that $\left(\sum_{i=1}^{m-1}a_i\upsilon_i\right)^p\leq \left(\sum_{i=1}^{m-1}a_i\upsilon_i\right)$ for $p\geq1$, and $t$ large enough. Then
\begin{equation}\label{eq:polarwassintegral}
\int_{1<r<\infty}\int_{\upsilon:\upsilon_i\leq\eps~\forall i} \left(\sum_{i=1}^{m-1}a_i\upsilon_i\right)^pg_t(r,\upsilon)\dd\upsilon \dd r \leq \begin{cases}
      \Theta(\log(b(t))) \Theta(b(t)^{-2\alpha}) & \textrm{ if }\alpha = 1, \\
      \tilde\Theta(b(t)^{-2\alpha})& \textrm{ if } 0<\alpha<1,\\
      \Theta(b(t)^{-\alpha-1})
     & \textrm{ if }\alpha > 1.
\end{cases}
\end{equation}

\textbf{Step 4: Analyzing the probability of $\{(r,\upsilon):r>1,\upsilon_i\leq\eps~\forall i\}$.} 
Next, we argue that the set $\{(r,\upsilon):r>1,\upsilon_i\leq\eps~\forall i\}$ receives the probability $\frac{a_m^\alpha}{t}(1+ \tilde\Theta(b(t)^{-\min\{1,\alpha\}}))$. To see this, we work with the original Cartesian coordinates system. Note that
\[ \left\{z\in\mathbb{R}^m_+: z = (r,\upsilon), r>1,\upsilon_i\leq\eps~\forall i\leq m-1\right\} =  C_m\cup \bar{C}_m,
\]
where the union above is a disjoint union with
\[
C_m=\left\{z\in\mathbb{R}^m_+:  z_i\leq \eps~\forall i\leq m-1,z_m> \frac{1-\sum_{i=1}^m a_iz_i}{a_m}\right\},
\]
and $\bar{C}_m$ is a subset of 
\[
\bar{C} = \left\{z\in\mathbb{R}_+^m: \|z\|_a>1,z_i,z_j > \eps~\exists i\neq j\right\}.
\]
First, we have that $P(\frac{Z}{b(t)}\in\bar{C}) = \tilde O(t^{-2})$. Also, we can compute that
\begin{align*}
   &  tP\left(\frac{Z}{b(t)}\in C_m\right)\\
   = & t\int_{z_1\leq\eps}.\hss.\hss.\int_{z_{m-1}\leq \eps} P\left(Z_m >b(t)\frac{1-\sum_{i=1}^m a_iz_i}{a_m}\right) ( b(t)\alpha)^{m-1}\prod_{i=1}^{m-1} (1 + b(t) z_i)^{-(\alpha+1)} \dd z_1.\hss.\hss. \dd z_{m-1}\\
= & t\int_{z_1\leq\eps}.\hss.\hss.\int_{z_{m-1}\leq \eps} \left(1+b(t)\frac{1-\sum_{i=1}^m a_iz_i}{a_m}\right)^{-\alpha} ( b(t)\alpha)^{m-1}\prod_{i=1}^{m-1} (1 + b(t) z_i)^{-(\alpha+1)} \dd z_1.\hss.\hss. \dd z_{m-1}\\
= & a_m^\alpha \int_{z_1\leq\eps}.\hss.\hss.\int_{z_{m-1}\leq \eps} \left(1 + a_mb(t)^{-1} - \sum_{i=1}^{m-1}a_iz_i\right)^{-\alpha} ( b(t)\alpha)^{m-1}\prod_{i=1}^{m-1} (1 + b(t) z_i)^{-(\alpha+1)} \dd z_1.\hss.\hss. \dd z_{m-1}.
\end{align*}
Based on the similar idea as before, we partition $\{(z_1,\ldots,z_{m-1}):z_i\leq \eps~\forall i\}$ into disjoint rectangles according to the split $z_i\leq b(t)^{-1}$ and $b(t)^{-1}\leq z_i\leq \eps$. In particular, consider $\mathcal{I}\subseteq\{1,\ldots,m-1\}$, and the rectangle 
\[\mathcal{R}_\mathcal{I}=\{z: 0\leq z_i \leq \frac{1}{b(t)}~\forall i\in\mathcal{I},\textrm{ and }\frac{1}{b(t)}\leq z_i\leq \eps ~\forall i\notin \mathcal{I}\}.\]
Then, recall that $\eps = \log(t)^{-1}$, by Taylor expansion, 
\begin{align*}
    &\int_{(z_1,\ldots,z_{m-1})\in\mathcal{R}_\mathcal{I}} \left(1 + a_mb(t)^{-1} - \sum_{i=1}^{m-1}a_iz_i\right)^{-\alpha} ( b(t)\alpha)^{m-1}\prod_{i=1}^{m-1} (1 + b(t) z_i)^{-(\alpha+1)} \dd z_1 \ldots \dd z_{m-1}\\
     = &  \int_{(z_1,\ldots,z_{m-1})\in\mathcal{R}_\mathcal{I}} \left(1 + \Theta(b(t)^{-1}) + \sum_{i=1}^{m-1}\Theta(z_i)\right) ( b(t)\alpha)^{m-1}\prod_{i=1}^{m-1} (1 + b(t) z_i)^{-(\alpha+1)} \dd z_1 \ldots \dd z_{m-1}\\
     = & \int_{(z_1,\ldots,z_{m-1})\in\mathcal{R}_\mathcal{I}}  ( b(t)\alpha)^{m-1}\prod_{i=1}^{m-1} (1 + b(t) z_i)^{-(\alpha+1)} \dd z_1 \ldots \dd z_{m-1}\\
     &+ \Theta(b(t)^{-1})\prod_{i\in\mathcal{I}}\int_{z_i\leq b(t)^{-1}}( b(t)\alpha)(1 + b(t) z_i)^{-(\alpha+1)} \dd z_i\prod_{i\notin\mathcal{I}}\int_{b(t)^{-1}\leq z_i\leq\eps}( b(t)\alpha)(1 + b(t) z_i)^{-(\alpha+1)} \dd z_i\\
     &+ \Theta(1)\sum_{i\notin\mathcal{I}}\Bigg(\int_{b(t)^{-1}\leq z_i\leq\eps}z_i( b(t)\alpha)(1 + b(t) z_i)^{-(\alpha+1)} \dd z_i \cdot \prod_{j\in\mathcal{I}}\int_{z_j\leq b(t)^{-1}}( b(t)\alpha)(1 + b(t) z_j)^{-(\alpha+1)} \dd z_j\\
     & \qquad\qquad\cdot\prod_{j\notin\mathcal{I}, j\neq i}\int_{b(t)^{-1}\leq z_j\leq\eps}( b(t)\alpha)(1 + b(t) z_j)^{-(\alpha+1)} \dd z_j\Bigg).
\end{align*}
We note that 
\begin{align*}
 &\sum_{\mathcal{R}_\mathcal{I}} \int_{(z_1,\ldots,z_{m-1})\in\mathcal{R}_\mathcal{I}}  ( b(t)\alpha)^{m-1}\prod_{i=1}^{m-1} (1 + b(t) z_i)^{-(\alpha+1)} \dd z_1 \ldots \dd z_{m-1}\\ = 
 &\int_{z_1,\ldots,z_{m-1}\leq\eps}  ( b(t)\alpha)^{m-1}\prod_{i=1}^{m-1} (1 + b(t) z_i)^{-(\alpha+1)} \dd z_1 \ldots \dd z_{m-1}\\
 = & (1- (1+b(t)\eps)^{-\alpha})^{m-1}\\
 = & 1 + \tilde \Theta(b(t)^{-\alpha}),
\end{align*}
and
\begin{align*}
 &
\Theta(b(t)^{-1})\prod_{i\in\mathcal{I}}\int_{z_i\leq b(t)^{-1}}( b(t)\alpha)(1 + b(t) z_i)^{-(\alpha+1)} \dd z_i\prod_{i\notin\mathcal{I}}\int_{b(t)^{-1}\leq z_i\leq\eps}( b(t)\alpha)(1 + b(t) z_i)^{-(\alpha+1)} \dd z_i\\
= &\Theta(b(t)^{-1}) (1 - (1 +\Theta(1))^{-\alpha})^{\text{card}(\mathcal{I})}( (1 +\Theta(1))^{-\alpha} - (1+b(t)\eps)^{-\alpha})^{m-1-\text{card}(\mathcal{I})}\\
 = & \Theta(b(t)^{-1}).
\end{align*}
Moreover, if $\text{card}(\mathcal{I})< m-1$ and $i\notin\mathcal{I}$, then
\begin{align*}
    & \int_{b(t)^{-1}\leq z_i\leq\eps}z_i( b(t)\alpha)(1 + b(t) z_i)^{-(\alpha+1)} \dd z_i\\
    = & \Theta(b(t)^{-1})\int_{1\leq \tilde z_i\leq b(t)\eps}\tilde z_i\alpha(1 + \tilde z_i)^{-(\alpha+1)} \dd\tilde z_i\qquad\textrm{ ``change of variables }\tilde z_i = b(t)z_i\textrm{''}\\
     = & \Theta(b(t)^{-1})\begin{cases}
     \Theta((b(t)\eps)^{1-\alpha}) & \textrm{ if }0<\alpha<1\\
    \Theta(\log(b(t)\eps) & \textrm{ if }\alpha=1\\
    \Theta(1) & \textrm{ if }\alpha>1
 \end{cases}\\
 = &  \begin{cases}
     \tilde\Theta((b(t)^{-\alpha}) & \textrm{ if }0<\alpha<1,\\
     \tilde\Theta(b(t)^{-1}) & \textrm{ if }\alpha=1,\\
     \Theta(b(t)^{-1}) & \textrm{ if }\alpha>1.
 \end{cases}
\end{align*}
Hence, in summary, after summing over all $\mathcal{I}$'s and $\mathcal{R}_{\mathcal{I}}$'s (there are only finitely many), we can conclude that
\begin{equation}\label{eq:cartesionproborder}
tP\left(\frac{Z}{b(t)}\in C_m\right) =  \begin{cases}
      a_m^\alpha(1+\tilde\Theta(b(t)^{-1})) & \textrm{ if }\alpha = 1,\\
      a_m^\alpha(1+\tilde\Theta(b(t)^{-\alpha})) & \textrm{ if } 0<\alpha<1. \\
      a_m^\alpha(1+\tilde\Theta(b(t)^{-1})) & \textrm{ if }\alpha > 1.
\end{cases}
\end{equation}
It is easy to see that  equations~\eqref{eq:polarwassintegral},~\eqref{eq:cartesionproborder}, and the fact that $P(\frac{Z}{b(t)}\in\bar{C}) = \tilde O(t^{-2})$,  are uniform for $A\in\mathcal{A}$, since $a_i$'s are uniformly bounded. We then have, by definition of Wasserstein distance, for any $p\geq1$, 
\[\sup_{A\in\mathcal{A}}\tilde{\mathbb{W}}_p^p\left(\mathcal{L}\left(\frac{Z}{\|Z\|_a}\Big|\|Z\|_a>b(t)\right),\mathcal{L}\left(\sum_{i=1}^m \frac{a_i^\alpha}{\sum_{j=1}^m a_j^\alpha} \delta_{\{\frac{e_{\cdot i}}{a_i}\}}\right)\right)\leq \tilde\Theta(b(t)^{-\min\{1,\alpha\}}).\]
Thus letting $b(t) = \tau$ (i.e., $t = \tau^{\alpha}$), we obtain
\[\sup_{A\in\mathcal{A}}\tilde{\mathbb{W}}_p^p\left(\mathcal{L}\left(\frac{Z}{\|Z\|_a}\Big|\|Z\|_a>\tau\right),\mathcal{L}\left(\sum_{i=1}^m \frac{a_i^\alpha}{\sum_{j=1}^m a_j^\alpha} \delta_{\{\frac{e_{\cdot i}}{a_i}\}}\right)\right)\leq \tilde\Theta(\tau^{-\min\{1,\alpha\}}).\]

\textbf{Step 5: Bounding the error when $Z$ is not Pareto.} Next, we consider the case that $Z = (Z_1,\ldots,Z_m)$ has a joint density that satisfies
\[
\left|\frac{g(z) - \alpha^m\prod_{i=1}^m (1+z_i)^{-(\alpha+1)}}{\alpha^m\prod_{i=1}^m (1+z_i)^{-(\alpha+1)}}\right|\leq \xi k^{-s}~\forall z \textrm{ and } g(z) \propto\prod_{i=1}^m (1+z_i)^{-(\alpha+1)} \textrm{ if } \|z\|_1 > \zeta k^{\frac{1-2s}{\alpha}},
\]
and let $\tilde Z = (\tilde Z_1,\ldots,\tilde Z_m)$ be~i.i.d.~Pareto with index $\alpha$. 
If $\tau > u \zeta k^{\frac{1-2s}{\alpha}}$, then $\|z\|_a>\tau$ implies that $\|z\|_1> \zeta k^{\frac{1-2s}{\alpha}}$, then 
\[
 \textrm{TV}\left(
\mathcal{L}\left(\frac{Z}{\|Z\|_a}\Big|\|Z\|_a>\tau\right),
\mathcal{L}\left(\frac{\tilde Z}{\|\tilde Z\|_a}\Big|\|\tilde Z\|_a>\tau\right)\right) = 0.
\]
Otherwise $\tau\leq u \zeta k^{\frac{1-2s}{\alpha}}$ , and we have
\begin{align*}
& \textrm{TV}\left(
\mathcal{L}\left(\frac{Z}{\|Z\|_a}\Big|\|Z\|_a>\tau\right),
\mathcal{L}\left(\frac{\tilde Z}{\|\tilde Z\|_a}\Big|\|\tilde Z\|_a>\tau\right)\right)\\
& \leq \sup_{B\subseteq\mathbb{R}_+^m: ~\|z\|_a>\tau~\forall z\in B} \left|\frac{P(Z\in B)}{P(\|Z\|_a>\tau)} - \frac{P(\tilde Z\in B)}{P(\|\tilde Z\|_a>\tau)}\right|\\
& \leq  \sup_{B\subseteq\mathbb{R}_+^m: ~\|z\|_a>\tau~\forall z\in B} \left|\frac{P(\tilde Z\in B)}{P(\|Z\|_a>\tau)} - \frac{P(\tilde Z\in B)}{P(\|\tilde Z\|_a>\tau)} \right| + \frac{1}{P(\|Z\|_a>\tau)} |P(Z\in B) - P(\tilde Z\in B)|\\
&\leq \left|\frac{P(\|\tilde Z\|_a>\tau)}{P(\| Z\|_a>\tau)} - 1\right| + \frac{1}{P(\|Z\|_a>\tau)}  \sup_{B\subseteq\mathbb{R}_+^m:~ \|z\|_a>\tau~\forall z\in B} |P(Z\in B) - P(\tilde Z\in B)|.
\end{align*}
Then
\[
 \textrm{TV}\left(
\mathcal{L}\left(\frac{Z}{\|Z\|_a}\Big|\|Z\|_a>\tau\right),
\mathcal{L}\left(\frac{\tilde Z}{\|\tilde Z\|_a}\Big|\|\tilde Z\|_a>\tau\right)\right)\leq 2\xi  k^{-s}\leq 2\xi \left(\frac{\tau}{u\zeta}\right)^{-\frac{\alpha s}{1-2s}}.
\]
Finally, note that the total variation distance TV upper bounds $\tilde{\mathbb{W}}_p^p$, since the $\|\cdot\|_a$-simplex $\aleph_+^a$ is compact (in particular, has a finite diameter). Hence, in summary, we conclude that
\[
\sup_{\mathcal{L}(X)\in M}\mathbb{W}_p^p(\mu_A^\tau,K_A) \leq \tilde O\left(\tau^{-\min\{1,\alpha\}}\right) + O\left(\tau^{-\frac{\alpha s}{1-2s}}\right).
\]
This completes the proof
\end{proof}

\begin{proof}[Proof of Proposition~\ref{thm:limitingmeasureconcentrationbound}]
 By Kantorovich duality~\cite[Theorem~5.10]{ref:villani2008optimal}, given any positive value of $n_\tau$, 
 \[
 \mathbb{W}_p^p(\mathbb{K}_A^{n_\tau},K_A) = \sup_{f\in\mathcal{F}} \int f \dd \mathbb{K}_A^{n_\tau} - \int f^c \dd K_A,
 \]
where $\mathcal{F}$ denotes the collection of integrable functions on the (discrete) support of $K_A$ (denoted by $\textrm{supp}(K_A)$), and 
\[
f^c(\omega) = \sup_{\omega'\in \textrm{supp}(K_A)} \left(f(\omega') - \|\omega'-\omega\|_1^p\right)
\]
is the $c$-transform of $f$.
Without loss of generality, for all $f\in\mathcal{F}$, we require $0\leq f\leq \text{diam}_1(\aleph_+^1)^p$, and note that $\text{diam}_1(\aleph_+^1)= 2$ is the diameter of $\aleph_+^1$ measured by $\ell_1$-norm. Since $f^c\geq f$, we have
\[
\mathbb{W}_p^p(\mathbb{K}_A^{n_\tau},K_A)\leq \sup_{f\in\mathcal{F}} \int f \dd \mathbb{K}_A^{n_\tau} - \int f \dd K_A =  \sup_{f\in\tilde{\mathcal{F}}} \int f \dd\mathbb{K}_A^{n_\tau} - \int f \dd K_A,
\]
where $\tilde{\mathcal{F}}\subset\mathcal{F}$ consists of functions taking value only in $\{0,2^p\}$. By the symmetrization lemma~\cite[Proposition~4.11]{ref:wainwright2019high}, 
\[
  E_{n_\tau}\left[\sup_{f\in\tilde{\mathcal{F}}} \int f \dd\mathbb{K}_A^{n_\tau} - \int f \dd K_A\right]\leq 2 E_{n_\tau}\left[\sup_{f\in\tilde{\mathcal{F}}}\left|\frac{1}{n_\tau}\sum_{i=1}^{n_\tau} \eps^{(i)} f(Y^{(i)})\right|\right],
\]
 where $\eps^{(i)}\in\{-1,1\}$ are independent random signs. Note that the cardinality of $\tilde{\mathcal{F}}$ is $2^m$, hence by Massart's finite class bound~\cite{ref:Massart2000some}, 
 \[
 E_{n_\tau}\left[\sup_{f\in\tilde{\mathcal{F}}}\left|\frac{1}{n_\tau}\sum_{i=1}^{n_\tau} \eps^{(i)} f(Y^{(i)})\right|\right]\leq \sqrt{\frac{2^{2p+1} m\cdot \textrm{log}(2)}{n_\tau}}.
 \]
 This completes the proof. 
\end{proof}

\begin{proof}[Proof of Proposition~\ref{thm:variance}]
\textbf{Step 1: Combining the bias and variance analyses.} Note that $\text{diam}_1(\aleph_+^1)^p = 2^p$, where $\text{diam}_1(\aleph_+^1)$ is the diameter of $\aleph_+^1$ measured by $\ell_1$-norm,  hence by~\cite[Proposition~20]{ref:weed2019sharp}, for any $n\ge1$ and any $t>0$, denoting $P_{n_\tau}$ as the conditional probability, we have
\[
P_{n_\tau}(\mathbb{W}_p^p(\mathbb{P}_n^{\tau},K_A) - E_{n_\tau}[\mathbb{W}_p^p(\mathbb{P}_n^{\tau},K_A)] \geq t)  \leq \exp\left(-2^{p+1}n_\tau t^2\right).
\]
 Thus with conditional probability at least $1-\eps_1$, we have
\[
\mathbb{W}_p^p(\mathbb{P}_n^{\tau},K_A) \leq E_{n_\tau}[\mathbb{W}_p^p(\mathbb{P}_n^{\tau},K_A)]  + \sqrt{ \frac{\log\frac{1}{\eps_1}}{2^{p+1}n_\tau}}.
\]
Trivially, this statement is uniform over model family~\eqref{eq:minimaxmodel}.
Recall that for all $n_\tau>0$,
\[
E_{n_\tau}[\mathbb{W}_p^p(\mathbb{P}_n^{\tau},K_A)]\leq 2^{p-1}\left( E_{n_\tau}[\mathbb{W}_p^p(\mathbb{P}_n^{\tau},\mathbb{K}_A^{n_\tau})] +E_{n_\tau}[ \mathbb{W}_p^p(\mathbb{K}_A^{n_\tau},K_A)]\right),
\]
and
\[
E_{n_\tau}[\mathbb{W}_p^p(\mathbb{P}_n^{\tau},\mathbb{K}_A^{n_\tau})]  = \mathbb{W}_p^p(\mu_A^{\tau},K_A).
\]
Also by Proposition~\ref{thm:limitingmeasureconcentrationbound}
\[
\sup_{\mathcal{L}(X)\in M} E_{n_\tau}[\mathbb{W}_p^p(\mathbb{K}_A^{n_\tau},K_A)] \leq  \sqrt{\frac{2^{2p+3} m\cdot \log(2)}{n_\tau}}.
\]

\textbf{Step 2: Concentration of $n_\tau$ around $nP(\|X\|_1>\tau)$.} Consider the random vector $X$, such that $\mathcal{L}(X)\in M$, and 
\begin{equation}\label{eq:smallmodelisminal}
P(\|X\|_1 >\tau) = \min_{\mathcal{L}(\tilde X)\in M} P(\|\tilde X\|_1 >\tau),\quad \textrm{ for all large }\tau.
\end{equation}
Such $X$ exists, for example, by considering $X=AZ$, where $\|A_{\cdot i}\|_1 = l~\forall i$ and $Z$ having the proportionality constant $(1-\xi)$ relative to the exact Pareto density in the region $\{z:\|z\|_1>\zeta\}$. By the definition of $\tilde M_k$ in~\eqref{eq:minimaxmodel}, we also have that
\begin{equation}\label{eq:smallmodelisalsomax}
    P(\|X\|_1 >\tau)\geq\Theta(\max_{\mathcal{L}(\tilde X)\in M} P(\|\tilde X\|_1 >\tau)),\quad \tau\to\infty.
\end{equation}
We note that $\|X\|_1$ is regularly-varing with index $\alpha$. Denote by $B(t)$ the quantile function of the distribution of $\|X\|_1$, such that
\[
P(\|X\|_1>B(t)) \simeq \frac{1}{t}~\textrm{ as }t\to\infty.
\]
According to~\cite[Theorem 9.1]{ref:resnick2007heavy}, after some algebra, we have that for any $v = v(n)$ such that $1\leq v \leq n, v = o(n)$ and $v\to\infty$, 
\begin{equation}\label{eq:processconverge}
\sqrt{\frac{nP(\|X\|_1>B(n/v)y^{-\frac{1}{\alpha}})}{v}}\left(\frac{\sum_{i=1}^n1_{\{\|X^{(i)}\|_1>B(n/v)y^{-\frac{1}{\alpha}}\}}}{\sqrt{nP(\|X\|_1>B(n/v)y^{-\frac{1}{\alpha}})}} - \sqrt{nP(\|X\|_1>B(n/v)y^{-\frac{1}{\alpha}})} \right)\Rightarrow W
\end{equation}
in $D[0,\infty)$ (i.e., the Skorokhod topology), where $W= (W(y):y\geq0)$ is the Brownian motion on $[0,\infty)$. We claim that there exists a particular sequence $v = v(n)$ such that $1\leq v\leq n, v =o(n), v\to\infty$ and 
\[
B(\frac{n}{v+1})\leq \tau_n< B(\frac{n}{v})\quad\forall n\textrm{~sufficiently large}.
\]
Note that $nP(\|X\|_1>\tau_n)\to\infty$,  $P(\|X\|_1>B(n))\simeq \frac{1}{n}$ and $B(\cdot)$ is monotone, thus we have that $\tau_n< B(n)$. Moreover since $\tau_n\to\infty$, we have that $\tau_n > B(1)$. Thus there exists $1\leq v\leq n$ such that \[
B(\frac{n}{v+1})\leq \tau_n< B(\frac{n}{v})\quad\forall n.
\]
Suppose that $v\to\infty$ does not hold, then $\frac{n}{v+1}=\Theta(n)$, hence \[
P(\|X\|_1>\tau_n) \leq P(\|X\|_1>B(\frac{n}{v+1}))\simeq \frac{1}{n},
\]
a contradiction! Suppose that $v=o(n)$ does not hold, then there exists a constant $U>0$ such that $\frac{n}{v} < U$ for all large $n$, hence $\tau_n < B(U)$, another contradiction! Thus, our claim is verified. Now let 
\begin{equation}\label{eq:yform}
y^{-\frac{1}{\alpha}}  = \frac{\tau_n}{B(\frac{n}{v})},
\end{equation}
we have 
\[
\frac{B(\frac{n}{v+1})}{B(\frac{n}{v})}\leq y^{-\frac{1}{\alpha}}\leq 1.
\]
 Since the tail of $\|X\|_1$ is regularly varying, by inversion, $B(\cdot)$ is also regularly varying~\cite[Proposition~2.6(v)]{ref:resnick2007heavy}. Hence, by the uniform convergence theorem~\cite[Proposition~2.4]{ref:resnick2007heavy}.
\[
\lim_{n\to\infty}\frac{B(\frac{n}{v+1})}{B(\frac{n}{v})} =   \lim_{n\to\infty}\frac{B(\frac{v}{v+1}\frac{n}{v})}{B(\frac{n}{v})} = 1.
\]
Thus we have $\lim_{n\to\infty} y^{-\frac{1}{\alpha}}\to 1$.

Choosing $y$ as in equation~\eqref{eq:yform} and plug-in $y$ in the process convergence~\eqref{eq:processconverge}, by continuous mapping theorem~\cite[Theorem 3.1]{ref:resnick2007heavy}, we conclude that
\[
\sqrt{\frac{nP(\|X\|_1>\tau_n)}{v}}\left(\frac{n_{\tau_n}}{\sqrt{nP(\|X\|_1>\tau_n)}} - \sqrt{nP(\|X\|_1>\tau_n)} \right)\Rightarrow \mathcal{N}(0,1).
\]
Also, note that
\[
P(\|X\|_1> B(\frac{n}{v}))\leq P(\|X\|_1>\tau_n)\leq P(\|X\|_1 > B(\frac{n}{v+1})),
\]
thus
\[
\sqrt{\frac{nP(\|X\|_1>\tau_n)}{v}}\to1.
\]
Thus
\[
\frac{n_{\tau_n}}{\sqrt{nP(\|X\|_1>\tau_n)}} - \sqrt{nP(\|X\|_1>\tau_n)} \Rightarrow \mathcal{N}(0,1).
\]
Hence, given $\eps_2$, there exists a quantile of standard normal $z_{\eps_2}$, such that
\[
\lim_{n\to\infty}P\left(\frac{n_{\tau_n}}{\sqrt{nP(\|X\|_1>\tau_n)}} - \sqrt{nP(\|X\|_1>\tau_n)}\geq - z_{\eps_2}\right)\geq 1-\eps_2.
\]
In other words, with asymptotic probability at least $1-\eps_2$, we have
\begin{equation}\label{eq:asymptoticevent}
\frac{n_{\tau_n}}{nP(\|X\|_1>\tau_n)} \geq 1  - \frac{z_{\eps_2}}{\sqrt{nP(\|X\|_1>\tau_n)}}\geq \frac{1}{2} \quad\textrm{ for all large }n.
\end{equation}
Now consider any arbitrary model in the family~\eqref{eq:minimaxmodel}, denote the corresponding random element by $\tilde X$, and the number of excesses over $\tau_n$ by $\tilde n_{\tau_n}$. Then, with at least the same probability of the event~\eqref{eq:asymptoticevent}, it holds that
\[
\tilde n_{\tau_n}\geq \frac{1}{2} nP(\|X\|_1>\tau_n)\geq \Theta(nP(\|\tilde X\|_1>\tau_n)),
\]
where we used the inequalities~\eqref{eq:smallmodelisminal} and~\eqref{eq:smallmodelisalsomax}.
Finally, with an asymptotic probability of at least 
$(1-\eps_1)(1-\eps_2)$, uniformly over $\mathcal{L}(\tilde X)\in M$, we have
\begin{align*}
\mathbb{W}_p^p(\mathbb{P}_n^{\tau_n},K_A)  & \leq 2^{p-1}\left( \mathbb{W}_p^p(\mu_A^{\tau_n},K_A) +\sqrt{\frac{2^{2p+3} m\cdot \textrm{log}(2)}{\Theta(nP(\|\tilde X\|_1>\tau_n))}}\right)\\
& +   \frac{1}{\sqrt{\Theta(nP(\|\tilde X\|_1>\tau_n))}}\sqrt{ \frac{\log\frac{1}{\eps_1}}{2^{p+1} }}.
\end{align*}
Sending $\eps_2\to0$ concludes the proof.
\end{proof}

\begin{proof}[Proof of Proposition~\ref{prop:sharpupperbound}]
Consider $d=m=2$, $A = \begin{pmatrix} a_1 & 0 \\ 0 & a_2\end{pmatrix}$, and $Z$ follows Pareto distribution and~i.i.d.~across dimensions. For more general cases of $m\geq d\geq 2$, we let 
\[
A = \left(\begin{matrix}
a_1 & \ldots & 0 \\
0 & \ldots & 0 \\
\vdots & \ddots & \vdots\\
0 & \ldots & a_d
\end{matrix} \,\Bigg | \begin{matrix}
1 & \ldots & 1 \\
0 & \ldots & 0 \\
\vdots & \ddots & \vdots\\
0 & \ldots & 0
\end{matrix}
\right),
\]
and follow similar arguments given here.

\textbf{Case 1: $\alpha>1$.} We first strengthen the statement~\eqref{eq:cartesionproborder} in the proof of Proposition~\ref{thm:bias} for $\alpha>1$. Following the notation outlined in the proof of Proposition~\ref{thm:bias}, we note that
\begin{align*}
   &  tP\left(\frac{Z}{b(t)}\in C_2\right)\\
   = & t\int_{z_1\leq\eps} P\left(Z_2 >b(t)\frac{1-a_1z_1}{a_2}\right) ( b(t)\alpha)(1 + b(t) z_1)^{-(\alpha+1)} \dd z_1\\
= & t\int_{z_1\leq\eps} \left(1+b(t)\frac{1-a_1z_1}{a_2}\right)^{-\alpha} ( b(t)\alpha)(1 + b(t) z_1)^{-(\alpha+1)} \dd z_1\\
= & a_2^\alpha \int_{z_1\leq\eps}\left(1 + b(t)^{-1} - a_1z_1\right)^{-\alpha} ( b(t)\alpha)(1 + b(t) z_1)^{-(\alpha+1)} \dd z_1.
\end{align*}
Then, recall that $\eps = \Theta(\log(t)^{-1})$, by Taylor expansion, 
\begin{align*}
    &\int_{z_1\leq\eps}\left(1 + b(t)^{-1} - a_1z_1\right)^{-\alpha} ( b(t)\alpha)(1 + b(t) z_1)^{-(\alpha+1)} \dd z_1\\
     = &  \int_{z_1\leq\eps}\left(1 -\alpha b(t)^{-1} +\alpha a_1z_1+\textit{``higher-order-terms''}\right)( b(t)\alpha)(1 + b(t) z_1)^{-(\alpha+1)} \dd z_1.
     \end{align*}
We note that
\[\left(1-\alpha b(t)^{-1}\right)\int_{z_1\leq\eps} ( b(t)\alpha)(1 + b(t) z_1)^{-(\alpha+1)} \dd z_1 = \left(1-\alpha b(t)^{-1}\right)\left(1-\tilde \Theta(b(t)^{-\alpha})\right).
\]
From the analysis in the proof of Proposition~\ref{thm:bias}, we know that 
\[
\int_{z_1\leq\eps}z_1b(t)\alpha(1+b(t)z_1)^{-(\alpha+1)} \dd z_1 = \Theta(b(t)^{-1}).
\]
Hence,
\begin{align*}
   &  tP\left(\frac{Z}{b(t)}\in C_2\right)\\
  = & a_2^\alpha \left(1 +\alpha a_1 \int_{z_1\leq\eps}z_1b(t)\alpha(1+b(t)z_1)^{-(\alpha+1)} \dd z_1 -\alpha b(t)^{-1}+ \textit{``higher-order-terms''} \right).
\end{align*}
Similarly, we can also obtain that
\begin{align*}
   &  tP\left(\frac{Z}{b(t)}\in C_1\right)\\
  = & a_1^\alpha \left(1 +\alpha a_2 \int_{z_2\leq\eps}z_1b(t)\alpha(1+b(t)z_2)^{-(\alpha+1)} \dd z_2 -\alpha b(t)^{-1} + \textit{``higher-order-terms''} \right).
\end{align*}
Thus, combined with the analysis in the proof of Proposition~\ref{thm:bias}, we know that there exists $a_1,a_2$, such that $A\in\mathcal{A}$, and
\[
P \left(\frac{X_1}{\|X\|_1} \in \left(0,\eps\right) \Bigg| \|X\|_1>b(t)\right)   =  \frac{a_2^\alpha}{a_1^\alpha+a_2^\alpha}\left(1+\Theta(b(t)^{-1}\right).
\]
Trivially, we require that $\tau_n\to\infty$ to achieve an error that converges to zero. Suppose that there exists a subsequence $\tau_{n_j},j=1,\ldots,\infty$ such that $\tau_{n_j}=O(n_j^{\frac{1}{2+\alpha}})$. By Glivenko–Cantelli theorem~\cite[Theorem~19.1]{ref:vaart1998asymptotic},
\[
\frac{\int_{\omega\in\aleph_+^1} 1_{\{\omega_1\in\left(0,\eps\right)\}}\mathrm{d} \mathbb{P}_{n_j}^{\tau_{n_j}} }{P \left(\frac{X_1}{\|X\|_1} \in \left(0,\eps\right) \Bigg| \|X\|_1>\tau_{n_j}\right)} \to 1\quad\text{in probability}\quad\text{as } j\to\infty.
\]
Hence, with an asymptotic probability of $1$, there will be a mass at least of order $\Theta(\tau_{n_j}^{-1})=\Theta(n_j^{-\frac{1}{2+\alpha}})$ that is misplaced with a distance bounded from zero (since $P(\frac{Z}{b(t)}\in\bar{C}) = \tilde O(t^{-2})$ from the analysis in the proof of Proposition~\ref{thm:bias}).  By the definition of $\mathbb{W}_p^p(\mathbb{P}_n^{\tau_n},K_A) $, and the condition stated in Proposition~\ref{prop:sharpupperbound}, we have that $\beta_{n_j}^p\geq\Theta(n_j^{-\frac{1}{2+\alpha}})$.

\textbf{Case 2: $\alpha\leq1$.} Next, we consider $\alpha\leq1$, and let $a_1=a_2=1$. Suppose that there exists a subsequence $\tau_{n_j},j=1,\ldots,\infty$ such that $\tau_{n_j}=O(n_j^{\frac{1}{3\alpha}})$. By similar analysis in the proof of Proposition~\ref{thm:bias}, we can compute
\[
P \left(\frac{X_1}{\|X\|_1} \in \left(\frac{1}{4},\frac{3}{4}\right) \Bigg| \|X\|_1>\tau_{n_j}\right)   =  \Theta(\tau_{n_j}^{-\alpha}).
\]Let $\phi(\omega) = (\omega_1-\frac{1}{2})^2$ be defined on $\aleph_+^1$. The $c$-transform of $\phi$ is given by (see~\cite[Definition 5.7]{ref:villani2008optimal})
\[
\phi^c(\omega') = \sup_{\omega\in\aleph_+^1} \left(\phi(\omega) - \|\omega-\omega'\|_1^p\right) = \sup_{\omega\in\aleph_+^1} \left((\omega_1-\frac{1}{2})^2 - \|\omega-\omega'\|_1^p\right).
\]
In particular, $\phi((0,1)) = \phi((1,0)) = \frac{1}{4}$, $\sup_{\omega'} \phi^c(\omega')\leq\frac{1}{4}$, and 
\[
\sup_{\omega'\in\aleph_+^1: \omega'_1 \in \left(\frac{1}{4},\frac{3}{4}\right)} \phi^c(\omega') <\frac{1}{4}.
\]
Hence, by Kantorovich duality~\cite[Theorem 5.10~(i)]{ref:villani2008optimal},
\[
\mathbb{W}_p^p(\mathbb{P}_n^{\tau_{n_j}},K_A) \geq    \int_{\aleph_+^1} \phi \mathrm{d}K_A- \int_{\aleph_+^1} \phi^c \mathrm{d} \mathbb{P}_{n_j}^{\tau_{n_j}}.
\]
Note that $K_A$ is a two-point distribution supported in the set $\{(0,1),(1,0)\}$, therefore \[\int_{\aleph_+^1} \phi \mathrm{d}K_A = \max_{\omega\in \aleph_+^1} \phi(\omega) = \frac{1}{4}.\] Thus we get from the condition stated in Proposition~\ref{prop:sharpupperbound}
\[
 \liminf_{j\to\infty} P\left(\left(\frac{1}{4}- \int_{\aleph_+^1} \phi^c \mathrm{d} \mathbb{P}_{n_j}^{\tau_{n_j}} \right) \leq C_\eps \beta_{n_j}^p
   \right)\geq1-\eps.
\]
By the analysis in the proof of Proposition~\ref{thm:variance} (those following equation~\eqref{eq:processconverge}), we know that 
\[
\frac{\int_{\omega\in\aleph_+^1} 1_{\{\omega_1\in\left(\frac{1}{4},\frac{3}{4}\right)\}}\mathrm{d} \mathbb{P}_{n_j}^{\tau_{n_j}} }{P \left(\frac{X_1}{\|X\|_1} \in \left(\frac{1}{4},\frac{3}{4}\right) \Bigg| \|X\|_1>\tau_{n_j}\right)} \to 1\quad\text{in probability}\quad\text{as } j\to\infty.
\]
Thus we conclude that $\beta_{n_j}^p\geq\Theta(n_j^{-\frac{1}{3}})$.

\textbf{Case 3: Assumption on the subsequence $\tau_{n_j}$ fails in Cases 1\&2.} Finally, consider all $\alpha>0$, and without loss of generality we can now assume 
\[\liminf_{n\to\infty}\frac{\tau_n}{n^{\frac{1}{\min\{2+\alpha,3\alpha\}}}}=\infty.\] 
Consider $a_1=a_2=1$, then by symmetry
\[
P \left(\frac{X_1}{\|X\|_1} \in \left(0,\frac{1}{2}\right) \Bigg| \|X\|_1>\tau_{n}\right)   =  \frac{1}{2}.
\]
Thus, conditioning on $n_{\tau_n}$, the misplaced mass from $(0,\frac{1}{2})$ to $(\frac{1}{2},0)$ (or vice versa) has an order $O_p(\frac{1}{\sqrt{n_{\tau_n}}})$. Hence $E_{n_{\tau_n}}[\mathbb{W}_p^p(\mathbb{P}_n^{\tau_n},K_A)]  \geq\Theta(\frac{1}{\sqrt{n_{\tau_n}}})$. Note that by bounded difference inequality (McDiarmid's inequality)~\cite{ref:mcdiarmid1989bounded}, for any $n\ge1$ and any $t>0$,
\[
P_{n_\tau}(\mathbb{W}_p^p(\mathbb{P}_n^{\tau},K_A) - E_{n_\tau}[\mathbb{W}_p^p(\mathbb{P}_n^{\tau},K_A)] \leq -t)  \leq \exp\left(-2^{p+1}n_\tau t^2\right).
\]
 Thus, with conditional probability at least $1-\eps_1$, we have
\[
\mathbb{W}_p^p(\mathbb{P}_n^{\tau},K_A) \geq E_{n_\tau}[\mathbb{W}_p^p(\mathbb{P}_n^{\tau},K_A)]  - \sqrt{ \frac{\log\frac{1}{\eps_1}}{2^{p+1}n_\tau}}\geq\Theta(\frac{1}{\sqrt{n_{\tau}}}),
\]
where the last inequality above holds if we choose $\eps_1$ sufficiently close to $1$. Combined with the fact that \[
\frac{n_{\tau_n}}{nP(\|X\|_1>\tau_{n})}\to 1\quad\text{in probability as }n\to\infty,
\]
that $\frac{1}{\sqrt{nP(\|X\|_1>\tau_n)}}\simeq \frac{1}{\sqrt{n\tau_n^{-\alpha}}}$, and the condition stated in Proposition~\ref{prop:sharpupperbound}, we conclude that $\liminf_{n\to\infty} n^{\frac{1}{2+\max\{1,\alpha\}}}\beta_n^p =\infty.$

\end{proof}

\subsection{Proof of Results in Section~\ref{sec:lowerbound}}\label{ap:lowerboundproof}
\begin{proof}[Proof of Lemma~\ref{lem:changeofmeasure}]
Let $X^{(1)},\ldots,X^{(n)}$ be i.i.d.~samples with the common marginal law $\mathcal{L}(X)$ and $\check X^{(1)},\ldots,\check X^{(n)}$ be i.i.d.~samples with the common marginal law $\mathcal{L}(\check X)$. 
For an arbitrary estimator $\hat{K}$, by a change of measure
\begin{align*}
 &\quad P^2(\mathbb{W}_p^p(\hat{K}(\check X^{(1)},\ldots,\check X^{(n)}),K_{\check A})\leq\beta_n^p)\\
&\leq P(\mathbb{W}_p^p(\hat{K}(X^{(1)},\ldots,X^{(n)}),K_{\check A})\leq\beta_n^p)\cdot\prod_{i=1}^n E\left[ \left(\frac{\check f(X^{(i)})}{f(X^{(i)})}\right)^2\right]  \\
& \leq C^{'} P(\mathbb{W}^p_p(\hat{K}(X^{(1)},\ldots,X^{(n)}),K_{\check A})\leq\beta_n^p),
\end{align*}
for some constant $C^{'}>0$, where we have used the Cauchy–Schwarz inequality in the first inequality and the condition $E\left[\left(\frac{\check f(X)}{f(X)}\right)^2\right]= 1+O(n^{-1})$ in the last inequality. Suppose that the estimator $\hat{K}$ satisfies both
\[
\liminf_{n\to\infty}P(\mathbb{W}_p^p(\hat{K}(\check X^{(1)},\ldots,\check X^{(n)}),K_{\check A})\leq\beta_n^p) =  1 ,
\]
and
\[
\liminf_{n\to\infty}P(\mathbb{W}_p^p(\hat{K}(X^{(1)},\ldots,X^{(n)}),K_A)\leq\beta_n^p) =  1 .
\]
Then, with a positive probability, it holds simultaneously that
\[
\mathbb{W}^p_p(\hat{K}(X^{(1)},\ldots,X^{(n)}),K_{\check A})\leq\beta_n^p\quad\text{and}\quad\mathbb{W}_p^p(\hat{K}(X^{(1)},\ldots,X^{(n)}),K_A)\leq\beta_n^p. 
\]
Then, by the triangle inequality, it must be that
\[
\mathbb{W}_p(K_A,K_{\check A})\leq 2\beta_n.
\]
\end{proof}
\begin{proof}[Proof of Theorem~\ref{thm:lowerbound}]
We follow the likelihood ratios approach in~\cite{ref:hall1984best}. Consider two models
\[
X^{(j)} = A^{(j)} Z^{(j)},\quad j =1,2.
\]
Note that we have abused notations slightly, in that $X^{(1)}$ and $X^{(2)}$ are not~i.i.d.~samples, but instead follow different distributions. Similarly for $Z^{(1)}$ and $Z^{(2)}$. 

\textbf{Step 1: Choice of $A^{(1)}$ and $A^{(2)}$.} We require that 
\[
\frac{A^{(1)}_{\cdot i}}{\|A^{(1)}_{\cdot i}\|_1} = \frac{A^{(2)}_{\cdot i}}{\|A^{(2)}_{\cdot i}\|_1}  = h_i,\quad i=1,\ldots,m,
\]
and that
\[
\min_{i\neq j} \|h_i - h_j\|_1 = \delta(m),
\]
where we denote by $\delta(m)$ the packing radius of $m$ points in $\aleph_+^1\subset\mathbb{R}^d$. 

Denote by $g^{(j)}$ the density of $Z^{(j)}$ in $\mathbb{R}^m_+$. For simplicity, let 
\[
a^{(j)}_i = \|A^{(j)}_{\cdot i}\|_1 \quad \forall i = 1,\ldots,m,\quad j = 1,2.
\]
Thus, one can write
\[
A^{(2)}  = A^{(1)} \textrm{diag}\left(\frac{a^{(2)}_1}{a^{(1)}_1},\ldots,\frac{a^{(2)}_m}{a^{(1)}_m}\right).
\]
Define 
\[
\tilde Z^{(2)} = \textrm{diag}\left(\frac{a^{(2)}_1}{a^{(1)}_1},\ldots,\frac{a^{(2)}_m}{a^{(1)}_m}\right) Z^{(2)},
\]
so that $X^{(2)} = A^{(2)}Z^{(2)} = A^{(1)} \tilde Z^{(2)}$,
and the density of $\tilde Z^{(2)}$ at $z$ in $\mathbb{R}^m_+$ is given by the generalized version of change of variables~\cite[Theorem~3.11]{ref:evans2015measure}
\begin{align*}
& \left(J\textrm{diag}\left(\frac{a^{(2)}_1}{a^{(1)}_1},\ldots,\frac{a^{(2)}_m}{a^{(1)}_m}\right)\right)^{-1} g^{(2)} \left(\frac{a^{(1)}_1}{a^{(2)}_1}z_1,\ldots,\frac{a^{(1)}_m}{a^{(2)}_m}z_m\right) \\
 =  & \left(\prod_{i=1}^m \frac{a_i^{(1)}}{a_i^{(2)}} \right) g^{(2)} \left(\frac{a^{(1)}_1}{a^{(2)}_1}z_1,\ldots,\frac{a^{(1)}_m}{a^{(2)}_m}z_m\right).
\end{align*}
Also by the generalized version of change of variables~\cite[Theorem~3.11]{ref:evans2015measure}, the density of $X^{(1)}$ in $\mathbb{R}^d_+$ can be written as
\[
\tilde g^{(1)}(x) = (JA^{(1)})^{-1}\int_{z\in\mathbb{R}^m_+:~ A^{(1)}z = x} g^{(1)}(z)\dd\mathcal{H}^{m-d},
\]
and the density of $X^{(2)}$ in $\mathbb{R}^d_+$ can be written as
\[
\tilde g^{(2)}(x) = (JA^{(1)})^{-1}\int_{z\in\mathbb{R}^m_+:~ A^{(1)}z = x} \left(\prod_{i=1}^m \frac{a_i^{(1)}}{a_i^{(2)}} \right) g^{(2)} \left(\frac{a^{(1)}_1}{a^{(2)}_1}z_1,\ldots,\frac{a^{(1)}_m}{a^{(2)}_m}z_m\right) \dd\mathcal{H}^{m-d},
\]
where $\mathcal{H}^{m-d}$ is the Hausdorff measure of dimension $m-d$ in $\mathbb{R}^m$. Thus, the likelihood ratio can be written as 
\[
\frac{\tilde g^{(2)}(x)}{\tilde g^{(1)}(x)} = \frac{\int_{z\in\mathbb{R}^m_+:~ A^{(1)}z = x} \left(\prod_{i=1}^m \frac{a_i^{(1)}}{a_i^{(2)}} \right) g^{(2)} \left(\frac{a^{(1)}_1}{a^{(2)}_1}z_1,\ldots,\frac{a^{(1)}_m}{a^{(2)}_m}z_m\right)\dd\mathcal{H}^{m-d}}{\int_{z\in\mathbb{R}^m_+:~ A^{(1)}z = x} g^{(1)}(z) \dd\mathcal{H}^{m-d}}.
\]
We further choose
\[
a_i^{(1)} = 1,\quad\forall i=1,\ldots,m,
\]
and
\[
a^{(2)}_1= 1 + B\delta(m) n^{-s}, a^{(2)}_2 = 1 - B\delta(m) n^{-s}, a_i^{(2)} = 1\quad\forall i\geq3. 
\]
where $B$ is some constant.
Then by the expression given by~\eqref{eq:limitcondmeasure},
\begin{equation}\label{eq:sephypythesis}
\mathbb{W}_p^p(K_A^{(1)},K_A^{(2)}) \gtrsim n^{-s}.
\end{equation}

\textbf{Step 2: Choice of $g^{(1)}$ and $g^{(2)}$.} 
We further choose $g^{(1)}(z) = \alpha^m \prod_{i=1}^m (1+z_i)^{-(\alpha+1)}~\forall z_i\geq0$, corresponding to the product density of Paretos. We choose $g^{(2)}$ such that
\[
g^{(2)}(z) = \begin{cases}
 c \alpha^m \prod_{i=1}^m (1+z_i)^{-(\alpha+1)} & \textrm{if }  \|z\|_{1}> \tau(1 + C n^{-s}), \\
\frac{a_1^{(2)}a_2^{(2)}}{a_1^{(1)}a_2^{(1)}}\alpha^m \left(1+\frac{a^{(2)}_1}{a^{(1)}_1}z_1\right)^{-(\alpha+1)}\left(1+\frac{a^{(2)}_2}{a^{(1)}_2}z_2\right)^{-(\alpha+1)} \prod_{i=3}^m (1+z_i)^{-(\alpha+1)} & \textrm{otherwise.}
\end{cases}
\]
Note that $\|x\|_1 = \|z\|_{1}$ if $A^{(1)}z= x$. The parameters $\tau,C$ above are chosen such that 
\[
P(\|X^{(1)}\|_1\geq \tau) = P(\|Z^{(1)}\|_1\geq\tau)= \Theta(n^{-(1-2s)}),
\]
and
\[
\tau(1 + C n^{-s}) \leq \zeta n^{\frac{1-2s}{\alpha}}, \quad\frac{1}{1-B\delta(m)n^{-s}} \leq 1+ Cn^{-s},
\]
where the last condition is equivalent with $C> B\delta(m)$ for $n$ sufficiently large. The constant $c$ is determined by the requirement that $g^{(2)}$ is a probability density, i.e.,
\begin{align*}
&\int_{z\in\mathbb{R}^m_+:\|z\|_1>\tau(1+Cn^{-s})} c\alpha^m\prod_{i=1}^m (1+z_i)^{-(\alpha+1)}\dd z\\
& = \int_{z\in\mathbb{R}^m_+:\|z\|_1>\tau(1+Cn^{-s})} \frac{a_1^{(2)}a_2^{(2)}}{a_1^{(1)}a_2^{(1)}}\alpha^m \left(1+\frac{a^{(2)}_1}{a^{(1)}_1}z_1\right)^{-(\alpha+1)}\left(1+\frac{a^{(2)}_2}{a^{(1)}_2}z_2\right)^{-(\alpha+1)} \prod_{i=3}^m (1+z_i)^{-(\alpha+1)} \dd z.
\end{align*}
We claim that $c^2 \leq 1+ O(n^{-2s})$. To see this, first note that
\[
\frac{a_1^{(2)}a_2^{(2)}}{a_1^{(1)}a_2^{(1)}} = 1 - B^2\delta(m)^2n^{-2s}\leq1.
\]
Hence, 
\begin{align*}
  &\quad \int_{z\in\mathbb{R}^m_+:\|z\|_{1}>\tau(1+Cn^{-s})} c\alpha^m\prod_{i=1}^m (1+z_i)^{-(\alpha+1)}\dd z\\
& \leq \int_{z\in\mathbb{R}^m_+:\|z\|_{1}>\tau(1+Cn^{-s})}\alpha^m \left(1+\frac{a^{(2)}_1}{a^{(1)}_1}z_1\right)^{-(\alpha+1)}\left(1+\frac{a^{(2)}_2}{a^{(1)}_2}z_2\right)^{-(\alpha+1)} \prod_{i=3}^m (1+z_i)^{-(\alpha+1)} \dd z\\
&= \int_{z\in\mathbb{R}^m_+:\|z\|_{1}>\tau(1+Cn^{-s})}\alpha^m \left(1+z_1 + B\delta(m) n^{-s}z_1\right)^{-(\alpha+1)}\\
&\qquad\cdot\left(1+z_2- B\delta(m) n^{-s}z_2\right)^{-(\alpha+1)} \prod_{i=3}^m (1+z_i)^{-(\alpha+1)} \dd z\\
&= \int_{z\in\mathbb{R}^m_+:\|z\|_{1}>\tau(1+Cn^{-s})}\alpha^m \left(1 + \frac{ B\delta(m) n^{-s}z_1}{1+z_1}\right)^{-(\alpha+1)}\left(1- \frac{B\delta(m) n^{-s}z_2}{1+z_2}\right)^{-(\alpha+1)}\\
&\qquad\cdot\prod_{i=1}^m (1+z_i)^{-(\alpha+1)} \dd z.
\end{align*}
Then, we apply Taylor expansions to 
\[
\left(1 + \frac{ B\delta(m) n^{-s}z_1}{1+z_1}\right)^{-(\alpha+1)} = 1 -(\alpha+1)\frac{ B\delta(m) n^{-s}z_1}{1+z_1} + O(n^{-2s}),
\]
and
\[\quad\left(1- \frac{B\delta(m) n^{-s}z_2}{1+z_2}\right)^{-(\alpha+1)} = 1 +(\alpha+1)\frac{ B\delta(m) n^{-s}z_2}{1+z_2} + O(n^{-2s}). 
\]
We note that $z_1,z_2$ are symmetrical in the region of integration
\[
\left\{z\in\mathbb{R}^m_+:\|z\|_{1}>\tau(1+Cn^{-s})\right\}.
\]
Thus the first-order terms of $O(n^{-s})$ cancels, and the dominating order becomes $O(n^{-2s})$. This shows that $c \leq 1 + O(n^{-2s})$, which concludes our previous claim that $c^2\leq1+O(n^{-2s})$. To show that $g^{(2)}(z)$ is in the model family of $\tilde M_n$ (cf. equation~\eqref{eq:minimaxmodel}), consider any $\|z\|_1< \tau(1+Cn^{-s})$, we have
\begin{align*}
&\quad \left|\frac{g^{(2)}(z) - \alpha^m\prod_{i=1}^m (1+z_i)^{-(\alpha+1)}}{\alpha^m\prod_{i=1}^m (1+z_i)^{-(\alpha+1)}}\right| \\
& = \left|(1 - B^2\delta(m)^2n^{-2s}) \left(1 + \frac{ B\delta(m) n^{-s}z_1}{1+z_1}\right)^{-(\alpha+1)}\left(1- \frac{B\delta(m) n^{-s}z_2}{1+z_2}\right)^{-(\alpha+1)} - 1\right| \\
& \leq \xi n^{-s},
\end{align*}
for sufficiently small $B$ and sufficiently large $n$, where we used the Taylor approximation in the last step.

\textbf{Step 3: Second-moment of the likelihood ratio.} By abusing the notation slightly, we denote by $E_{X\sim \mathcal{L}( A^{(1)}Z^{(1)})} \left[ \left(\frac{\tilde g^{(2)}(X)}{\tilde g^{(1)}(X)}\right)^2\right]$ the second moment of the likelihood ratio $\frac{\tilde g^{(2)}(X)}{\tilde g^{(1)}(X)}$, under the same law of $X^{(1)}$. Note that
\begin{align*}
    & E_{X\sim \mathcal{L}( A^{(1)}Z^{(1)})} \left[ \left(\frac{\tilde g^{(2)}(X)}{\tilde g^{(1)}(X)}\right)^2\right]\\
    & =\int_{\mathbb{R}^{d}_+} \frac{\left(\tilde g^{(2)}(x)\right)^2}{\tilde g^{(1)}(x)} \dd x \\
    & = \int_{\mathbb{R}^d_+} \frac{(JA^{(1)})^{-1}\left(\int_{z\in\mathbb{R}^m_+:~ A^{(1)}z = x} \left(\prod_{i=1}^m \frac{a_i^{(1)}}{a_i^{(2)}} \right) g^{(2)} \left(\frac{a^{(1)}_1}{a^{(2)}_1}z_1,\ldots,\frac{a^{(1)}_m}{a^{(2)}_m}z_m\right) \dd \mathcal{H}^{m-d}\right)^2}{\int_{z\in\mathbb{R}^m_+:~ A^{(1)}z = x} g^{(1)}(z) \dd \mathcal{H}^{m-d}} \dd x\\
    & =  \int_{\mathbb{R}^d_+} \frac{(JA^{(1)})^{-1}\left(\int_{z\in\mathbb{R}^m_+:~ A^{(1)}z = x} \left(\prod_{i=1}^m \frac{a_i^{(1)}}{a_i^{(2)}} \right) \frac{g^{(2)} \left(\frac{a^{(1)}_1}{a^{(2)}_1}z_1,\ldots,\frac{a^{(1)}_m}{a^{(2)}_m}z_m\right)}{\sqrt{g^{(1)}(z)}}\sqrt{g^{(1)}(z)} \dd \mathcal{H}^{m-d}\right)^2}{\int_{z\in\mathbb{R}^m_+:~ A^{(1)}z = x} g^{(1)}(z) \dd \mathcal{H}^{m-d}}\dd x\\
    & \leq \int_{\mathbb{R}^d_+} (JA^{(1)})^{-1} \int_{z\in\mathbb{R}^m_+:~ A^{(1)}z = x} \left(\prod_{i=1}^m \frac{a_i^{(1)}}{a_i^{(2)}} \right)^2 \frac{\left(g^{(2)}\right)^2 \left(\frac{a^{(1)}_1}{a^{(2)}_1}z_1,\ldots,\frac{a^{(1)}_m}{a^{(2)}_m}z_m\right)}{g^{(1)}(z)}\dd\mathcal{H}^{m-d} \dd x.
\end{align*}
By the generalized version of change of variables formula~\cite[Theorem~3.11]{ref:evans2015measure}, the last display is equivalent to
\begin{align*}
     & =\int_{z\in\mathbb{R}^m_+}\left(\prod_{i=1}^m \frac{a_i^{(1)}}{a_i^{(2)}} \right)^2 \frac{\left(g^{(2)}\right)^2 \left(\frac{a^{(1)}_1}{a^{(2)}_1}z_1,\ldots,\frac{a^{(1)}_m}{a^{(2)}_m}z_m\right)}{g^{(1)}(z)} \dd z\\
    & = \int_{z\in\mathbb{R}^m_+: \|z\|_{1}\leq \tau}\left(\prod_{i=1}^m \frac{a_i^{(1)}}{a_i^{(2)}} \right)^2 \frac{\left(g^{(2)}\right)^2 \left(\frac{a^{(1)}_1}{a^{(2)}_1}z_1,\ldots,\frac{a^{(1)}_m}{a^{(2)}_m}z_m\right)}{g^{(1)}(z)} \dd z \\
    & \quad  + \int_{z\in\mathbb{R}^m_+:\|z\|_{1}> \tau}\left(\prod_{i=1}^m \frac{a_i^{(1)}}{a_i^{(2)}} \right)^2 \frac{\left(g^{(2)}\right)^2 \left(\frac{a^{(1)}_1}{a^{(2)}_1}z_1,\ldots,\frac{a^{(1)}_m}{a^{(2)}_m}z_m\right)}{g^{(1)}(z)} \dd z.
\end{align*}
Note that if  $\|z\|_{1}\leq\tau$, then
\[
\left\|\left(\frac{a^{(1)}_1}{a^{(2)}_1}z_1,\ldots,\frac{a^{(1)}_m}{a^{(2)}_m}z_m\right)\right\|_1\leq \tau(1 + C n^{-s}),
\]
and hence
\begin{align*}
     &\int_{z\in\mathbb{R}^m_+: \|z\|_{1}\leq \tau}\left(\prod_{i=1}^m \frac{a_i^{(1)}}{a_i^{(2)}} \right)^2 \frac{\left(g^{(2)}\right)^2 \left(\frac{a^{(1)}_1}{a^{(2)}_1}z_1,\ldots,\frac{a^{(1)}_m}{a^{(2)}_m}z_m\right)}{g^{(1)}(z)} \dd z\\
     & = \int_{z\in\mathbb{R}^m_+: \|z\|_{1}\leq \tau}\left(\prod_{i=1}^m \frac{a_i^{(1)}}{a_i^{(2)}} \right)^2 \frac{\left(g^{(2)}\right)^2 \left(\frac{a^{(1)}_1}{a^{(2)}_1}z_1,\ldots,\frac{a^{(1)}_m}{a^{(2)}_m}z_m\right)}{\left(g^{(1)}\right)^2(z)}g^{(1)}(z) \dd z\\
    & = \int_{z\in\mathbb{R}^m_+: \|z\|_{1}\leq \tau} g^{(1)}(z) \dd z\\
    & = P(\|Z^{(1)}\|_1\leq\tau).
\end{align*}
Moreover,
\begin{align*}
    & \qquad\int_{z\in\mathbb{R}^m_+:\|z\|_{1}> \tau}\left(\prod_{i=1}^m \frac{a_i^{(1)}}{a_i^{(2)}} \right)^2 \frac{\left(g^{(2)}\right)^2 \left(\frac{a^{(1)}_1}{a^{(2)}_1}z_1,\ldots,\frac{a^{(1)}_m}{a^{(2)}_m}z_m\right)}{g^{(1)}(z)}\dd z\\
    & = \int_{z\in\mathbb{R}^m_+:\|z\|_{1}> \tau}\left(\prod_{i=1}^m \frac{a_i^{(1)}}{a_i^{(2)}} \right)^2 \frac{\left(g^{(2)}\right)^2 \left(\frac{a^{(1)}_1}{a^{(2)}_1}z_1,\ldots,\frac{a^{(1)}_m}{a^{(2)}_m}z_m\right)}{(g^{(1)})^2(z)}g^{(1)}(z) \dd z\\
    & \leq \int_{z\in\mathbb{R}^m_+:\|z\|_{1}> \tau}\max\left\{1,c^2\left(\prod_{i=1}^m \frac{a_i^{(1)}}{a_i^{(2)}} \right)^2\left(\frac{1+\frac{a^{(1)}_1}{a^{(2)}_1}z_1}{1+z_1}\right)^{-2(\alpha+1)}\left(\frac{1+\frac{a^{(1)}_2}{a^{(2)}_2}z_2}{1+z_2}\right)^{-2(\alpha+1)}\right\} g^{(1)}(z) \dd z\\
    & \leq \max\biggl\{P(\|Z^{(1)}\|_{a^{(1)}}>\tau),\\
    & \qquad\int_{z\in\mathbb{R}^m_+:\|z\|_{1}> \tau}c^2\left(\prod_{i=1}^m \frac{a_i^{(1)}}{a_i^{(2)}} \right)^2\left(\frac{1+\frac{a^{(1)}_1}{a^{(2)}_1}z_1}{1+z_1}\right)^{-2(\alpha+1)}\left(\frac{1+\frac{a^{(1)}_2}{a^{(2)}_2}z_2}{1+z_2}\right)^{-2(\alpha+1)} g^{(1)}(z) \dd z\biggr\}.
\end{align*}
Recall that we have shown before that $c^2 \leq 1 + O( n^{-2s})$. We can calculate
\begin{align*}
 \prod_{i=1}^m\frac{a_i^{(1)}}{a_i^{(2)}} =\frac{1}{(1+B\delta(m) n^{-s})(1-B\delta(m) n^{-s})} = 1 + O( n^{-2s}).
\end{align*}
Moreover, we have
\begin{align*}
    &\int_{z\in\mathbb{R}^m_+:\|z\|_{1}> \tau}\left(\frac{1+\frac{a^{(1)}_1}{a^{(2)}_1}z_1}{1+z_1}\right)^{-2(\alpha+1)}\left(\frac{1+\frac{a^{(1)}_2}{a^{(2)}_2}z_2}{1+z_2}\right)^{-2(\alpha+1)} g^{(1)}(z) \dd z\\
    & = \int_{z\in\mathbb{R}^m_+:\|z\|_{1}> \tau}\left(1 + \frac{(\frac{a^{(1)}_1}{a^{(2)}_1}-1)z_1}{1+z_1}\right)^{-2(\alpha+1)}\left(1 + \frac{(\frac{a^{(1)}_2}{a^{(2)}_2}-1)z_2}{1+z_2}\right)^{-2(\alpha+1)} g^{(1)}(z) \dd z\\
    &=\int_{z\in\mathbb{R}^m_+:\|z\|_{1}> \tau}\left(1 + \frac{\frac{-B\delta(m)n^{-s}}{1+B\delta(m)n^{-s}}z_1}{1+z_1}\right)^{-2(\alpha+1)}\left(1 + \frac{\frac{B\delta(m)n^{-s}}{1-B\delta(m)n^{-s}}z_2}{1+z_2}\right)^{-2(\alpha+1)} g^{(1)}(z) \dd z.
\end{align*}
As before, we perform Taylor expansion on the terms
\[
\left(1 + \frac{\frac{-B\delta(m)n^{-s}}{1+B\delta(m)n^{-s}}z_1}{1+z_1}\right)^{-2(\alpha+1)} = 1  + 2(\alpha+1)\frac{B\delta(m)n^{-s}}{1+B\delta(m)n^{-s}}\frac{z_1}{1+z_1} + O(n^{-2s}),
\]
and
\[\left(1 + \frac{\frac{B\delta(m)n^{-s}}{1-B\delta(m)n^{-s}}z_2}{1+z_2}\right)^{-2(\alpha+1)} = 1 - 2(\alpha+1)\frac{B\delta(m)n^{-s}}{1-B\delta(m)n^{-s}}\frac{z_2}{1+z_2} + O(n^{-2s}).
\]
We note that 
$z_1,z_2$ are symmetrical in the region of integration
\[
\left\{z\in\mathbb{R}^m_+:\|z\|_{1}>\tau\right\}.
\]
The dominating-order terms involve
\[
\frac{B\delta(m)n^{-s}}{1+B\delta(m)n^{-s}} - \frac{B\delta(m)n^{-s}}{1-B\delta(m)n^{-s}} = O (n^{-2s}).
\]
This concludes that
\[
\int_{z\in\mathbb{R}^m_+:\|z\|_{1}> \tau}\left(\prod_{i=1}^m \frac{a_i^{(1)}}{a_i^{(2)}} \right)^2 \frac{\left(g^{(2)}\right)^2 \left(\frac{a^{(1)}_1}{a^{(2)}_1}z_1,\ldots,\frac{a^{(1)}_m}{a^{(2)}_m}z_m\right)}{g^{(1)}(z)}dz = (1 + O( n^{-2s}))P(\|Z^{(1)}\|_1>\tau).
\]
Hence, in summary, we have found that
\[
E_{X\sim \mathcal{L}( A^{(1)}Z^{(1)})} \left[ \left(\frac{\tilde g^{(2)}(X)}{\tilde g^{(1)}(X)}\right)^2\right]= 1 + O(n^{-1}).
\]

\textbf{Step 4: Standard change-of-measure arguments.} The rest of the proof is routine (see~\cite{ref:hall1984best}). Recall the separation of the two hypotheses in equation~\eqref{eq:sephypythesis}. We apply Lemma~\ref{lem:changeofmeasure} to argue that
\[
\beta_n^p \geq\frac{1}{2^p} \mathbb{W}_p^p(K_A^{(1)},K_A^{(2)})\gtrsim n^{-s}.
\]
This completes the proof.
\end{proof}

\subsection{Proof of Results in Section~\ref{sec:potentialestimator}}\label{ap:potentialestimatorproof}
\subsubsection{Proof of Results in Section~\ref{sec:potential1dim}}
\begin{proof}[Proof of Theorem~\ref{thm:oracleonedim}]
\textbf{Step 1: Analyzing the Pareto case.} Let $\{\bar Z_1^{(1)},\ldots,\bar Z_1^{(n)}\}$ be an~i.i.d.~sample of size $n$ from the Pareto distribution. Recall that if $\bar Z_1$ is distributed as a Pareto with index $\alpha$, then the density of $\bar Z_1$ admits the form
 \[
 g_1(z_1) = \alpha(1+z_1)^{-(\alpha+1)}.
 \]
 Denote by $B(t)$ the quantile function of the Pareto distribution, which is given by
\[
P(\bar Z_1>B(t)) = \frac{1}{t},~\forall t>0.
\]
According to~\cite[Theorem~9.1]{ref:resnick2007heavy}, after some algebra, we have that for any $v = v(n)$ such that $1\leq v \leq n, v = o(n)$ and $v\to\infty$, 
\begin{equation}\label{eq:processconvergeonedim}
\sqrt{\frac{nP(\bar Z_1>B(n/v)y^{-\frac{1}{\alpha}})}{v}}\left(\frac{\sum_{i=1}^n1_{\{\bar Z_1^{(i)}>B(n/v)y^{-\frac{1}{\alpha}}\}}}{\sqrt{nP(\bar Z_1>B(n/v)y^{-\frac{1}{\alpha}})}} - \sqrt{nP(\bar Z_1>B(n/v)y^{-\frac{1}{\alpha}})} \right)\Rightarrow W
\end{equation}
in $D[0,\infty)$, where $W= (W(y):y\geq0)$ is a Brownian motion on $[0,\infty)$. Choosing $\tau_n= u\zeta n^{\frac{1-2s}{\alpha}}$, we 
 rescale 
\[
y^{-\frac{1}{\alpha}}= \frac{\tau_n}{\theta_1 B(\frac{n}{v})}\tilde y^{-\frac{1}{\alpha}},
\]
then we have for $\bar X_1=\theta_1 \bar Z_1$ (and $\bar X_1^{(i)}=\theta_1 \bar Z_1^{(i)}~\forall i$),
\begin{equation}\label{eq:processconvergerescale}
    \sqrt{\frac{nP(\bar X_1>\tau_n \tilde y^{-\frac{1}{\alpha}})}{v}}\left(\frac{\sum_{i=1}^n1_{\{\bar X_1^{(i)}>\tau_n\tilde y^{-\frac{1}{\alpha}}\}}}{\sqrt{nP(\bar X_1>\tau
_n\tilde y^{-\frac{1}{\alpha}})}} - \sqrt{nP(\bar X_1>\tau_n\tilde y^{-\frac{1}{\alpha}})} \right)\Rightarrow W.
\end{equation}
Also, by a similar analysis as in the proof of Proposition~\ref{thm:variance}, for $\tau_n= \Theta(n^{\frac{1-2s}{\alpha}})$, there exists a sequence of $v=v(n)$, such that
\[
 \frac{\tau_n}{B(\frac{n}{v})}\to 1.
\]
Hence the convergence~\eqref{eq:processconvergerescale} is uniform for bounded values of $\tilde y^{-\frac{1}{\alpha}}$ and $\theta_1$. Moreover, for bounded $\tilde y^{-\frac{1}{\alpha}}$ and $\theta_1$, we have
\[
P(\bar X_1>\tau_n \tilde y^{-\frac{1}{\alpha}}) = \Theta(n^{2s-1}),\quad \sqrt{\frac{nP(\bar X_1>\tau_n \tilde y^{-\frac{1}{\alpha}})}{v}} = \Theta(1).
\]
Hence 
\[
\sup_{\theta_1,\tilde y^{-\frac{1}{\alpha}}\in\textrm{ compact subsets of }\mathbb{R}_+}
\left|\frac{\sum_{i=1}^n1_{\{\bar X_1^{(i)}>\tau_n\tilde y^{-\frac{1}{\alpha}}\}}}{n} - P(\bar X_1>\tau_n\tilde y^{-\frac{1}{\alpha}}) \right| = O_p(n^{s-1}).
\]

\textbf{Step 2: Analyzing the general case.} Next, for $X_1 = \theta_1Z_1$ where $\mathcal{L}(Z_1) \in \cup_{k=1}^\infty \tilde M^1_k$. Note that we can always choose $\tilde y$, such that 

\[
P(X_1>\tau_n) = P(\bar X_1>\tau_n\tilde y^{-\frac{1}{\alpha}})
\]
 Since
 \[
 \frac{P(X_1>\tau_n)}{P(\bar X_1>\tau_n)} = \frac{P(Z_1 >\frac{\tau_n}{\theta_1})}{P( \bar Z_1 >\frac{\tau_n}{\theta_1})}
 \]
 is bounded from $0$ and $\infty$ asymptotically as $n\to\infty$, we have that $\tilde y^{-\frac{1}{\alpha}}$ is in a compact interval for large enough $n$. 
Hence for any $\eps\in(0,1)$, there exists $t_\eps$, such that 
\begin{equation}\label{eq:uniformconvergence}
\liminf_{n\to\infty}
\inf_{\mathcal{L}(X_1)\in M^1}P\left(n^{1-s} \left|\frac{\sum_{i=1}^n1_{\{X_1^{(i)}>\tau_n\}}}{n} - P(X_1>\tau_n) \right|\leq t_\eps\right) \geq 1-\eps,
\end{equation}
where $\{X_1^{(i)},i=1,\ldots,n\}$ is an i.i.d.~sample of size $n$ from any law $\mathcal{L}(X_1)\in M^1$. Since $\frac{\tau_n}{\theta_1} > \zeta n^{\frac{1-2s}{\alpha}}$, if $\mathcal{L}(X_1)\in M_k^1$ with $k\geq n$, then we know that 
\[
P(X_1>\tau_n)  = (1+O(n^{-s}))(1+\frac{\tau_n}{\theta_1})^{-\alpha},
\]
and that 
\[
\frac{\int_{\zeta  n^{\frac{1-2s}{\alpha}}}^\infty g_1(z_1)dz_1}{(1+\zeta n^{\frac{1-2s}{\alpha}})^{-\alpha}} = 1 + O(n^{-s}).
\]
Otherwise $\mathcal{L}(X_1)\in M_k^1$ with $k<n$, then
\[
P(X_1>\tau_n) = \frac{\int_{\zeta  n^{\frac{1-2s}{\alpha}}}^\infty g_1(z_1)dz_1}{(1+\zeta n^{\frac{1-2s}{\alpha}})^{-\alpha}}(1+\frac{\tau_n}{\theta_1})^{-\alpha}.
\]
Let $\hat{r} = \hat{r} \left(X_1^{(1)},\ldots,X_1^{(n)}\right)$ be an estimator such that
\[
\liminf_{n\to\infty}
\inf_{\mathcal{L}(X_1)\in M^1}P\left(n^{s} \left|\frac{\hat{r} (1+\zeta n^{\frac{1-2s}{\alpha}})^{-\alpha}}{\int_{\zeta  n^{\frac{1-2s}{\alpha}}}^\infty g_1(z_1)dz_1}-1\right|\leq C_\eps\right) \geq 1-\eps,
\]
and let an estimate of $\theta_1$ be given by 
\[
\frac{\sum_{i=1}^n1_{\{X_1^{(i)}>\tau_n\}}}{n} = \hat{r}(1+\frac{\tau_n}{\hat\theta_1})^{-\alpha}.
\]
Then
\begin{align*}
  & \log\frac{\sum_{i=1}^n1_{\{X_1^{(i)}>\tau_n\}}}{n}  - \log P(X_1>\tau_n) \\
   = & \log\hat{r} - \log \left(\frac{\int_{\zeta  n^{\frac{1-2s}{\alpha}}}^\infty g_1(z_1)dz_1}{(1+\zeta n^{\frac{1-2s}{\alpha}})^{-\alpha}}\right)-\alpha \log\frac{1+\frac{\tau_n}{\hat\theta_1}}{{1+\frac{\tau_n}{\theta_1}}}-\log(1+O(n^{-s})).
\end{align*}
Note that $P(X_1>\tau_n)=\Theta(n^{2s-1})$. Thus, with probability at least $1-2\eps$, we have
\[
\log(1+O(n^{-s})) =\log(1+O(n^{-s})) -\alpha \log\frac{1+\frac{\tau_n}{\hat{\theta}_1}}{{1+\frac{\tau_n}{\theta_1}}}-\log(1+O(n^{-s})).
\]
Thus
\[
\log(1+O(n^{-s})) =-\alpha \log\left(1+\left(\frac{\theta_1}{\hat{\theta}_1}-1\right)\frac{\frac{\tau_n}{\theta_1}}{1+\frac{\tau_n}{\theta_1}}\right).
\]
Since $\frac{\frac{\tau_n}{\theta_1}}{1+\frac{\tau_n}{\theta_1}}\to1$ and $\theta_1\in[l,u]$, we conclude that $\hat{\theta}_1 = \theta_1 + O(n^{-s})$. 
\end{proof}

\subsubsection{Proof of Results in Section~\ref{sec:potentialmultidim}}
We develop our proposed procedure in the following steps. First, suppose that we have the perfect knowledge of $\bar{A}$, then, upon observing $X$, we compute $\bar{A}^{-1}X$ as observations for $ \diag(\theta_1,\ldots, \theta_m) Z$. We have the following result.

\begin{lemma}\label{lem:oraclebarA}
Assume we have the perfect knowledge of $\bar A$.   Let $\hat{r} = \hat{r} \left(X^{(1)},\ldots,X^{(n)}\right)$ be an estimator such that for any $\eps\in(0,1)$, there exists $C_\eps$ such that
\begin{equation}\label{eq:oraclehatr}
    \liminf_{n\to\infty}
\inf_{\mathcal{L}(X)\in M}P\left(n^{s} \left|\frac{\hat{r} \int_{z\in\mathbb{R}^m_+:\|z\|_1>\zeta n^{\frac{1-2s}{\alpha}}}\alpha^m\prod_{i=1}^m (1+z_i)^{-(\alpha+1)}\mathrm{d}z}{\int_{z\in\mathbb{R}^m_+:\|z\|_1>\zeta n^{\frac{1-2s}{\alpha}}} g(z)\mathrm{d}z}-1\right|\leq C_\eps\right) \geq 1-\eps.
\end{equation}
  Choosing $\tau_n= u\zeta n^{\frac{1-2s}{\alpha}}$, and for any $i=1,\ldots,m$, let $\hat\theta_i$ solve the estimating equation
  \[
\frac{\sum_{j=1}^n1_{\{(\bar{A}^{-1}X)_i^{(j)}>\tau_n\}}}{n} = \hat{r}(1+\frac{\tau_n}{\hat\theta_i})^{-\alpha}.
\]
Denoting $\hat\theta = (\hat\theta_1,\ldots,\hat\theta_m)$, for any $\eps\in(0,1)$, there exists $\tilde C_\eps$, such that
    \[
    \liminf_{n\to\infty}
\inf_{\mathcal{L}(X)\in M}P\left(n^{s} \|\hat\theta-\theta\|_1\leq \tilde C_\eps\right) \geq 1-\eps.
    \]
\end{lemma}

Next, suppose that we have an estimate of $\bar{A}$, denoted by $\tilde A$. Using $\tilde A$ as a ``plug-in'', upon observing $X$, we consider $\tilde X = \tilde {A}^{-1}X$ as noisy observations for $ \diag(\theta_1,\ldots, \theta_m) Z$.

\begin{lemma}\label{lem:pluginbarA}
Suppose there exists an estimator $\tilde A$ for $\bar{A}$, such that for any $\eps\in(0,1)$, there exists $\underline{C}_\eps$ such that 
\[
    \liminf_{n\to\infty}
\inf_{\mathcal{L}(X)\in M}P\left(n^{s} \|\tilde A - \bar A\|_1\leq \underline C_\eps\right) \geq 1-\eps.
    \]
   Further, let $\hat{r}$ be the estimator satisfying the condition~\eqref{eq:oraclehatr}.  Choosing $\tau_n= u\zeta n^{\frac{1-2s}{\alpha}}$, and for any $i=1,\ldots,m$, let $\hat\theta_i$ solve the estimating equation
  \[
\frac{\sum_{j=1}^n1_{\{\tilde X_i^{(j)}>\tau_n\}}}{n} = \hat{r}(1+\frac{\tau_n}{\hat\theta_i})^{-\alpha}.
\]
Denoting $\hat\theta = (\hat\theta_1,\ldots,\hat\theta_m)$, for any $\eps\in(0,1)$, there exists $\tilde C_\eps$, such that
    \[
    \liminf_{n\to\infty}
\inf_{\mathcal{L}(X)\in M}P\left(n^{s} \|\hat\theta-\theta\|_1\leq \tilde C_\eps\right) \geq 1-\eps.
    \]
\end{lemma}

Finally, we have a construction of the required estimator $\tilde A$.
\begin{lemma}\label{lem:estbarA}
    Assume $\frac{1}{2+\max\{1,\alpha\}} < s <\min\{\frac{1}{2},\frac{1}{\alpha}\}$. Let $\tau_n = \left(\frac{\log n }{n}\right)^{-\frac{1}{\alpha}}$, and let $\tilde A$ be the matrix whose columns are the $m$ cluster centers identified with k-means using $\mathbb{P}_n^{\tau_n}$. Then, for any $\eps\in(0,1)$, there exists $\underline{C}_\eps$ such that after a column-permutation of $\tilde{A}$, it holds that
\[
    \liminf_{n\to\infty}
\inf_{\mathcal{L}(X)\in M}P\left(n^{s} \|\tilde A - \bar A\|_1\leq \underline C_\eps\right) \geq 1-\eps.
    \]
\end{lemma}

\begin{proof}[Proof of Lemma~\ref{lem:oraclebarA}]
Since we observe $\bar{A}^{-1}X=\diag(\theta_1,\ldots,\theta_m)Z$, we conduct the estimation procedure separately for $(\bar{A}^{-1}X)_i = \theta_i Z_i$ over the $m$ coordinates. To illustrate, suppose we consider $i=1$. Notice that if $g(z)$ belongs to $\tilde M_k$ (cf. equation~\eqref{eq:lawZ}), then 
\[
-\xi k^{-s} \alpha^m\prod_{i=1}^m (1+z_i)^{-(\alpha+1)}\leq g(z) - \alpha^m\prod_{i=1}^m (1+z_i)^{-(\alpha+1)}\leq \xi k^{-s} \alpha^m\prod_{i=1}^m (1+z_i)^{-(\alpha+1)}.
\]
Integrating out $(z_2,\ldots,z_m)$ to obtain the marginal density of $z_1$, denoted by $g_1(z_1)$, we have
\[
-\xi k^{-s} \alpha(1+z_1)^{-(\alpha+1)}\leq g_1(z_1) - \alpha(1+z_1)^{-(\alpha+1)}\leq \xi k^{-s} \alpha (1+z_1)^{-(\alpha+1)}.
\]
Moreover, $g_1(z_1)$ satisfies 
\[
 g_1(z_1) \propto (1+z_1)^{-(\alpha+1)} \textrm{ if } z_1> \zeta k^{\frac{1-2s}{\alpha}},
\]
since $z_1> \zeta k^{\frac{1-2s}{\alpha}}$ would imply that $\|z\|_1> \zeta k^{\frac{1-2s}{\alpha}}$. Hence, $g_1(z_1)$ belongs to $\tilde M^1_k$ (cf. equation~\eqref{eq:lawZ1dim}). Finally, conditioning on the event (which holds with probability at least $1-\eps$)
\[
n^{s} \left|\frac{\hat{r} \int_{z\in\mathbb{R}^m_+:\|z\|_1>\zeta\cdot n^{\frac{1-2s}{\alpha}}}\alpha^m\prod_{i=1}^m (1+z_i)^{-(\alpha+1)}\mathrm{d}z}{\int_{z\in\mathbb{R}^m_+:\|z\|_1>\zeta n^{\frac{1-2s}{\alpha}}} g(z)\mathrm{d}z}-1\right|\leq C_\eps.
\]
Then if $k<n$, we have
\[n^{s} \left|\frac{\hat{r} (1+\zeta n^{\frac{1-2s}{\alpha}})^{-\alpha}}{\int_{\zeta n^{\frac{1-2s}{\alpha}}}^\infty g_1(z_1)dz_1}-1\right|\leq C_\eps.
\]
Otherwise $k\geq n$, we have $\hat{r} = 1+O(n^{-s})$. Hence, by Theorem~\ref{thm:oracleonedim}, we can estimate $\theta_1$ with an error of $O(n^{-s})$ with probability at least $1-\eps$.
\end{proof}

\begin{proof}[Proof of Lemma~\ref{lem:pluginbarA}]
    Note that since $A$ is a square matrix, $JA = |\textrm{det}(A)|$, and by the definition of the model class~\eqref{eq:minimaxmodel}, we have $|\textrm{det}(A)|\geq \sigma$. Since $\|A_{\cdot i}\|_1\in[l,u]$, we have that  $|\textrm{det}(\bar A)|$ is bounded from zero. 
    Conditioning on the events $n^{s} \|\tilde A - \bar A\|_1\leq \underline C_\eps$, we can write
    \[
    \tilde A^{-1}X = (\bar{A} + O(n^{-s}))^{-1}X,\quad\textrm{as }n\to\infty.
    \]
    Hence 
    \[
    ( \tilde A^{-1}X)_i = (\bar{A}^{-1}X)_i (1+O(n^{-s})) \quad\textrm{as }n\to\infty,~\forall i.
    \]
    Consider the component $i=1$, there exists $c_1,c_2$, such that for all $n$ large enough
   \[
   \frac{\sum_{j=1}^n1_{\{(\bar{A}^{-1} X)_1^{(j)}>\tau_n(1+c_1n^{-s})\}}}{n}\leq \frac{\sum_{j=1}^n1_{\{( \tilde A^{-1}X)_1^{(j)}>\tau_n\}}}{n}\leq\frac{\sum_{j=1}^n1_{\{(\bar{A}^{-1} X)_1^{(j)}>\tau_n(1-c_2n^{-s})\}}}{n},
   \]
   and
   \[
   P\left((\bar{A}^{-1} X)_1>\tau_n(1+c_1n^{-s})\right)\leq P((\tilde{A}^{-1}X)_1>\tau_n)\leq P\left((\bar{A}^{-1} X)_1>\tau_n(1-c_2n^{-s})\right).
   \]
   From equation~\eqref{eq:uniformconvergence} and the proof of Lemma~\ref{lem:oraclebarA}, we obtain
   \begin{small}
   \[
   \liminf_{n\to\infty}
\inf_{\mathcal{L}(X)\in M}P\left(n^{1-s} \left| \frac{\sum_{j=1}^n1_{\{(\bar{A}^{-1} X)_1^{(j)}>\tau_n(1+c_1n^{-s})\}}}{n} - P\left((\bar{A}^{-1} X)_1>\tau_n(1+c_1n^{-s})\right) \right|\leq C_\eps\right) \geq 1-\eps,
   \]
   \end{small}
  where note that $\tau_n(1+c_1n^{-s}) = \Theta(\tau_n)$. Similarly, we have
  \begin{small}
  \[
   \liminf_{n\to\infty}
\inf_{\mathcal{L}(X)\in M}P\left(n^{1-s} \left| \frac{\sum_{j=1}^n1_{\{(\bar{A}^{-1} X)_1^{(j)}>\tau_n(1-c_2n^{-s})\}}}{n} - P\left((\bar{A}^{-1} X)_1>\tau_n(1-c_2n^{-s})\right) \right|\leq t_\eps)\right) \geq 1-\eps.
   \]
   \end{small}
   Then, note that if $\mathcal{L}(X)\in M_k$ and  $k<n$, then
   \begin{align}
        & P\left((\bar{A}^{-1} X)_1>\tau_n(1-c_2n^{-s})\right) -P\left((\bar{A}^{-1} X)_1>\tau_n(1+c_1n^{-s})\right)\notag\\
        & \propto \left( \left(1+ \frac{\tau_n(1-c_2n^{-s})}{\theta_1}\right)^{-\alpha} - \left(1+ \frac{\tau_n(1+c_1n^{-s})}{\theta_1}\right)^{-\alpha}\right)\notag\\
        & = \left(\frac{\left(1+ \frac{\tau_n(1-c_2n^{-s})}{\theta_1}\right)^{-\alpha}}{\left(1+ \frac{\tau_n(1+c_1n^{-s})}{\theta_1}\right)^{-\alpha}}  - 1\right)\left(1+ \frac{\tau_n(1+c_1n^{-s})}{\theta_1}\right)^{-\alpha}\notag\\
         & = O(n^{-s}) O(n^{2s-1}) = O(n^{s-1})\label{eq:genksmall}.
   \end{align}
   Otherwise $k\geq n$, then
   \begin{align}
        & P\left((\bar{A}^{-1} X)_1>\tau_n(1-c_2n^{-s})\right) -P\left((\bar{A}^{-1} X)_1>\tau_n(1+c_1n^{-s})\right)\notag\\
        & = \left(1+ \frac{\tau_n(1-c_2n^{-s})}{\theta_1}\right)^{-\alpha} (1+O(n^{-s})) - \left(1+ \frac{\tau_n(1+c_1n^{-s})}{\theta_1}\right)^{-\alpha}(1+O(n^{-s}))\notag\\
        & = O(n^{-s}) O(n^{2s-1}) = O(n^{s-1})\label{eq:genklarge}.
   \end{align}
   Combining these, we have for some $C'_\eps>0$, it holds that 
   \[
    \liminf_{n\to\infty}
\inf_{\mathcal{L}(X)\in M}P\left(n^{1-s} \left|\frac{\sum_{j=1}^n1_{\{( \tilde A^{-1}X)_1^{(j)}>\tau_n\}}}{n} - P((\tilde{A}^{-1}X)_1>\tau_n) \right|\leq C'_\eps\right) \geq 1-\eps.
    \]
    Finally, notice that from equations~\eqref{eq:genksmall} and~\eqref{eq:genklarge},  we can also conclude that
   \[
   P((\tilde{A}^{-1}X)_1>\tau_n) = (1+O(n^{-s}))P((\bar{A}^{-1}X)_1>\tau_n).
   \]
   Hence, the proofs of Theorem~\ref{thm:oracleonedim} and Lemma~\ref{lem:oraclebarA} apply, and we can estimate $\theta_1$ with an error of $O(n^{-s})$ with probability at least $1-\eps$.
\end{proof}

\begin{proof}[Proof of Lemma~\ref{lem:estbarA}]
Since $\frac{1}{2+\max\{1,\alpha\}}< s< \frac{1}{2}$, we know that by Proposition~\ref{thm:bias}, \[\sup_{\mathcal{L}(X)\in M}\mathbb{W}_1(\mu_A^\tau,K_A)\leq \tilde O(\tau^{-\min\{1,\alpha\}}).\] 
Since $|\textrm{det}(\bar{A})|$ is the volume of the parallelepiped defined by columns of $\bar{A}$, and $|\textrm{det}(\bar{A})|$ is (uniformly over model class~\eqref{eq:minimaxmodel}) lower bounded from zero, we know that the distance between $\bar{A}_{\cdot i}$ and $\bar{A}_{\cdot j}$ is (uniformly over model class~\eqref{eq:minimaxmodel}) lower bounded for arbitrary $i\neq j$. Consider a $\Theta(n^{-s})$ enlargement of the set of columns of $\bar{A}$, denoted by $\bar{A}^e$, i.e.,
\[
\bar{A}^e = \{\omega \in \aleph_+^1: \exists i, \textrm{ such that } \|\omega-\bar{A}_{\cdot i}\|_1\leq \Theta(n^{-s})\}.
\]
Set $\tau^{-\alpha} = \frac{\log n}{n}$, we find that if $\alpha\leq1$, then
\[\sup_{\mathcal{L}(X)\in M}\mathbb{W}_1(\mu_A^\tau,K_A) \leq \tilde O(n^{-1}),\]
and if $\alpha>1$, then
\[\sup_{\mathcal{L}(X)\in M}\mathbb{W}_1(\mu_A^\tau,K_A) \leq \tilde O(n^{-\frac{1}{\alpha}}).\]
Hence by definition (note that $s<\frac{1}{\alpha}$ by the assumption in the statement of Lemma~\ref{lem:estbarA})
\[
\sup_{\mathcal{L}(X)\in M}P\left(\frac{X}{\|X\|_1}\notin \bar{A}^e\Big|\|X\|_1>\tau\right)\leq  \tilde O\left(n^{-(\frac{1}{\max\{1,\alpha\}}-s)}\right).
\]
Thus, we have
\[
\sup_{\mathcal{L}(X)\in M} P_{n_\tau}\left(\exists 1\leq j\leq n_\tau: X^{(j,n_\tau)}\notin\bar{A}^e\right)\leq  1- \left(1-\tilde O\left(n^{-(\frac{1}{\max\{1,\alpha\}}-s)}\right)\right)^{n_\tau}.
\]
From similar analysis in the proof of Proposition~\ref{thm:variance}, we have that for any $\eps\in(0,1)$, there exists $ t_\eps$, such that
\[
\liminf_{n\to\infty} \inf_{\mathcal{L}(X)\in M} P\left(\frac{n_\tau}{n\left(\frac{\log n}{n}\right)}\leq t_\eps\right) \geq 1-\eps.
\]
Conditioning on the event $n_\tau = O(\log n) $, we have
\[
\sup_{\mathcal{L}(X)\in M} P_{n_\tau}\left(\exists 1\leq j\leq n_\tau: X^{(j,n_\tau)}\notin\bar{A}^e\right)\leq  1- \left(1-\tilde O\left(n^{-(\frac{1}{\max\{1,\alpha\}}-s)}\right)\right)^{\log n},
\]
where the right-hand side above goes to zero. Moreover, following the same analysis, we can also conclude that
\[
\inf_{\mathcal{L}(X)\in M} P_{n_\tau}\left(\exists 1\leq j\leq n_\tau: X^{(j,n_\tau)}\in\bar{A}_{\cdot i}^e\right)\to1,
\]
where $\bar{A}_{\cdot i}^e$ denotes the $\Theta(n^{-s})$ enlargement around the singleton $\{\bar{A}_{\cdot i}\}$ (hence $\bar{A}^e = \cup_{i=1}^m \bar{A}_{\cdot i}^e$). 
Hence, with asymptotic probability at least $1-2\eps$, we can apply k-means to obtain the $m$ cluster centers with an error of $O(n^{-s})$~\cite[Algorithm~14.1]{hastie2009elements}.
\end{proof}

\begin{proof}[Proof of Theorem~\ref{thm:bigthmpotentialestimator}]
Combining Lemmas~\ref{lem:oraclebarA},~\ref{lem:pluginbarA} and \ref{lem:estbarA}, and using the definition of $\mathbb{W}_p^p$, the result follows trivially.
\end{proof}

\subsection{Additional Simulation Results in Section~\ref{sec:numerics}}\label{app:additionalsim}
Here, we provide additional details in demonstrating that the tuning of $\kappa$ is robust and the choice of $\kappa=1$ is also reasonably optimal. We fix $\tilde\kappa=0.5$, and select $\kappa$ from $\{0.01,0.1,1,2\}$. We found that except for $\kappa=0.01$, the other choices of $\kappa$ display comparable convergence performance, see Figure~\ref{fig:kappatunevis} and Table~\ref{tab:ratekappa}, and the diagnostic plots strike a reasonable bias-variance tradeoff, see Figure~\ref{fig:kappatune}.

\begin{figure}[h]
    \centering
    \subfigure[$\kappa=0.01$]{
	 \includegraphics[width=0.45\columnwidth]{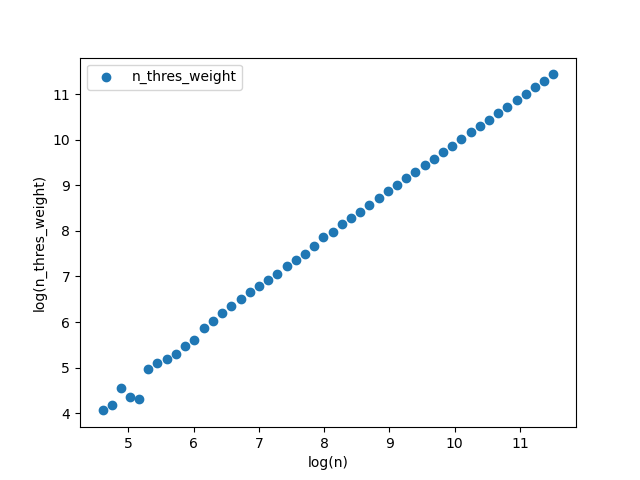}} \hspace{1mm}
	\subfigure[$\kappa=0.1$]{
	\includegraphics[width=0.45\columnwidth]{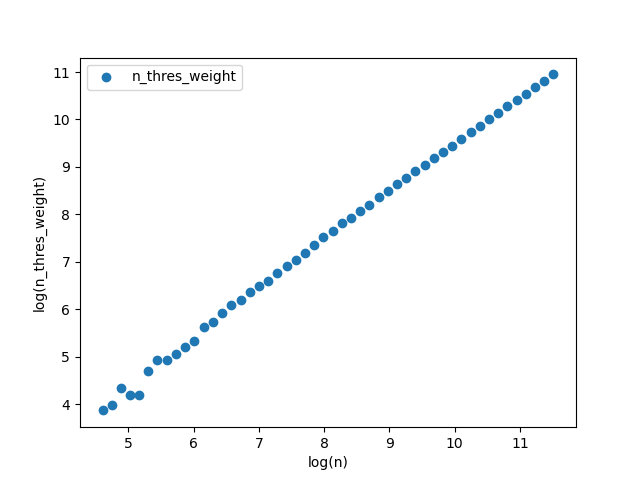}}
	    \hspace{1mm}
	\subfigure[$\kappa=1$]{
		\includegraphics[width=0.45\columnwidth]{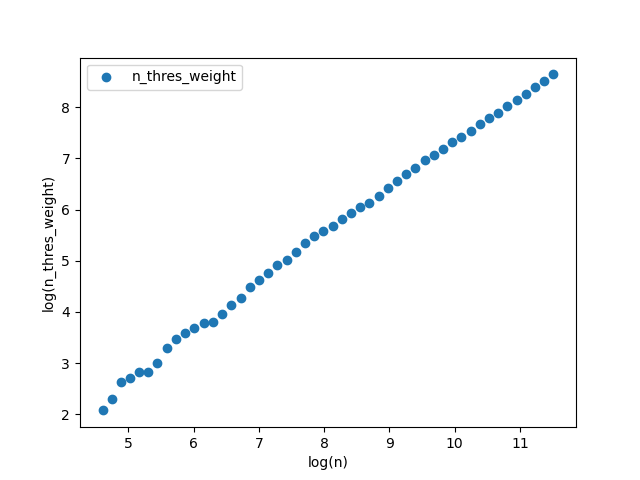}}
	\subfigure[$\kappa=2$]{
		\includegraphics[width=0.45\columnwidth]{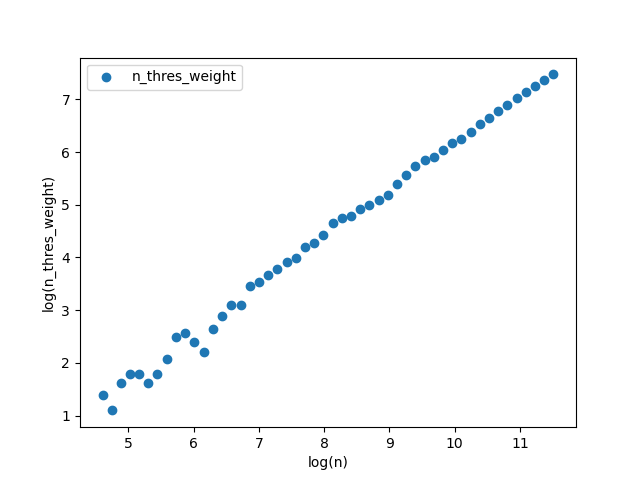}} 
	\caption{Logarithm of number of samples exceeding the threshold $\tau$ versus $\log(n)$, where $n$ is the total sample size, for each choice of $\kappa$.}
    \label{fig:kappatune}
\end{figure}

\begin{figure}[h]
    \centering
    \subfigure[$\kappa=0.01$]{
	 \includegraphics[width=0.45\columnwidth]{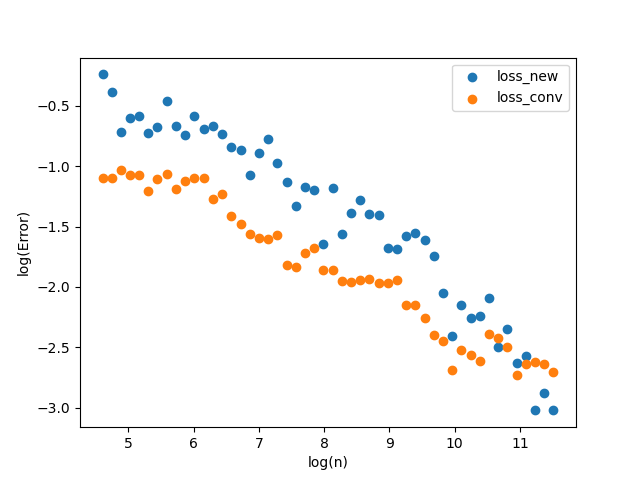}} \hspace{1mm}
	\subfigure[$\kappa=0.1$]{
	\includegraphics[width=0.45\columnwidth]{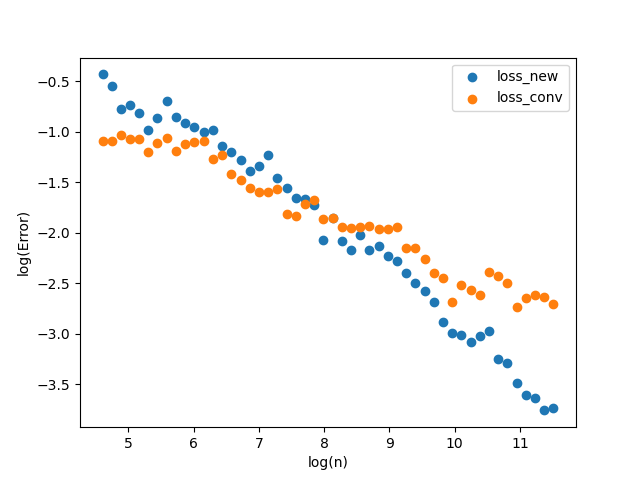}}
	    \hspace{1mm}
	\subfigure[$\kappa=1$]{
		\includegraphics[width=0.45\columnwidth]{vis_0.5_1.png}}
	\subfigure[$\kappa=2$]{
		\includegraphics[width=0.45\columnwidth]{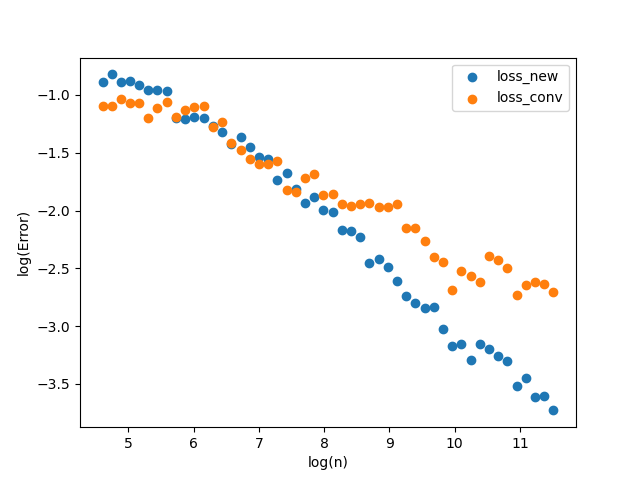}} 
	\caption{Convergence of methods versus $\log(n)$, where $n$ is the total sample size, for each choice of $\kappa$.}
    \label{fig:kappatunevis}
\end{figure}

\begin{table}[h!]
\centering
\begin{tabular}{||c|c|c|c|c||} 
\hline
$\kappa$ & 0.01 & 0.1 & 1 & 2\\
\hline
rate & -0.352 & -0.470 & -0.456 & -0.446\\
\hline
\end{tabular}
\caption{Convergence rate for difference choices of hyperparameter $\kappa$.}
\label{tab:ratekappa}
\end{table}

\newpage
\bibliographystyle{siam}
\bibliography{bibliography}
\end{document}